\documentclass[12pt,reqno,a4paper]{amsart}
\usepackage{blindtext}
\usepackage{fullpage}
\usepackage{mathtools}
\usepackage{longtable}
\usepackage{amsmath,amssymb,amsthm}
\usepackage{amscd}
\usepackage{bm}
\usepackage{hyperref}   
\usepackage{dsfont}
\usepackage{enumerate}
\usepackage{epsfig}
\usepackage{float, graphicx}
\usepackage{latexsym, amsxtra}
\usepackage{mathrsfs}
\usepackage{multicol}
\usepackage[normalem]{ulem}
\usepackage{psfrag}
\usepackage[parfill]{parskip}
\usepackage{stmaryrd}
\usepackage{tikz}
\usepackage[T1]{fontenc}
\usepackage{url}
\usepackage{verbatim}
\usepackage{indentfirst}
\usepackage{tikz-cd}
\usepackage{booktabs}
\usepackage{svg}

\usepackage{mathtools}

\usepackage{caption} 

\makeatletter
\def\thm@space@setup{%
	\thm@preskip=2ex \thm@postskip=2ex
}
\makeatother

\hypersetup{hidelinks}

\newtheorem{thm}{Theorem~}[section]
\newtheorem{lem}[thm]{Lemma~}

\newtheorem{prop}[thm]{Proposition~}

\newtheorem{cor}[thm]{Corollary~}

\theoremstyle{remark}
\newtheorem{rmk}[thm]{Remark~}
\newtheorem{ex}[thm]{Example~}

\theoremstyle{definition}  
\newtheorem{defn}[thm]{Definition~}

\newcommand{\CC}{\mathbb{C}}
\newcommand{\ZZ}{\mathbb{Z}}
\newcommand{\RR}{\mathbb{R}}
\newcommand{\LL}{\mathbb{L}}
\newcommand{\PP}{\mathbb{P}}

\newcommand{\QQ}{\mathbb{Q}}

\newcommand{\BB}{\mathbb{B}}
\newcommand{\KK}{\mathbb{K}}

\newcommand{\calA}{\mathcal{A}}

\newcommand{\calC}{\mathcal{C}}

\newcommand{\calM}{\mathcal{M}}

\newcommand{\calO}{\mathcal{O}}
\newcommand{\calP}{\mathcal{P}}

\newcommand{\calL}{\mathcal{L}}
\newcommand{\calU}{\mathcal{U}}

\newcommand\Aut{\mathrm{Aut}}

\newcommand{\GL}{\mathrm{GL}}
\newcommand\SL{\mathrm{SL}}
\newcommand\SU{\mathrm{SU}}

\newcommand\PU{\mathrm{PU}}
\newcommand{\GU}{\mathrm{GU}}
\newcommand{\U}{\mathrm{U}}

\newcommand\Res{\mathrm{Res}}

\newcommand\PSL{\mathrm{PSL}}

\newcommand{\Pic}{\mathrm{Pic}}
\newcommand{\diag}{\mathrm{diag}}

\newcommand{\Hom}{\mathrm{Hom}}

\newcommand{\Tr}{\mathrm{Tr}}
\newcommand{\Ad}{\mathrm{Ad}}
\newcommand{\Gal}{\mathrm{Gal}}
\newcommand{\sgn}{\mathrm{sgn}}

\newcommand{\fp}{\mathfrak{p}}

\newcommand{\bs}{\backslash}
\newcommand{\dbs}{\bs\hspace{-1mm} \bs}

\title{Commensurability among Deligne--Mostow Monodromy Groups}
\vspace{1.2cm}
\author{Chenglong Yu, Zhiwei Zheng}
\date{}

\newcommand{\Addresses}{{
		\bigskip
		\footnotesize
		
		C.~Yu, \textsc{Tsinghua University, Beijing, China}\par\nopagebreak
		\textit{Email address}: \texttt{yuchenglong@tsinghua.edu.cn}
		
		\medskip
		
		Z.~Zheng, \textsc{Tsinghua University, Beijing, China}\par\nopagebreak
		\textit{Email address}: \texttt{zhengzhiwei@mail.tsinghua.edu.cn}
}}

\begin{document}
\bibliographystyle{amsalpha}

\begin{abstract}
This paper gives the commensurability classification of Deligne--Mostow ball quotients and shows that the $104$ Deligne--Mostow lattices form $38$ commensurability classes. First, we find commensurability relations among Deligne--Mostow monodromy groups, which are not necessarily discrete. This generalizes previous work by Sauter and Deligne--Mostow in dimension two. In this part, we consider certain projective surfaces with two fibrations over the projective line, which induce two sets of Deligne--Mostow data. Correspondences between moduli spaces provide a geometric realization of commensurability relations. Secondly, we obtain commensurability invariants from conformal classes of Hermitian forms and toroidal boundary divisors. This completes the commensurability classification of Deligne--Mostow lattices and provides an alternative approach to the results of Kappes--M{\"o}ller and McMullen on non-arithmetic Deligne--Mostow lattices.
\end{abstract}
	
	\maketitle
 \setcounter{tocdepth}{1}
	\tableofcontents
	
\section{Introduction}
This paper investigates the commensurability relations among Deligne--Mostow monodromy groups in $\PU(1,n)$. Applying higher-dimensional cyclic covers inspired by \cite{yu2024calabi}, we provide a geometric approach to the results of Sauter \cite{sauter1990isomorphisms} and Deligne--Mostow \cite{deligne1993commensurabilities} on commensurability relations among Deligne--Mostow monodromy groups in $\PU(1,2)$. We also derive new relations for $n > 2$, see Theorem \ref{theorem: main by mentioning table 2}. Furthermore, using invariants related to Hermitian forms and toroidal compactifications, we can distinguish different commensurability classes and provide a complete classification of commensurability relations among all $104$ Deligne--Mostow lattices, which is presented in Table \ref{table: main table}. In total, there are $38$ commensurability classes of Deligne--Mostow lattices. Previously, it was known there are $10$ commensurability classes of non-arithmetic Deligne--Mostow lattices by the work of Sauter \cite{ sauter1990isomorphisms}, Deligne--Mostow \cite{deligne1993commensurabilities}, Kappes--M\"oller \cite{kappes2016lyapunov} and McMullen \cite{McMullen2017Gauss-Bonnet}.

In simple Lie groups, every lattice must be arithmetic, either by Margulis's superrigidity theorem \cite{margulis1991discrete} or by results of Corlette \cite{corlette1992archimedean} and Gromov--Schoen \cite{gromov1992harmonic}, except for the series of $\PU(1,n)$ and $\mathrm{SO}(1,n)$. For each positive integer $n$, there are infinitely many non-arithmetic lattices in the group $\mathrm{SO}(1,n)$ by Gromov--Piatetski-Shapiro \cite{gromov1987non}. In the case of $\PU(1,n)$ with $n\geq 2$, there are only $22$ commensurability classes of non-arithmetic lattices found so far for $n = 2$ by \cite{deligne1986monodromy}, \cite{mostow1986generalized, mostow1988discontinuous}, \cite{thurston1998shapes}, \cite{deraux2016new, deraux2021new}, and only two for $n = 3$ by \cite{deligne1986monodromy}, \cite{couwenberg2005geometric}, \cite{deraux2020new}. A primary source of non-arithmetic lattices in $\PU(1,n)$ is the Deligne--Mostow theory, which we recall in the following.

Let $\mu = (\mu_1, \cdots, \mu_{n+3})$ be an $(n+3)$-tuple of rational numbers where $0 < \mu_i < 1$ and $\sum\limits_{i=1}^{n+3} \mu_i = 2$. Deligne and Mostow studied the monodromy groups $\Gamma_\mu \subset \PU(1, n)$ of certain hypergeometric functions associated with such data. These groups naturally act on the complex hyperbolic ball $\BB^n$ of dimension $n$. There are three natural aspects on the classification of those groups: discreteness, arithmeticity, and commensurability. Deligne--Mostow \cite{deligne1986monodromy}, Mostow \cite{mostow1986generalized, mostow1988discontinuous} and Thurston \cite{thurston1998shapes} established criteria for the discreteness and arithmeticity of $\Gamma_\mu$. For $n \geq 2$, there are 104 tuples $\mu$ such that $\Gamma_\mu$ are discrete lattices in $\PU(1,n)$. This includes $94$ examples satisfying the so-called half-integer condition and $10$ exceptional cases, as listed in \cite{mostow1988discontinuous}. Among these $104$ examples, there are $19$ non-arithmetic lattices in $\PU(1,2)$ and one non-arithmetic lattice in $\PU(1,3)$.

The classification of discreteness and arithmeticity for $\Gamma_\mu\subset \PU(1, 1)$ is also known and related to the hyperbolic triangle groups. The complete list of infinitely many discrete lattices $\Gamma_\mu\subset \PU(1, 1)$ is given in \cite[Theorem 3.8]{mostow1988discontinuous}. These discrete lattices are related to hyperbolic triangle groups up to commensurability. The classification of arithmetic triangle groups and their commensurability classes is given by Takeuchi \cite{takeuchi1977commensurability}. The commensurability classes of all triangle groups are given by the work of Petersson \cite{Petersson1937berDE}, Greenberg \cite{Greenberg1963MaximalFG}, and Singerman \cite{Singerman1972FinitelyMF}. In particular, all commensurability classes of discrete $\Gamma_\mu\subset \PU(1, 1)$ are known.

For Deligne--Mostow tuples $\mu$ and $\nu$, we write $\mu\sim \nu$ if and only if $\Gamma_\mu$ and $\Gamma_\nu$ are commensurable in $\PU(1,n)$, namely, up to conjugation in $\PU(1,n)$, they share a common finite-index subgroup. Sauter \cite{sauter1990isomorphisms} and Deligne--Mostow \cite{deligne1993commensurabilities} found the following commensurability relations among Deligne--Mostow monodromy groups in $\PU(1,2)$: 

\begin{thm}[Sauter, Deligne--Mostow]
\label{theorem: deligne mostow sauter}
There are the following commensurability pairs of Deligne--Mostow data when $n=2$.
\begin{enumerate}[(i)]
\item $(a, a, b, b, 2-2a-2b) \sim (1-a, 1-b, 1-a-b, a + b-{1\over 2}, a+b-{1\over 2})$ for any positive real numbers $a, b$ with $a+b \in ({1\over 2}, 1)$; 
\item $(a, a, a, a, {2-4a})\sim (a, a, a, {1\over 2}-a, {3\over 2}-2a)$ for any real number $a \in ({1\over 4}, {1\over 2})$; 
\item $({1\over 6}, {1\over 6}, {1\over 6}, {7\over 6}-a, {1\over 3}+a)\sim (a, a, a, {2\over 3}-a, {4\over 3}-2 a)$ for any real number $a \in({1\over 6}, {2\over 3})$.
\end{enumerate}
\end{thm}

Notice that in Theorem \ref{theorem: deligne mostow sauter}, by iterating relation $(i)$ twice, one can obtain relation $(ii)$ as follows:
\[
(a,a,a,a,2-4a)\sim (1-a, 1-a, 2a-{1\over 2}, 2a-{1\over 2}, 1-2a)\sim (a,a,a, {1\over 2}-a, {3\over 2}-2a).
\]

The invariants that complete the commensurability classification of non-arithmetic lattices were obtained by Kappes--M\"oller \cite{kappes2016lyapunov} and McMullen \cite{McMullen2017Gauss-Bonnet} independently. They are called Lyapunov spectrum and volume ratios of cone manifolds, and are related to ratios of Chern numbers of Hodge bundles under Galois conjugation. These invariants are similar to Hirzebruch's proportionality for lattices \cite{Hirzebruch1958AutomorpheFU}, \cite{Mumford1977HirzebruchsPT}. The $20$ non-arithmetic Deligne--Mostow lattices fall into $9$ commensurability classes. However, this approach does not apply to arithmetic lattices, since the Lyapunov spectrum for arithmetic ball quotients are always zero, or Galois conjugations do not provide other conic complex hyperbolic structures on the moduli spaces.
	
Two new series of infinitely many commensurability relations of Deligne--Mostow monodromy groups in $\PU(1,3)$ are found in this paper. 
\begin{thm}
 \label{theorem: dimension 3}
There are the following commensurability pairs of Deligne--Mostow data when $n=3$.
		\begin{enumerate}
	\item $({1\over 6}, {1\over 6}, {1\over 6}, {1\over 6}, 1-a, {1\over 3}+a)\sim (a, a, a, {2\over 3}-a, {2\over 3}-a, {2\over 3}-a)$ for any real number $a \in (0, {2\over 3})$; 
\item $(a, a, a, a, 1-2a, 1-2a)\sim (a, a, a, {1\over 2}-a, {1\over 2}-a, 1-a)$ for any real number $a\in (0, {1\over 2})$.
		\end{enumerate}
\end{thm}
	
We have a unified algebro-geometric proof for Theorem \ref{theorem: deligne mostow sauter}, Theorem \ref{theorem: dimension 3}, and other pairs with explicit indices of common subgroups. The key geometric object of our approach is the cyclic cover of $(\PP^1)^2$ branched along a divisor of bidegree $(3,3)$. Under suitable numerical conditions (see Proposition \ref{proposition: types}), such a cover is a surface with a sub-Hodge structure of ball type. Meanwhile, this surface admits two natural fibrations over the projective line. To each fibration, we attach a Deligne--Mostow tuple. The Deligne--Mostow monodromy groups associated with the two tuples are shown to be commensurable. To state Theorem \ref{theorem: main by mentioning table 2}, we need to use some larger monodromy groups, denoted by $\Gamma_{\mu, H}$ with $H$ any group of permutations preserving points with the same weight. This notion will be recalled in \S\ref{subsection: review deligne mostow theory}. The common finite-index subgroup comes from the moduli spaces of such surfaces, and the indices divide the degrees of GIT moduli spaces.

\begin{thm} 
\label{theorem: main by mentioning table 2} 
The Deligne--Mostow monodromy groups $\Gamma_\mu$ and $\Gamma_\nu$ in each row of Table \ref{table: relation from geo} are commensurable. For each row, there are naturally groups $H_1$ and $H_2$ of permutations associated with $\mu$ and $\nu$, respectively, such that, up to conjugation in $\PU(1,n)$, the monodromy groups $\Gamma_{\mu, H_1}$ and $\Gamma_{\nu, H_2}$ share a common subgroup whose index in $\Gamma_{\mu, H_1}$ divides $\deg \pi_1$, and whose index in $\Gamma_{\nu, H_2}$ divides $\deg \pi_2$. 
\end{thm}

The pairs in Theorem \ref{theorem: deligne mostow sauter} are cases 8.1, 5.3 (or 7.2), and 3.2 in Table \ref{table: relation from geo}. The pairs in Theorem \ref{theorem: dimension 3} are cases 3.1 and 5.1 (or 7.1) in Table \ref{table: relation from geo}.

To obtain a complete classification of commensurability relations among Deligne--Mostow lattices, we also need commensurability invariants. There are two types of invariants used in our paper. Firstly we prove that if $\Gamma_\mu$ and $\Gamma_\nu$ are commensurable in $\PU(1, n)$ for $n\geq 2$, then the corresponding Hermitian spaces share the same cyclotomic field $K$; see Proposition \ref{prop: the same CM field}. This generalizes the previous known commensurability invariant called adjoint trace fields, which are the real cyclotomic fields in this case. The condition $n\geq 2$ is necessary here because there are exceptions for $n=1$ given by commensurable arithmetic triangle groups. This is essentially related to different automorphism groups of type $A_n$ Dynkin diagram with $n=1$ and $n\geq 2$. Since the two Hermitian spaces are defined over the same cyclotomic fields, it makes sense to discuss the conformality between the two Hermitian spaces. It turns out that the conformal classes of Hermitian spaces form commensurability invariants for Deligne--Mostow monodromy groups $\Gamma_\mu$. Furthermore, we give an explicit way to calculate the Hermitian forms based on colliding weights in Deligne--Mostow data and comparing the conformal classes by determinants and signatures. When $\Gamma_\mu$ is arithmetic, we prove that the conformal classes provide the complete commensurability invariants. So this completes the commensurability classification of arithmetic Deligne--Mostow lattices. Part of the conformal invariant is the signature under Galois conjugation, and this invariant already appeared before in \cite[\S6.2]{deraux2021new} and is called signature spectrum and non-arithmeticity index. 

The second commensurability invariant discussed in \S\ref{section: toroidal boundary divisors} is the boundary divisors of toroidal compactifications. This is essentially the same invariant called Heisenberg groups at cusps which appeared in Deraux's work \cite[\S6]{deraux2020new} distinguishing Deligne--Mostow and Couwenberg--Heckman--Looijenga non-arithmetic lattices in $\PU(1,3)$. The above two invariants can be used to distinguish non-arithmetic Deligne--Mostow lattices; see Corollary \ref{cor: distinguishing nonarithemetic DM by signature} and Corollary \ref{Corollary: nonarithemtic with denominator 12}.

The main theorem about the commensurability classification of Deligne--Mostow lattices is as follows.
\begin{thm}
\label{theorem: full classification}
    The $104$ Deligne--Mostow lattices in $\PU(1,n)$ with $n\geq 2$ form $38$ commensurability classes listed in Table \ref{table: main table}.
\end{thm}

\noindent {\it Organization:}
The content of this paper is roughly divided into two parts. The first part is about commensurability relations via moduli of certain projective surfaces. In \S\ref{section: background} we review the backgrounds, including the Esnault--Viehweg formula for cyclic covers and the Deligne--Mostow theory. In \S\ref{section: cyclic cover of P1^2} we will use the Esnault--Viehweg formula to derive the numerical conditions and classifications for cyclic covers $S$ of $\PP^1\times \PP^1$ with Hodge structures of ball type. In \S\ref{section: monodromy of fibrations}, we calculate the monodromy for the two projections $S\to \PP^1$ around the discriminant points. This allows us to explicitly write down the two Deligne--Mostow tuples associated with $S$. In \S\ref{section: moduli space}, we show that the two Deligne--Mostow monodromy groups arising from the two fibrations are commensurable via the moduli spaces of surfaces $S$. 

The second part deals with commensurability invariants. We discuss the conformal class invariant and prove that it classifies the commensurability relations among arithmetic Deligne--Mostow lattices in \S\ref{section: arithmetic lattices}. In \S\ref{section: degeneration method} we establish relations between Hermitian forms of adjacent ranks via a degeneration method. This provides the main tables of classifications in \S\ref{section: tables}. The commensurability invariant, boundary divisors of toroidal compactifications, is discussed in \S\ref{section: toroidal boundary divisors}.

\noindent {\it Acknowledgments:} We thank Prof. Mikhail Borovoi for answering our questions on MathOverflow, which is crucial for the proof of Proposition \ref{prop: the same CM field}. We thank YMSC and BIMSA for providing excellent environment for our collaboration. It is our pleasure to thank Eduard Looijenga, Ben Moonen, Michael Rapoport, Gerard van der Geer, Jinxing Xu and Kang Zuo for their comments and suggestions.

\section{Preliminary Result}
\label{section: background}
In this section we review the Esnault--Viehweg Formula (Proposition \ref{Proposition: Esnault--Viehweg}) which will be used in \S\ref{section: cyclic cover of P1^2} to classify certain cyclic covers with Hodge structures of ball types. We also recall necessary knowledge about Deligne--Mostow theory.

\subsection{Cyclic Covers and Esnault--Viehweg Formula}
\label{subsection: cyclic cover and EV formula}
We first recall the construction of cyclic covers. Let $X$ be an $n$-dimensional smooth projective variety with line bundles $L, L_1, \cdots, L_m\in \Pic(X)$. Suppose for each $j\in\{1,\cdots, m\}$, the divisor $D_j$ is irreducible and defined by a section $f_j\in H^0(X, L_j)$. Let $d, a_1,\cdots, a_m$ be positive integers such that $\gcd(d,a_1,\cdots, a_m)=1$ and suppose $D=a_1D_1+\cdots +a_m D_m\in |dL|$. Then the normalization $Z$ of the algebraic variety defined by the equation
\begin{equation}
\label{equation: define cyclic cover}
z^d=(f_1)^{a_1}\cdots (f_m)^{a_m}
	\end{equation}
in $\PP(\calO\oplus L)$ is called the $d$-fold cyclic cover of $X$ branching along $D$. 

If $D_1+\cdots+D_m$ is simple normal crossing, then the variety $Z$ has quotient singularities. In this case the cohomology groups $H^k(Z, \QQ)$ has pure $\QQ$-Hodge structure of weight $k$ by \cite{steenbrink1977mixed}. The cyclic group $C_d=\langle\zeta_d=\exp({2\pi\sqrt{-1}\over d})\rangle$ operates on $Z$ by $\zeta_d\cdot(z,x)=(z, \zeta_d\cdot x)$. This induces an operation of $C_d$ on the cohomology group of $Z$ by $g\cdot \gamma=(g^{-1})^*(\gamma)$ for any $g\in C_d$ and $\gamma\in H^*(Z)$. Let $\chi\colon C_d\to \CC^*$ be the tautological character. The operation of $C_d$ induces the character decomposition 
\begin{equation*}
H^n(Z, \QQ[\zeta_d])=\bigoplus_{i=0}^{d-1}H^n_{\chi^i}(Z, \QQ[\zeta_d])
	\end{equation*}
and
\begin{equation*}
H^{p,q}(Z)=\bigoplus_{i=0}^{d-1}H^{p,q}_{\chi^i}(Z).
	\end{equation*}

Next we recall the Esnault--Viehweg Formula. We denote by $\{r\}=r-[r]$ the fractional part of a real number $r$. Recall $D=\sum\limits_{j=1}^m a_j D_j$. 
\begin{prop}[Esnault--Viehweg \cite{esnault1992lectures}]
\label{Proposition: Esnault--Viehweg}
Let be given $i\in\{1,\cdots,d-1\}$ with $\gcd(i,d)=1$. Suppose $1\leq a_j\leq d-1$ for $1\le j\le m$. Then the components of the Hodge decomposition of $H^{n}(Z)$ can be described as follows:
\begin{equation*}
H^{n-q,q}_{\chi^i}(Z)\cong H^q(X, \Omega_X^{n-q}(\log D)\otimes \calO (-\sum_{j=1}^m\{{ia_j\over d}\}D_j)).
\end{equation*}
Especially, for $q=0$, we have 
\begin{equation*}
H^{n,0}_{\chi^i}(Z)\cong H^0(X, K_X\otimes  \calO(\sum_{j=1}^m (1-\{{ia_j\over d}\}) D_j)).
\end{equation*}
\end{prop}

\subsection{Hodge Structures of Ball Type}
The following definition is used throughout this paper.
\begin{defn}[{$\QQ[\zeta_d]$}-Hodge structure of ball type]
A $\QQ[\zeta_d]$-Hodge structure of ball type is a $\QQ[\zeta_d]$-vector space $V$ (of dimension $n$) together with a Hermitian pairing
\begin{equation*}
h\colon  V\times V\to \QQ[\zeta_d]
		\end{equation*}
of signature $(1, n-1)$ and a filtration
		\begin{equation*}
F^1\subset F^0=V_\CC\coloneqq V\otimes_{\QQ[\zeta_d]}\CC
		\end{equation*}
such that $\dim F^1=1$ and $h(x,x)>0$ for any $x\in F^1-\{0\}$.
	\end{defn}

We can construct $\QQ[\zeta_d]$-Hodge structure of ball type from certain polarized $\QQ$-Hodge structures of weight two or one together with cyclic group operation. We describe such constructions in the following two examples.
\begin{ex}
Let $d\geq 3$ be an integer and $(W, b)$ a polarized $\QQ$-Hodge structure of weight two and signature $(2, m)$. Here $b\colon W\times W\to \QQ$ is a symmetric bilinear form of signature $(2, m)$ and naturally extends to $W_\CC$. Suppose the Hodge decomposition is $W_\CC=W^{2,0}\oplus W^{1,1}\oplus W^{0,2}$. Then $b(x,\overline{x})>0$ for $x\in W^{2,0}-\{0\}$. Assume there is a cyclic group $C_d$ action on $W$, then we have the characteristic decomposition
		\begin{equation*}
			W_{\QQ[\zeta_d]}=\bigoplus_{i=0}^{d-1}(W_{\QQ[\zeta_d]})_{\chi^i}.
		\end{equation*}
Here $\chi$ is the tautological character of $C_d$. Assume moreover $W^{2,0}\subset (W_\CC)_\chi$. Then we can take $V=(W_{\QQ[\zeta_d]})_\chi$ and define a Hermitian form $h$ on $V$ via $h(x,y)=b(x,\overline{y})\in \QQ[\zeta_d]$ for any $x,y\in V$. Then $(V, h)$ together with the Hodge filtration $F^1=W^{2,0}\subset V_\CC=F^0$ is a $\QQ[\zeta_d]$-Hodge structure of ball type. See \cite[Domain of type IV, Section 3]{looijenga2016moduli} for a discussion on examples of this type.
	\end{ex}
	
	\begin{ex}
Let $(W, b)$ be a polarized $\QQ$-Hodge structure of weight one. Here $b\colon W\times W\to \QQ$ is a symplectic bilinear form and naturally extends to $W_\CC$. Suppose the Hodge decomposition is $W_\CC=W^{1,0}\oplus W^{0,1}$. By Hodge-Riemann bilinear relation, $\sqrt{-1} b(x, \overline{x})>0$ for any $x\in W^{1,0}$. Assume there is a $C_d$ action on $W$, then we have
		\begin{equation*}			W_{\QQ[\zeta_d]}=\bigoplus_{i=0}^{d-1}(W_{\QQ[\zeta_d]})_{\chi^i}.
		\end{equation*}
Here $\chi$ is the tautological character of $C_d$. Assume moreover $\dim(W^{1,0}_\chi)=1$. Let $a\in \QQ[\zeta_d], a\notin \RR$ and assume $\sqrt{-1}(a-\overline{a})<0$. Then we can take $V=(W_{\QQ[\zeta_d]})_\chi$ and define a Hermitian form $h$ on $V$ via $h(x,y)=(a-\overline{a})\cdot b(x, \overline{y})$ for any $x, y\in V$. Then $(V, h)$ together with the Hodge filtration $F^1=W^{1,0}_\chi \subset V_\CC=F^0$ is a $\QQ[\zeta_d]$-Hodge structure of ball type. See \cite[Domain of type III, \S3]{looijenga2016moduli} for a discussion on examples of this type.
	\end{ex}

\subsection{Deligne--Mostow Theory}
\label{subsection: review deligne mostow theory}
In this section we review the Deligne--Mostow theory. Let $n$ be a nonnegative integer. Take a tuple $\mu=(\mu_1, \cdots, \mu_{n+3})$ of rational numbers in $(0,1)$ such that $\sum\limits_{i=1}^{n+3}\mu_i\in \ZZ$. Let $d$ be the least common denominator of $\mu_i$ and $a_i=d\mu_i$. Consider $n+3$ points $A=\{x_1, \cdots, x_{n+3}\}$ on $\PP^1$. Let $C_\mu$ be the $d$-fold cyclic cover of $\PP^1$ with branching locus $\sum\limits_{i=1}^{n+3} a_i x_i$. Explicitly, the curve $C_\mu$ is the normalization of the curve 
\begin{equation}
\label{equation: cyclic cover of P^1}
z^d=(x-x_1)^{a_1}\cdots (x-x_{n+3})^{a_{n+3}}.
\end{equation} 
The Hodge numbers of character eigenspaces directly follow from Esnault--Viehweg formula (Proposition \ref{Proposition: Esnault--Viehweg}).
\begin{prop}
\label{prop: signature of DM hermitian space}
There are equalities
\begin{align*}
\dim H^{0,1}_{\chi}(C_\mu)=\dim H^{1,0}_{\overline{\chi}}(C_\mu)=\sum\limits_{i=1}^{n+3}\mu_i-1, \\
\dim H^{0,1}_{\overline{\chi}}(C_\mu)
= \dim H^{1,0}_{\chi}(C_\mu)
= n+2-\sum\limits_{i=1}^{n+3}\mu_i.
\end{align*}
\end{prop}

\begin{defn}[Deligne--Mostow Hermitian space]
\label{definition: Deligne--Mostow Hermitian space}
The natural Poincar\'e pairing 
\[
H^{1}_{\overline{\chi}}(C,\QQ[\zeta_d])\times H^{1}_{\chi}(C,\QQ[\zeta_d])\to \QQ[\zeta_d]
\] 
together with natural isomorphism $H^{1}_{\overline{\chi}}(C,\QQ[\zeta_d])\cong \overline{H^{1}_{\chi}(C,\QQ[\zeta_d])}$ induces a skew-Hermitian form on $H^{1}_{\overline{\chi}}(C,\QQ[\zeta_d])$. After multiplying the pairing by $\zeta_d-\overline{\zeta}_d$, it become a Hermitian space. We denote this Hermitian space by $(V_\mu, h_\mu)$.
\end{defn}

Under the tautological embedding of $\QQ[\zeta_d]\subset \CC$, the Hermitian pairing $h_\mu$ is positive definite on $H^{1,0}_{\overline{\chi}}(C)$ and negative definite on $H^{0,1}_{\overline{\chi}}(C)$. In this case, the Esnault-Viehweg formula (Proposition \ref{prop: signature of DM hermitian space}) gives its signature.

\begin{cor}
The Deligne--Mostow Hermitian space $(V_\mu, h_\mu)$ has signature 
\[(\sum\limits_{i=1}^{n+3}\mu_i-1, n+2-\sum\limits_{i=1}^{n+3}\mu_i)\]
under the tautological embedding of $\QQ[\zeta_d]\subset \CC$. 
\end{cor}

In this paper, we mainly use the following two cases.
\begin{prop}
\label{proposition: sum 2 implies ball}
    When $\sum\limits_{k=1}^{n+3} \mu_k=2$, then $(V_\mu, h_\mu)$ has signature $(1, n)$ and 
$H^{1,0}_{\overline{\chi}}(C)\subset H^{1}_{\overline{\chi}}(C)$ form a $\QQ[\zeta_d]$-Hodge structure of ball type.
\end{prop}

\begin{prop}
\label{Proposition: Hodge numbers for curve C}
Suppose $n=0$ and $\mu_1+\mu_2+\mu_3=1$, then we have
\begin{equation*}
\dim H^1_\chi(C,\CC)=\dim H^{1,0}_\chi(C)=1   
\end{equation*}
\end{prop}

For a Deligne--Mostow tuple $\mu=(\mu_1, \cdots, \mu_{n+3})$, we have a moduli space $\calM_\mu$ of $n+3$ points on $\PP^1$ with weight $\mu$, given by GIT quotient
\begin{equation*}
\calM_\mu=\SL(2,\CC)\dbs [(\PP^1)^{n+3}-\text{diagonals},  \Large\boxtimes_{i=1}^{n+3}\calO(d\mu_i)]
\end{equation*}
Assume $\sum\limits_{i=1}^{n+3}\mu_i=2$. Let $\BB_\mu$ be the $n$-dimensional complex hyperbolic ball associated with $(V_\mu, h_\mu)$ in Definition \ref{definition: Deligne--Mostow Hermitian space} and Proposition \ref{proposition: sum 2 implies ball}, and denote by $\Gamma_\mu$ the monodromy group for the variation of Hodge structures of ball type on $(\PP^1)^{n+3}-\text{diagonals}$. 

When some of the weights $\mu_j$ are the same, there are other monodromy groups containing $\Gamma_\mu$ with finite index; see \cite{mostow1986generalized} and \cite{deligne1993commensurabilities}. Denote by $\Sigma$ the full permutation group of indices $\{1, \cdots, n+3\}$ with the same weights. Let $H$ be a subgroup of $\Sigma$ and act on $\calM_\mu$. We denote by $\calM_{\mu,H}\coloneqq H\bs \calM_\mu$. Then the monodromy representation extends to $\pi_1(\calM_{\mu,H})\to \PU(1,n)$. The image is denoted by $\Gamma_{\mu, H}$. It contains $\Gamma_\mu$ as a subgroup whose index divides $|H|$.

Deligne--Mostow (\cite{deligne1986monodromy}, \cite{mostow1988discontinuous}) proved that there are exactly 104 tuples $\mu$ with $n\ge 2$ such that $\Gamma_\mu$ are discrete in $\PU(1,n)$. For these cases, it turns out that in this case $n$ is at most $9$, and $\Gamma_\mu$ is always a lattice (but not necessarily arithmetic). For $94$ of them satisfying the so-call half-integer condition, the period map
\begin{equation*}
\calP\colon \calM_{\mu,\Sigma}\to \Gamma_{\mu, \Sigma}\bs \BB_\mu    
\end{equation*}
is an open embedding. For a later statement of our main classification, we list the 104 discrete cases in Table \ref{table: main table}.

\subsection{Homology with Coefficients in a Local System}
\label{subsection: locally finite homology}
The Deligne--Mostow Hermitian space has another interpretation in terms of cohomology of local systems on punctured sphere. We first recall the definitions of singular homology and locally finite (Borel-Moore) singular homology with coefficients in a local system. Let $\LL$ be a $\QQ[\zeta_d]$-local system on a manifold $X$. Let $C_k(X,\LL)$ be the set of finite formal sums of pairs $(\gamma, e)$, where $\gamma$ is a continuous map $\gamma\colon \Delta^k\to X$, and $e\in\Gamma(\Delta^k, \gamma^*\LL)$. We denote $\gamma\cdot e=(\gamma, e)$ for short. The differential 
\begin{equation*}
\partial\colon C_k(X,\LL)\to C_{k-1}(X,\LL) 
\end{equation*}
can be defined via $\partial (\gamma\cdot e)=\partial\gamma\cdot e$. The homology $H_k(X,\LL)$ is defined as the $k$-th homology of the complex $(C_*(X,\LL),\partial)$. The dual of the complex defines the cohomology $H^k(X, \LL^\vee)$, with $\LL^\vee$ the dual of $\LL$, and has a natural pairing with $H_k(X, \LL)$. 

Let $C_k^{lf}(X,\LL)$ be the set of formal sums of $\gamma\cdot e\in C_k(X,\LL)$ such that for any $x\in X$, there exists an open neighborhood of $x$ in $X$ intersecting with finitely many $\gamma(\Delta^k)$ for members $\gamma$ of the sum. There are also differentials $\partial\colon C_k^{lf}(X,\LL)\to C_{k-1}^{lf}(X,\LL)$. Let $H_k^{lf}(X,\LL)$ be the $k$-th homology of the complex $(C_*^{lf}(X,\LL),\partial)$. This is called locally finite singular homology with coefficients in the local system $\LL$. The dual gives cohomology with compact support $H^k_c(X, \LL^\vee)$.

In this paper we consider the case $X=\PP^1-A$, where $A=\{x_1,\cdots, x_{n+3}\}$ is set of $n+3$ distinct points on $\PP^1$ in the defining equation \eqref{equation: cyclic cover of P^1} for cyclic cover $\pi\colon C\to \PP^1$. The sheaf $\pi_*(\QQ[\zeta_d])$ has the cyclic group action defined by $g\cdot \gamma=(g^{-1})^*(\gamma)$ for any $g\in C_d$ and section $\gamma$. Then it decomposes as direct sum of character eigenspaces. 
\begin{prop}
  The $\overline{\chi}$-eigensheaf $\pi_*(\QQ[\zeta_d])_{\overline{\chi}}$ is a rank-one local system when restricted to $\PP^1-A$. Denote this local system by $\LL_\mu$. The monodromy of $\LL_\mu$ along a simple counterclockwise circle around $x_i$ in $\PP^1$ is multiplication with $e^{2\pi \sqrt{-1}\mu_i}$.
\end{prop}

\begin{proof}
    See \cite[\S12.9]{deligne1986monodromy} and notice that our notation of cyclic group operation on curve $C$ differs from that used in \cite{deligne1986monodromy} by an inverse.
\end{proof}
In the following discussion, we omit the subindex $\mu$ and denote the local system by $\LL$. The cohomology of local system $\LL$ is related to the character decompositions of $H^1(C)$ in \cite[Proposition 2.6.1, Corollary 2.21, \S 2.23]{deligne1986monodromy}.

\begin{prop}[Deligne--Mostow]
\label{proposition: deligne mostow}
If $ \mu_k\notin \ZZ$ for all $1\leq k\leq n+3$, then the natural map 
\[
H^1_c(\PP^1-A, \LL) \to H^1(\PP^1-A, \LL)
\]
is an isomorphism and both are isomorphic to $H^1_{\overline{\chi}}(C,\QQ[\zeta_d])$ by Leray-Hirsch theorem. The natural map
\begin{equation}
\label{equation: iso between homology and locally finite homology}
    H_1(\PP-A, \LL)\to H_1^{lf}(\PP-A, \LL)
\end{equation}
is also an isomorphism.
\end{prop}

Furthermore, this is also another description of Deligne--Mostow Hermitian spaces $V_\mu$. There is a natural bilinear form 
\begin{equation*}
H_1(\PP^1-A, \LL)\times H_1^{lf}(\PP^1-A, \LL^{\vee})\to \QQ[\zeta_d]
\end{equation*}
given by Poincar\'e duality or more explicitly, intersection of cycles. The isomorphism $H_1(\PP^1-A, \LL)\to H_1^{lf}(\PP^1-A, \LL)$ in Proposition \ref{proposition: deligne mostow} together with $\LL^\vee\cong \overline{\LL}$ induces a skew-Hermitian form on $H_1^{lf}(\PP^1-A, \LL)$:
\begin{equation}
\label{equation: paring for locally finite cohomology}
(\cdot, \cdot)\colon H_1^{lf}(\PP^1-A, \LL)\times H_1^{lf}(\PP^1-A, \LL) \to \QQ[\zeta_d].
\end{equation} 

Thus, the homology-cohomology pairing and the Poincar\'e pairing yield the isomorphisms
\[
H^1(\PP^1-A, \LL)\cong (H_1(\PP^1-A, \LL^\vee))^\vee \cong H_1^{lf}(\PP^1-A, \LL).
\]

So there is the following proposition.
\begin{prop}
\label{prop: iso from C to P^1}
After multiplying by $\zeta_{d}-\overline{\zeta}_{d}$, the pairing \eqref{equation: paring for locally finite cohomology} gives a Hermitian isometry
\[
H_1^{lf}(\PP^1-A, \LL)\cong H^1_{\overline{\chi}}(C)\cong (V_\mu, h_\mu)
\]
\end{prop}

In \S \ref{section: degeneration method}, we need a basis for $H_1^{lf}(\PP^1-A, \LL)$ for explicit calculation of $(V_\mu, h_\mu)$. Recall the construction in {\cite[Proposition 2.5.1]{deligne1986monodromy}}. Take two points $x_i,x_j\in A$ and consider a path $\gamma\colon [0,1]\to \PP^1$ such that $\gamma(0)=x_i, \gamma(1)=x_j$ and $\gamma(0,1)\subset \PP^1-A$. Take a section $e\in \Gamma((0,1), \gamma^* \LL)$. The open interval $(0,1)$ can be written as a countable sum of closed intervals $(0,1)=\sum\limits_{i=1}^{\infty} I_i$.  Then 
\begin{equation*}
\gamma\cdot e=\sum\limits_{i=1}^{\infty} \gamma|_{I_i}\cdot e|_{I_i} 
\end{equation*}
is a locally finite formal sum of elements in $C_1(\PP^1-A, \LL)$, and this defines an element in $H_1^{lf}(\PP^1-A, \LL)$. 

\begin{prop}[{\cite[Proposition 2.5.1]{deligne1986monodromy}}]
\label{prop: basis for lf homology}
Consider $n+1$ paths $\gamma_1, \cdots, \gamma_{n+1}$ such that $\gamma_i$ connects $x_i$ to $x_{i+1}$. Take a nonzero section $e_i\in \Gamma((0,1), \gamma_i^* \LL)$. Then the elements $\gamma_i\cdot e_i$ with $1\le i\le n+1$ form a basis of $H_1^{lf}(\PP^1-A, \LL)$.
\end{prop}

\section{Cyclic Covers of $\PP^1\times \PP^1$}
\label{section: cyclic cover of P1^2}
This section discusses the key geometric construction of this paper, the $d$-fold cyclic covers of $\PP^1\times \PP^1$. Many of these surfaces have previously been studied by Moonen \cite{moonen2018deligne} from a different perspective. One of our focus is on the two fibrations of such surfaces.

\subsection{Numerical Conditions}
Using the notation in \S\ref{subsection: cyclic cover and EV formula} for $X=\PP^1\times \PP^1$, surface $S=Z$ and applying the Esnault--Viehweg formula in this case, we obtain the following result for Hodge numbers.
\begin{prop}
\label{proposition: Hodge numbers for surface S}
Suppose the line bundles $L_j$ and $a_j$ satisfies 
\begin{equation}
\label{equation: sum of Lj}
\sum_{j=1}^m L_j=\calO(3,3)
\end{equation}
and
\begin{equation}
\label{equation: sum of ajdLj}
\sum_{j=1}^m {a_j\over d}L_j=\calO(1,1)
\end{equation}
in $\Pic(\PP^1\times \PP^1)$, then we have 
\begin{equation}
\dim H^{2,0}_\chi(S)=1,  \dim H^{0,2}_\chi(S)=0.  
\end{equation}
	\end{prop}
	\begin{proof}
Similarly as the case of curves, we apply Esnault--Viehweg formula to $X=\PP^1\times \PP^1$. When $i=1$ and $q=0$, we have 
\begin{equation*}
H^{2,0}_\chi (S)\cong H^0(\PP^1\times \PP^1, K_{\PP^1\times \PP^1}\otimes \calO(\sum_{j=1}^m D_j-\sum_{j=1}^{m} {a_j\over d}D_j)) \cong H^0(\PP^1\times \PP^1, \calO). 
\end{equation*}
So $\dim H^{2,0}_\chi (S)=1$. 
		
Choose $i=d-1$ and $q=0$ in Esnault--Viehweg formula, we have 
\begin{equation*}
H^{2,0}_{\overline{\chi}} (S)\cong H^{0}(\PP^1\times \PP^1, K_{\PP^1\times \PP^1}\otimes \calO(\sum_{j=1}^m D_j-\sum_{j=1}^m(1-{a_j\over d})D_j))\cong H^0(\PP^1\times\PP^1, \calO(-1,-1))=0.
\end{equation*}
So $H^{0,2}_\chi(S)\cong  \overline{H^{2,0}_{\overline{\chi}} (S)}=0$. 
\end{proof}

\begin{rmk}
Conditions \eqref{equation: sum of Lj} and \eqref{equation: sum of ajdLj} are sufficient but not necessary condition for such Hodge numbers.
\end{rmk}

More explicit classification of such numerical conditions is summarized in Proposition \ref{proposition: types}.
\begin{prop}
\label{proposition: types}
Up to permutations, conditions \eqref{equation: sum of Lj} and \eqref{equation: sum of ajdLj} are equivalent to the following cases:
	\begin{enumerate}
		\item $L_1=(3,3), a_1=1$;
		\item $L_1=(3,2), L_2=(0,1), d=3, a_1=a_2=1$;
		\item $L_1=(3,1), L_2=L_3=(0,1), {a_1\over d}={1\over 3}, {a_2\over d}+{a_3\over d}={2\over 3}$; 
		\item $L_1=(2,2), L_2=(1,1), 2a_1+a_2=d$;
		\item $L_1=(2,2), L_2=(1,0), L_3=(0,1), 2a_1+a_2=d, a_2=a_3$;
		\item $L_1=(2,1), L_2=(1,2), d=3, a_1=a_2=1$;
		\item $L_1=(2,1), L_2=(1,1), L_3=(0,1), 2a_1+a_2=d, a_1=a_3$;
		\item $L_1=(2,1), L_2=(1,0), L_3=(0,1), L_4=(0,1), 2a_1+a_2=a_1+a_3+a_4=d$;
		\item $L_1=L_2=L_3=(1,1), a_1+a_2+a_3=d$;
		\item $L_1=L_2=(1,1), L_3=(1,0), L_4=(0,1), a_1+a_2+a_3=d, a_3=a_4$;
		\item $L_1=(1,1), L_2=L_3=(1,0), L_4=L_5=(0,1), a_1+a_2+a_3=a_1+a_4+a_5=d$;
		\item $L_1=L_2=L_3=(1,0), L_4=L_5=L_6=(0,1), a_1+a_2+a_3=a_4+a_5+a_6=d$.
	\end{enumerate}
 \end{prop}

A direct corollary of Proposition \ref{proposition: Hodge numbers for surface S} is that 
	\begin{cor}
		When conditions \eqref{equation: sum of Lj} and \eqref{equation: sum of ajdLj} are satisfied, the eigenspace $H^2_\chi(S,\QQ[\zeta_d])$ has a $\QQ[\zeta_d]$-Hodge structure of ball type.
	\end{cor}

\subsection{Relation to Deligne--Mostow Theory}
Let $X=\PP^1\times \PP^1$ and let $S$ be a $d$-fold cover of $X$ satisfying conditions \eqref{equation: sum of Lj} and \eqref{equation: sum of ajdLj}. Let $p\colon S\to \PP^1$ be the projection of $S$ to one of the factors. The operation of $C_d$ on $S$ preserves the fibers of $p$. We denote by $A\subset \PP^1$ the discriminant set of $p$. The higher pushforward sheaf $R^k p_* \QQ[\zeta_d]$ has an induced operation of $C_d$ and decomposes as direct sum of character eigensheaves. We focus on $\calL=(R^1 p_{*}\QQ[\zeta_d])_\chi$. A generic fiber of $p$ is a $d$-fold cover of $\PP^1$ with total degree of branching points equals to $d$. From Proposition \ref{Proposition: Hodge numbers for curve C}, we have $\dim H^1_\chi(C,\QQ[\zeta_d])=1$ for a generic fiber $C$ of $p$. Thus the eigensheaf $\calL$ is a rank-one $\QQ[\zeta_d]$-local system on $\PP^1-A$ and denoted by $\LL$. In \S\ref{section: monodromy of fibrations} we will calculate the monodromy of $\calL$ around each point in $A$. Before that we identify the balls from the surface $S$ and the Deligne--Mostow data associated with a fibration $p\colon S\to \PP^1$. The case $d=3$ and $a_1=\cdots=a_m=1$ has been proved in \cite{yu2024calabi}. In this section we generalize the result (\cite[\S7.3]{yu2024calabi}) to the general case.

Let $U=p^{-1}(\PP^1-A)\subset S$ be the open subset consisting of smooth fibers. 

\begin{lem}
	\label{lemma: S U}
Assume $|A|=\dim H^{1,1}_\chi(S)+3$ and the monodromy of $\LL$ around each point of $A$ is not identity. Then the natural map $H^2_\chi(S, \QQ[\zeta_d])\to H^2_\chi(U, \QQ[\zeta_d])$ is an isomorphism of $\QQ[\zeta_d]$-vector spaces.
\end{lem}
		
\begin{proof}
	Consider the Leray spectral sequence
			\begin{equation*}
	E_2^{k-q,q}=H^{k-q}(\PP^1, R^q p_*\QQ[\zeta_d]) \implies H^{k}(S, \QQ[\zeta_d]),
			\end{equation*}
which converges at the second page. Since $C_d$ operates on the fibration, the sheaves $R^{q} p_*\QQ[\zeta_d]$ decomposes as direct sum of eigensubsheaves, and we have
			\begin{equation*}
(E_2^{k-q,q})_\chi=H^{k-q}(\PP^1, (R^q p_*\QQ[\zeta_d])_\chi)\implies H^{k}_\chi(S, \QQ[\zeta_d]),
			\end{equation*}
also converging at the second page. Note that the cyclic group $C_d$ operates on $H^0(C)$ and $H^2(C)$ as identity. So the stalks of $\chi$-eigensubsheaves $(R^0p_*\QQ[\zeta_d])_\chi$ and $(R^2p_*\QQ[\zeta_d])_\chi$ are zero away from $A$. In other words, these two sheaves are supported on $A$. This implies the vanishing of $H^2(\PP^1, (R^0p_*\QQ[\zeta_d])_\chi)$. Taking $k=2$ in the spectral sequence, the filtration on $H^{2}_\chi(S, \QQ[\zeta_d])$ becomes an exact sequence \begin{equation*}
0\to H^1(\PP^1, \calL)\to H^2_\chi(S, \QQ[\zeta_d])\to H^0(\PP^1,(R^2p_*\QQ[\zeta_d])_\chi)\to 0
\end{equation*} 
Similarly, the same spectral sequence for smooth fibration $U\to \PP^1-A$ gives isomorphism $H^2(U, \QQ[\zeta_d])_\chi\cong H^1(\PP^1-A, \LL)$ since $(R^2p_*\QQ[\zeta_d])_\chi$ is supported on $A$. Comparing the two spectral sequences, the morphism $H^1(\PP^1, \calL)\to H^1(\PP^1-A, \LL)$ is surjective since the monodromy of $\LL$ is nontrivial. We conclude that $H^2_\chi(S,\QQ[\zeta_d])\to H^2_\chi(U, \QQ[\zeta_d])$ is surjective. Since 
\begin{equation*}
\dim H^2_\chi(S,\CC)=1+\dim H^{1,1}_\chi(S)=|A|-2=\dim H^1(\PP^1-A,\LL),
\end{equation*}
we know that $H^2_\chi(S,\QQ[\zeta_d])\to H^2_\chi(U, \QQ[\zeta_d])$ is an isomorphism.
\end{proof}
	
\begin{thm}
\label{theorem: surface S to DM cohomology}
Assume $|A|=\dim H^{1,1}_\chi(S)+3$ and the monodromy of $\LL$ around each point of $A$ is not identity. Then after a rescaling of the Hermitian form, there is a natural isomorphism of $\QQ[\zeta_d]$-Hodge structures of ball type between 
		\begin{equation*}
H^2_\chi(S,\QQ[\zeta_d])\cong H^1(\PP^1-A, \LL)
	\end{equation*}
	\end{thm}
\begin{proof}	
By Lemma \ref{lemma: S U} and isomorphism $H^1(\PP^1-A, \LL)\cong H^2_\chi(U, \QQ[\zeta_d])$ we have an isomorphism $H^2_\chi(S, \QQ[\zeta_d])\cong H^1(\PP^1-A, \LL)$ as $\QQ[\zeta_d]$-vector spaces. Next we prove this isomorphism also preserves $\QQ[\zeta_d]$-Hodge structure on both sides. 

Analogously, considering $\LL^\vee$, we have the isomorphism: 
\begin{equation*}
H^2_{c, \overline{\chi}} (U, \QQ[\zeta_d])\cong H_c^1(\PP^1-A, \LL^\vee)\cong H^1(\PP^1-A, \LL^\vee)\cong H^2_{\overline{\chi}}(S,\QQ[\zeta_d]).
\end{equation*}
There is a nondegenerate bilinear form 
\begin{equation*}
H^1(\PP^1-A, \LL)\times H^1_c(\PP^1-A, \LL^\vee)\to H^2_c(\PP^1-A,\QQ[\zeta_d])\cong \QQ[\zeta_d]. 
\end{equation*}
This is compatible with the natural Poincar\'e pairing 
\begin{equation*}
H^2_\chi(U, \QQ[\zeta_d]) \times H^2_{c,\overline{\chi}}(U, \QQ[\zeta_d])\to H^4_c(U,\QQ[\zeta_d])\cong \QQ[\zeta_d].
\end{equation*}
So the Poincar\'e pairing gives rise to a Hermitian form on $H^1(\PP^1-A, \LL)$. This is the same Hermitian form on $H^2_\chi(S, \QQ[\zeta_d])$ induced by Poincar\'e pairing and $H^2_{\overline{\chi}}(S, \QQ[\zeta_d])\cong \overline{H^2(S, \QQ[\zeta_d])_\chi}$. The isomorphism between $H^4_c(U,\QQ[\zeta_d])$ and $H^2_c(\PP^1-A,\QQ[\zeta_d])$ is given by Leray-Hirsch theorem. So $$H^2_\chi(S,\QQ[\zeta_d])\cong H^1(\PP^1-A, \LL)$$ is an isomorphism as Hermitian spaces after a rescaling factor. The rescaling factor is given by the Hermitian form on one-dimensional $\QQ[\zeta_d]$ vector space $H^1(\PP^1-\{x_1, x_2, x_3\},\KK)$. The factor is 
\begin{equation*}
{1-\alpha_1\alpha_2\over (1-\alpha_1)(1-\alpha_2)}, \text{  where } \alpha_i=e^{-2\pi{\sqrt{-1}}{a_i\over d}}
\end{equation*}
by Proposition \ref{prop: intersection form}.
		
The Hodge filtrations on $H^2_\chi(S,\QQ[\zeta_d])$ and $H^1(\PP^1-A, \QQ[\zeta_d])$ are identified by
\begin{equation*}
H^0_\chi(S, K_S)\cong H^0(\PP^1, (p_*K_S)_\chi)\cong H^0(\PP^1, \Omega^1(\sum_{i} \mu_i a_i)(\LL)).
\end{equation*}
Here, the notion of a line bundle defined by rational divisors twisted by local system $\LL$ is the same as \cite[\S2.11]{deligne1986monodromy}. And this isomorphism is compatible with the isomorphism 
		\begin{equation*}
H^2(S, \QQ[\zeta_d])_\chi\cong H^1(\PP^1-A, \LL).
		\end{equation*}
So this is an isomorphism between $\QQ[\zeta_d]$-Hodge structures of ball-type.
\end{proof}
	
\begin{rmk}
When surface $S$ is a triple cover of the quadratic surface $\PP^1\times \PP^1$ branched along a generic curve of genus $4$, Theorem \ref{theorem: surface S to DM cohomology} specializes to \cite[Theorem 3]{kondo2000moduli}. 
\end{rmk}

\section{Monodromy of Fibrations}
\label{section: monodromy of fibrations}
Recall that $A$ is the branch locus of fibration $p\colon S\to \PP^1$, the $\chi$-eigensubsheaf $R^1p_*(\QQ[\zeta_d])_\chi$ is a rank-one local system on $\PP^1-A$. Thus, the monodromy around each point $a_k\in A$ is given by a complex number $\alpha_k$, which, as we will show, has the form $\alpha_k=\exp({2\pi \sqrt{-1}\mu_k})$ for some rational number $\mu_k$. In this section, we describe a way to calculate the numbers $\mu_k$ via drawing paths on $\PP^1$. This method was used in \cite[\S2.2]{deligne1986monodromy} to understand the generators for monodromy groups $\Gamma_\mu$. For elliptic fibrations $S\to \PP^1$, we can also obtain the same result from resolution of singularities of $S$ and Kodaira's classification of singular fibers; see \cite[\S7]{yu2024calabi} for some typical examples.

A general fiber of $p$ is a cyclic cover $\pi\colon C\to \PP^1$ with the form
\begin{equation*}
z^d=(x-x_1)^{a_1}(x-x_2)^{a_2}(x-x_3)^{a_3}.
\end{equation*}
Let $\KK$ be a local system on $\PP^1-\{x_1, x_2, x_3\}$ with monodromy $\exp{(2\pi\sqrt{-1}(-{a_i\over d}))}$ around each point $x_i$. Then Proposition \ref{prop: iso from C to P^1} gives an isomorphism
\[
H^1_\chi(C)\cong H_1^{lf}(\PP^1-\{x_1, x_2, x_3\},\KK).
\]
Let $\beta\colon [0,1]\to \PP^1$ be a path such that $\beta(t)\in \PP^1-\{x_1, x_2, x_3\}$ for $0<t<1$ and $\beta(0)=x_i, \beta(1)=x_j$ where $i\neq j\in \{1, 2, 3\}$, and choose a nonzero section $e\in H^0((0,1), \beta^*\KK)$. Then the pair $(\beta, e)$ defines a basis in $H_1^{lf}(\PP^1-\{x_1, x_2, x_3\},\KK)$ by Proposition \ref{prop: basis for lf homology}.
	
Next we calculate the monodromy for isotrivial degeneration for different singular fibers in the following local models and use them to carry out the monodromy data associated with surfaces $S$ and branch divisors $D$.
	
\subsection{Degree-two Ramification Case}
 Let $\Delta\coloneqq \{t\in \CC\colon |t|<2\}$ be the disc of radius $2$. Let $x_3\in \PP^1$ be a constant that $|x_3|>\sqrt{2}$. Let $a_1, a_3, d$ be positive integers such that $\gcd(a_1, a_3, d)=1$ and $2a_1+a_3=d$. Consider the one-parameter family of curves $p\colon \calC\to \Delta$ defined by 
	\begin{equation*}
		z^d=(x^2-t)^{a_1}(x-x_3)^{a_3}. 
	\end{equation*}
This family admits a fiberwise operation of $C_d$ via $\zeta_d(z,x)=(\zeta_d^{-1} z, x)$. We denote by $C_t=p^{-1}(t)$ the fiber. We have isomorphisms $C_t\cong C$ compatible with the operations of $C_d$ when $t\neq 0$. The eigensubsheaf $(R^1p_* \QQ[\zeta_d])_\chi$ forms a rank-one local system $\LL$ on $\Delta-\{0\}$. We call this type of degeneration a degree-two ramification. The picture of the branch locus for $\calC\to \Delta\times \PP^1$ is illustrated in Figure \ref{figure: degree-two ramification} and we call this local model the degree-two ramification case (or the half nodal case comparing to the nodal case in \S\ref{subsection: nodal case}).

 \begin{figure}[htp]
		\centering
		\includegraphics[width=9cm]{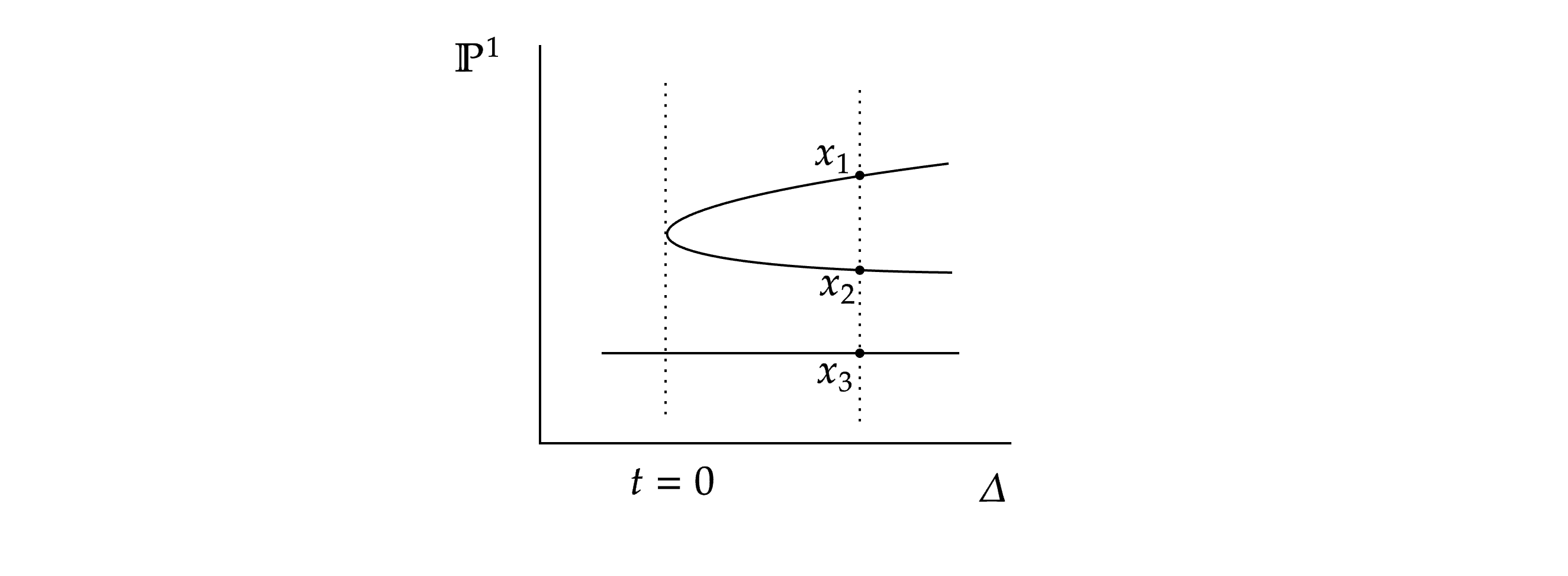}
		\caption{Degree-two ramification}
		\label{figure: degree-two ramification}
	\end{figure}
	
	\begin{prop}
\label{proposition: degree-two ramification}
In the degree-two ramification case, the monodromy of $\LL$ around $t=0$ is multiplication by $\exp({2\pi \sqrt{-1}({1\over 2}- {a_1\over d})})$.
\end{prop}
	
\begin{proof}
Let $t(s)=\exp({2\pi \sqrt{-1} s})$ with $s\in [0,1]$ be a counterclockwise loop around origin in $\Delta$. Along the loop, we have a continuous choice of $x_1=\sqrt{t}=\exp({\pi \sqrt{-1} s})$ and $x_2=-\sqrt{t}=\exp(\pi \sqrt{-1} (s+1))$ with respect to the parameter $s$. Let $B$ be a base point on $\PP^1-\{\Delta\cup \{x_3\}\}$. For each $s$, there is local system $\KK$ on $\PP^1-\{x_1, x_2, x_3\}$ defined as above by the cyclic cover $C_t\to \PP^1$. We fix an element $e' \in \KK|_B$ from the stalk of $\KK$ at $B$ and keep it unchanged along the loop. Let $\beta$ be a path on $\PP^1$ connecting $x_1$ to $x_2$. If we connect $B$ to the path $\beta$ and parallel transport $e'$ to $\beta$. Then we have a basis vector $(\beta, e)$ for $H^1(\PP^1-\{x_1, x_2, x_3\},\KK)$. Figure \ref{figure: degree-two ramification s} shows the picture for parallel transport $\beta(s)$ and $e(s)$ along the loop $t(s)$.

	\begin{figure}
			\centering
	\includegraphics[width=9cm]{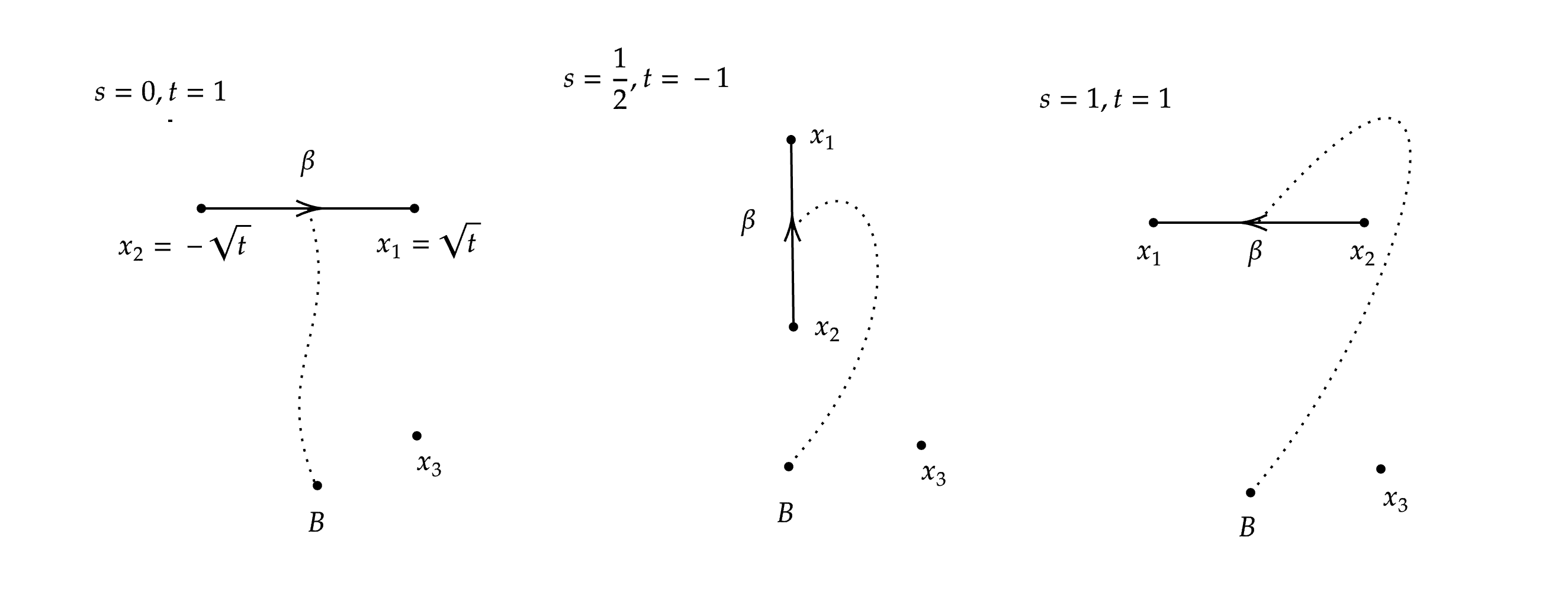}
		\caption{Parallel transport of $(\beta, e)$ for degree-two ramification}
			\label{figure: degree-two ramification s}
		\end{figure}
  
Comparing the picture for $s=0$ and $s=1$, we find that $\beta(1)=-\beta(0)$. We have 
\begin{equation*}
e(1)=\exp({2\pi \sqrt{-1}(1-{a_1\over d})})e(0)
\end{equation*}
since the dot line connecting $e(0)$ and $e(1)$ wind around $x=1$ counterclockwise once. So we have 
\begin{equation*}
(\beta(1), e(1))=-\exp({2\pi \sqrt{-1}(1-{a_1\over d})}) (\beta(0), e(0)).
\end{equation*}
So the monodromy of $\LL$ around $t=0$ is multiplication by $\exp({2\pi \sqrt{-1}({1\over 2}-{a_1\over d})})$.
\end{proof}
	
\subsection{Nodal Case}
\label{subsection: nodal case}
	We keep the same notation of $\Delta$ for the disc of radius $2$. Let $x_1(t)=t$, $x_2(t)=-t$, and $x_3\in \PP^1$ be a constant outside $\Delta$. Let $a_1, a_2, a_3, d$ be positive integers such that $\gcd(a_1, a_2, a_3, d)=1$ and $a_1+a_2+a_3=d$. Consider the one-parameter family of curves $p\colon \calC\to \Delta$ defined by 
	\begin{equation*}
		z^d=(x-x_1(t))^{a_1}(x-x_2(t))^{a_2}(x-x_3)^{a_3}. 
	\end{equation*}
	Similarly as before, we obtain a rank-one local system $\LL$ on $\Delta-\{0\}$. We call this local model the nodal case and the picture of the branch locus is in Figure \ref{figure: nodal case}. 

  \begin{figure}[htp]
		\centering
		\includegraphics[width=9cm]{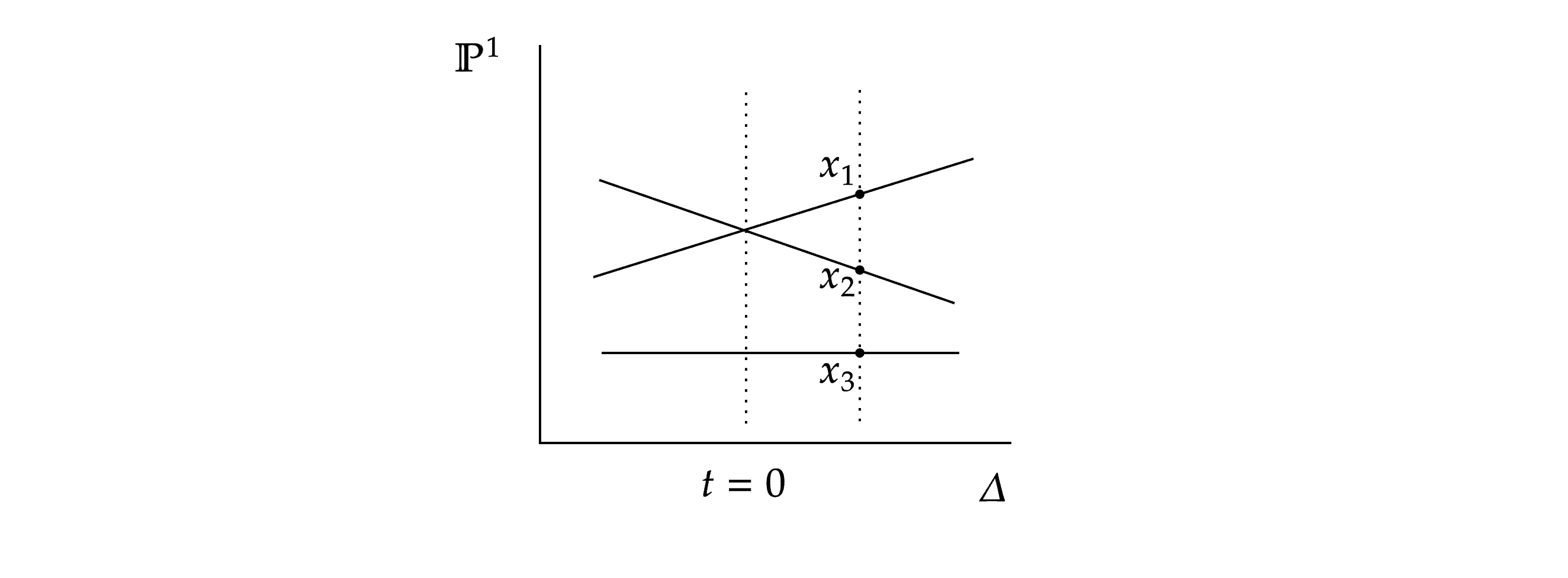}
		\caption{Nodal ramification}
		\label{figure: nodal case}
	\end{figure}
	
	\begin{prop}
		The monodromy of $\LL$ in the nodal case around $t=0$ is multiplication by $\exp({2\pi \sqrt{-1}(-{a_1\over d}-{a_2\over d})})$
	\end{prop}
	\begin{proof}
		We use the same notation in the proof of Proposition \ref{proposition: degree-two ramification}. Figure \ref{figure: nodal transport} shows the parallel transport of $(\beta, e)$.
	\end{proof}

\begin{figure}[htp]
		\centering
		\includegraphics[width=9cm]{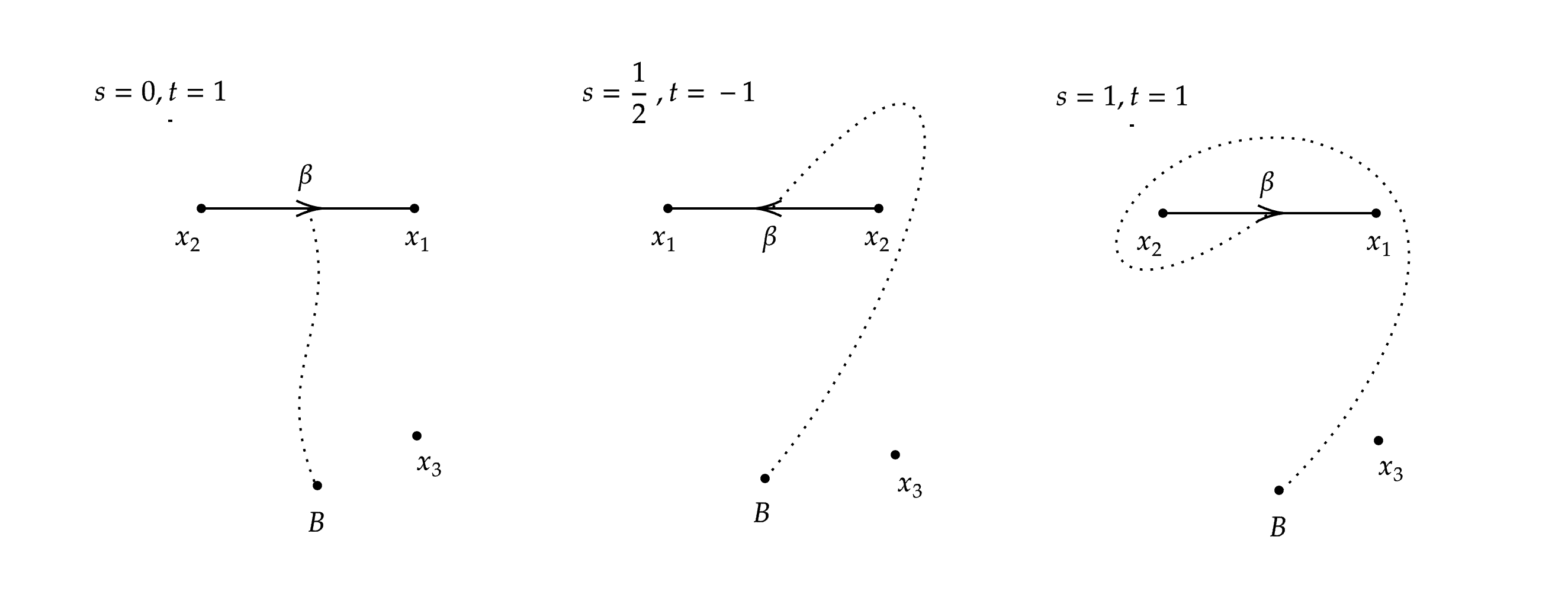}
		\caption{Parallel Transportation in Nodal Case}
		\label{figure: nodal transport}
	\end{figure}

	\subsection{Vertical Line Case}
	In this local model, we consider $x_1, x_2, x_3$ distinct constants on $\PP^1$ and the one-parameter family defined by equation
	\begin{equation*}
		z^d=t^{a_0}(x-x_1)^{a_1}(x-x_2)^{a_2}(x-x_3)^{a_3}.  
	\end{equation*}

\begin{figure}[htp]
		\centering
		\includegraphics[width=9cm]{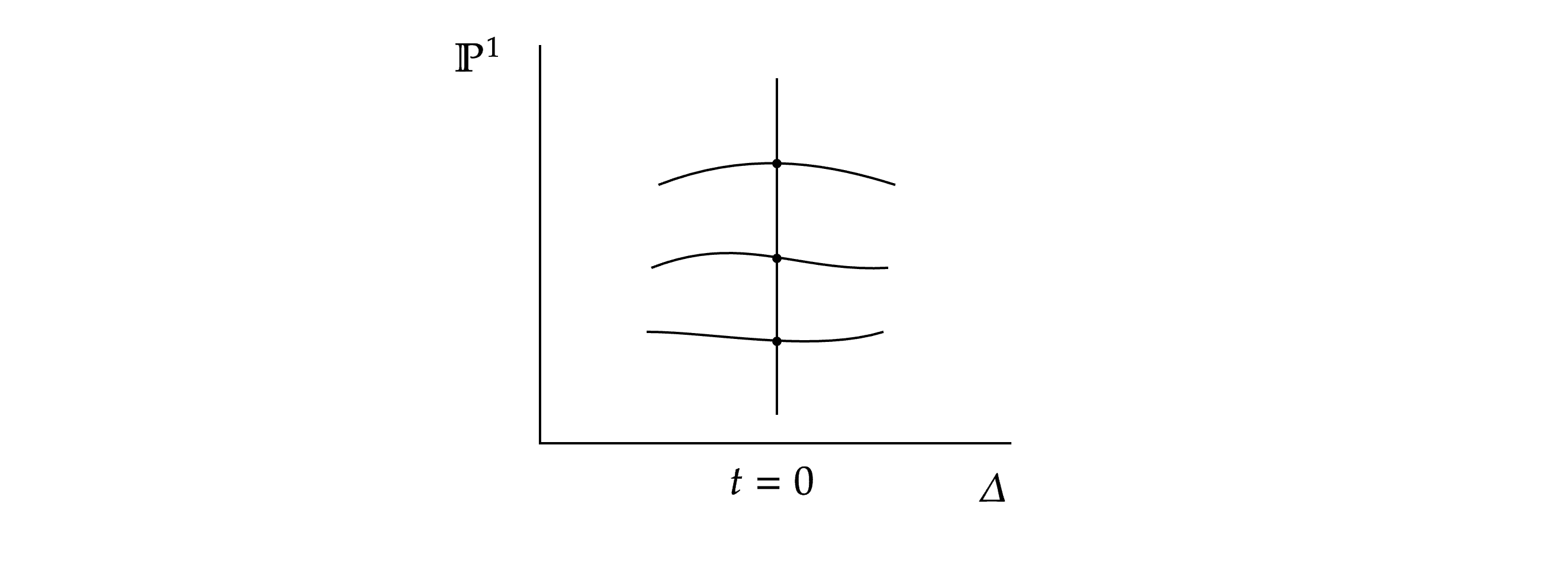}
		\caption{Vertical Line Case}
		\label{figure: vertical line}
	\end{figure}
 
	Similarly, this family induces a rank-one local system $\LL$ on $\Delta-\{0\}$. 
	Figure \ref{figure: vertical line} shows the branching locus, and we call this local model the vertical line case. 
	\begin{prop}
		The monodromy of $\LL$ in the vertical line case around $t=0$ is multiplication by $\exp({2\pi \sqrt{-1}(-{a_0\over d})})$.
	\end{prop}
	\begin{proof}
	We present another approach here by differential forms.	Let $z=Z\cdot t^{a_0\over d}$. We rearrange the equation as 
		\begin{equation*}
			Z^d=(x-x_1)^{a_1}(x-x_2)^{a_2}(x-x_3)^{a_3}.
		\end{equation*}
		After change of coordinate to $T=t^{a_0\over d}$, we see that the family parametrized by $T$ can be filled with a smooth central fiber and forms a constant family. When $t=\exp({2\pi\sqrt{-1}s})$ changes from $s=0$ to $s=1$, we have a monodromy changing $Z$ to $\zeta_d^{-{a_0}}\cdot Z$. If we use differential form 
		\begin{equation*}
			\omega={dx\over (x-x_1)^{1-{a_1\over d}}\cdot (x-x_2)^{1-{a_2\over d}}\cdot (x-x_3)^{1-{a_3\over d}}}={Z\cdot dx \over (x-x_1)(x-x_2)(x-x_3)}
		\end{equation*}
		to represent the basis for $H^{1,0}_\chi (C)$ for the constant family. We can see that the monodromy of $t=\exp({2\pi\sqrt{-1}s})$ moving from $s=0$ to $s=1$ changes $\omega$ to $\zeta_d^{-a_0}\cdot \omega$. 
	\end{proof}
\begin{rmk}	
This method using differential forms only works for families which are constant after a change of coordinates. For other families the differential forms may not represent a parallel section. The method of using singular cohomology still works here. The base point $B$ in the family winds around the central fiber, so the monodromy around this divisor gives the same result. 
\end{rmk}

\subsection{Tacnode Case}
We consider the tacnode case, which has two subcases with Figure \ref{figure: vertical tacnode} and Figure \ref{figure: nonvertical tacnode}.
\begin{figure}[htp]
		\centering
		\includegraphics[width=9cm]{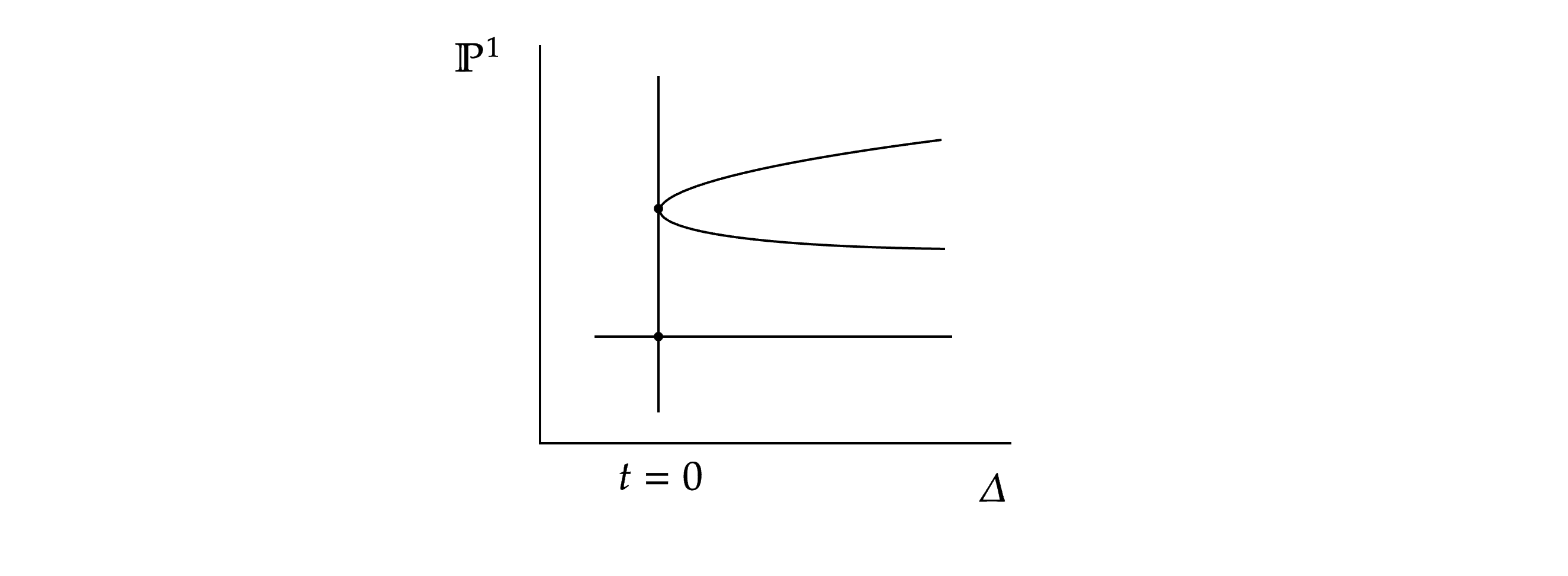}
		\caption{Vertical Tacnode Case}
		\label{figure: vertical tacnode}
	\end{figure}

\begin{figure}[htp]
		\centering
		\includegraphics[width=9cm]{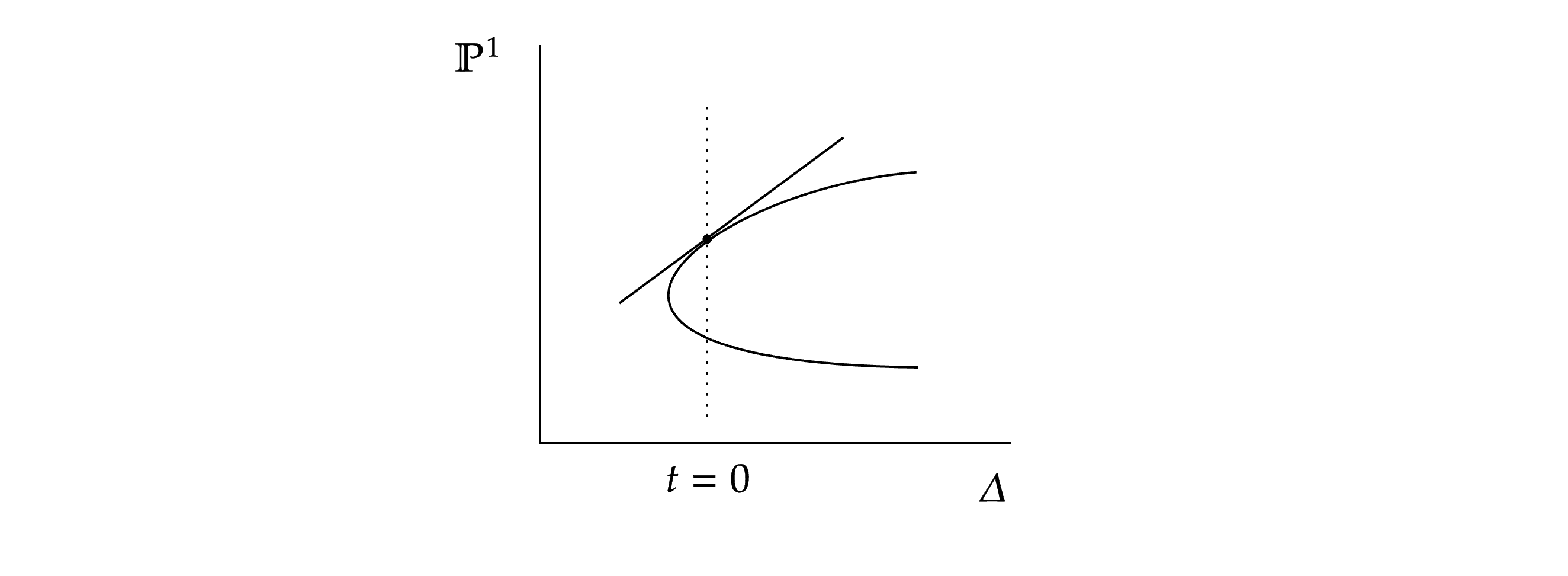}
		\caption{Non-vertical Tacnode Case}
		\label{figure: nonvertical tacnode}
	\end{figure}

For the vertical or non-vertical tacnode case, suppose the line has weight ${a_0\over d}$ and the divisor tangent with the line has weight ${a_1\over d}$, then we have the following monodromy calculation.
\begin{prop}
For vertical tacnode case in Figure \ref{figure: vertical tacnode}, the monodromy of $\LL$ around $t=0$ is multiplication by $\exp({2\pi \sqrt{-1}({1\over 2}-{a_0\over d}-{a_1\over d})})$. For non-vertical tacnode case in Figure \ref{figure: nonvertical tacnode}, the monodromy of $\LL$ around $t=0$ is multiplication by $\exp({2\pi \sqrt{-1}(-{2a_0\over d}-{2a_1\over d})})$.
\end{prop}
\begin{proof}
By moving the line, we see that the monodromy for vertical tacnode case is equal to the multiplication of monodromy for vertical case and half-nodal case, and the monodromy for non-vertical tacnode case equals to the multiplication of monodromy for two nodal cases.
\end{proof}

\section{Moduli Spaces and Monodromy Groups}
\label{section: moduli space}
In this section, we obtain a family version of Theorem \ref{theorem: surface S to DM cohomology} and compare the moduli spaces of surfaces $S$ with the Deligne--Mostow ball quotients. This implies the commensurability relations between two Deligne--Mostow monodromy groups induced by two fibrations of $S$ over $\PP^1$. The main result is Theorem \ref{theorem: main}. 

\subsection{Local Torelli for Moduli of Surfaces $S$}
Take a configuration $T=(L_1, \cdots, L_m)$ of line bundles on $\PP^1\times \PP^1$ and a tuple of positive integers $(d,a_1, \cdots, a_m)$ that satisfy conditions in Proposition \ref{proposition: Hodge numbers for surface S} or equivalently in the list of Proposition \ref{proposition: types}. Let $\calU\subset \prod\limits_{i=1}^m |L_i|$ be a nonempty Zariski open subset consisting of divisors $D=D_1+D_2+\cdots +D_m$ with $D_i\in |L_i|$ and the following conditions are satisfied: 
\begin{enumerate}
\item[(i)]  $\calU$ is preserved by the natural action of $G=\SL(2, \CC)\times \SL(2, \CC)$; 
\item[(ii)] all elements of $\calU$ are stable and have trivial stabilizer;
\end{enumerate}
By \cite[Proposition 3.3]{yu2024calabi}, generic elements of $\prod\limits_{i=1}^m |L_i|$ are stable, hence the second condition can be imposed.

\begin{rmk}
A natural choice of linearization would be $\calO(d-a_1)\boxtimes\cdots\boxtimes\calO(d-a_m)$ as it matches the polarization of ball quotient by Hodge bundle, but we do not need it here. 
\end{rmk}

Take GIT quotient $\calM\coloneqq G\dbs \calU$. The space $\calU$ is a $\PSL(3, \CC)$-bundle over $\calM$. By Luna slice theorem, after suitably shrinking $\calU$ and $\calM$, there exists a section in $\calU$ over $\calM$.

For a divisor $D$ of type $T$ we have a $d$-fold cover $S\to \PP^1\times \PP^1$. We denote by $\widetilde{\calU}$ the subset of $H^0(\PP^1\times \PP^1, \calO(3,3))$ consisting of elements $f_1\cdots f_m$ such that $f_i\in H^0(\PP^1\times \PP^1, L_i)$ defines an irreducible smooth divisor. Then $\widetilde{\calU}\to \calU$ is a $\CC^\times$-fiber bundle. There is a canonical family of $S$ on $\widetilde{\calU}$ given by $y^d=f_1^{a_1}\cdots f_m^{a_m}$. We can sufficiently shrink $\calM$, so that any points in the section over it, say $f_1\cdots f_m\in H^0(\PP^1\times \PP^1, \calO(3,3))=\CC[x_1, x_2; y_1, y_2]_{(3,3)}$, have a nonzero coefficient for $x_1^3 y_1^3$. Putting this coefficient to be $1$, we obtain a lifting of $\calM$ to $\widetilde{\calU}$. Thus there is a family of surfaces $S$ on $\calM$. 

We fix a base point $b\in \calM$, and let $S_b$ be the corresponding surface. The cohomology $H^2_{\chi}(S_b, \QQ[\zeta_d])$ has a Hodge structure of ball type. Let $n=\dim H^{1,1}_\chi(S_0)$. For any loop in $\calM$ based on $D_b$, using the family of $S$ over $\calM$, we obtain an automorphism of $H^2_{\chi}(S_b, \QQ[\zeta_d])$. This gives rise to the monodromy representation 
\[
\rho\colon \pi_1(\calM, b)\to \PU(H^2_\chi(S, \QQ[\zeta_d]))
\] 
which does not depend on the choice of the family, since the target is the projective unitary group. Define $\Gamma=\mathrm{Im}(\rho)$ the monodromy group. This group does not depend on the choice of $\calM$, since for any smooth quasi-projective variety $U$ and a Zariski-open subset $U_0$, the map $\pi_1(U_0)\to \pi_1(U)$ is surjective. Note that $\Gamma$ is not necessarily discrete. 

The positive complex lines in $H^2_{\chi}(S_b, \CC)$ form a complex hyperbolic ball $\BB^n$. For any element $D\in \calU$ with corresponding surface $S$ and a path $\gamma$ from $D$ to $D_b$ in $\calM$. Using the family of surfaces over $\calM$, we can transport $H^{2,0}_\chi (S)\subset H^{2}_\chi (S)$ to a positive line in $H^2_{\chi}(S_b, \CC)$. This defines a point in $\BB^n$, which only depends on the homotopy class of the path $\gamma$. Let $\widetilde{\calM}$ be the cover of $\calM$ corresponding to $\ker (\rho\colon \pi_1(\calM)\to \PU(1,n))$. Then there is a $\Gamma$-equivariant period map $\calP\colon \widetilde{\calM}\to \BB^n$. 
	
\begin{prop}[Local Torelli]
 \label{Proposition: local Torelli}
The period map $\calP\colon \widetilde{\calM} \to \BB^n$ is locally biholomorphic and
\[
\dim \calM=\dim H^{1,1}_\chi(S)=n.
\]
\end{prop}
\begin{proof}
The proof is similar to that of \cite[Theorem 2.5]{yu2024calabi}. Recall $D=\sum\limits_{i=1}^m D_i$ is simple normal crossing. Let $j_i\colon D_i\to X=\PP^1\times \PP^1$. We have the following short exact sequence
		\begin{equation*}
			0\to T_X(-\log D)\to T_X\to \bigoplus_i {(j_i)}_*N_{D_i/X}\to 0.
		\end{equation*}
		The induced long exact sequence gives
		\begin{equation*}
			H^0(X, T_X)\to \bigoplus_i H^0(D_i, N_{D_i/X})\to H^1(X, T_X(-\log D))\to 0.
		\end{equation*}
		Here we can identify the deformation space of $D_i$ in $X$ with $|L_i|$, and $H^0(D_i, N_{D_i/X})$ with $T_{D_i} |L_i|$. Hence the cokernel of $H^0(X, T_X)\to \bigoplus_i H^0(D_i, N_{D_i/X})$ is naturally isomorphic to $T_{[D]}\calM$.
		Hence $T_{[D]}\calM\to H^1(X, T_X(-\log D))$ is an isomorphism and we call this map the Kodaira-Spencer map. From Esnault--Viehweg formula in Proposition \ref{Proposition: Esnault--Viehweg}, we have \begin{equation*}
			H^{2,0}_\chi(S)\cong H^0(X, K_X\otimes \calO(D)\otimes L^{-1})
		\end{equation*}
		and \begin{equation*}
			H^{1,1}_\chi(S)\cong H^1(X, \Omega_X^1(\log D)\otimes L^{-1}) 
		\end{equation*}
		Under the assumption of $D$, we know that $K_X\otimes \calO(D)\otimes L^{-1}\cong \calO_X$. Since the wedge product 
		\begin{equation*}
			\bigwedge^2 \Omega_X^1(\log D)=K_X(D),
		\end{equation*}
		there is a non-degenerate bilinear pairing of vector bundles
		\begin{equation*}
			\Omega_X^1(\log D)\times \Omega_X^1(\log D)\to K_X(D). 
		\end{equation*}
		This induces an isomorphism 
		\begin{equation*}
			\Omega_X^1(\log D)\cong K_X(D)\otimes   (\Omega_X^1(\log D))^\vee \cong K_X(D)\otimes T_X(-\log D).
		\end{equation*}
		Tracing the relation between Kodaira-Spencer map and infinitesimal variation of Hodge structures, the tangent map of period map
		\begin{equation*}
			T_{[D]}\calM\to \Hom(H^{2,0}_\chi (S), H^{1,1}_\chi (S))
		\end{equation*}
		is an isomorphism.
	\end{proof}

\subsection{Relation of Moduli Spaces Induced by Projection}
We next relate the moduli space $\calM$ to Deligne--Mostow theory. We follow notation in \S\ref{subsection: review deligne mostow theory}. For Deligne--Mostow data $\mu=(\mu_1,\cdots, \mu_{n+3})$ with $\sum\limits_{i=1}^{n+3}\mu_i=2$, the $\QQ[\zeta_d]$-Hodge structures on $H^1(\PP^1-A,\LL_\mu)$ induce monodromy representation 
\[
\rho_{\mu,H} \colon \pi_1(\calM_{\mu,H})\to \PU(1, n).
\]
Let $\widetilde{\calM_{\mu,H}}$ be the covering of $\calM_{\mu,H}$ corresponding to $\ker \rho_{\mu,H}$. The period map is $\calP_{\mu,H}\colon \widetilde{\calM_{\mu,H}}\to \BB^n$. 
	
Consider the fibration $p\colon S\to \PP^1$ that projects to one component. Let $A$ be the discriminant set and let $H$ be the group of permutations of $A$ preserving discriminant points of the same type. For example, in the case shown in Figure \ref{figure: (2,2)+(1,0)+(0,1)illustrateH}, the points $x_1$ and $x_2$ have the same weight but not the same type, so $H$ is a proper subgroup of $\Sigma$.

  \begin{figure}[htp]
		\centering
\includegraphics[width=9cm]{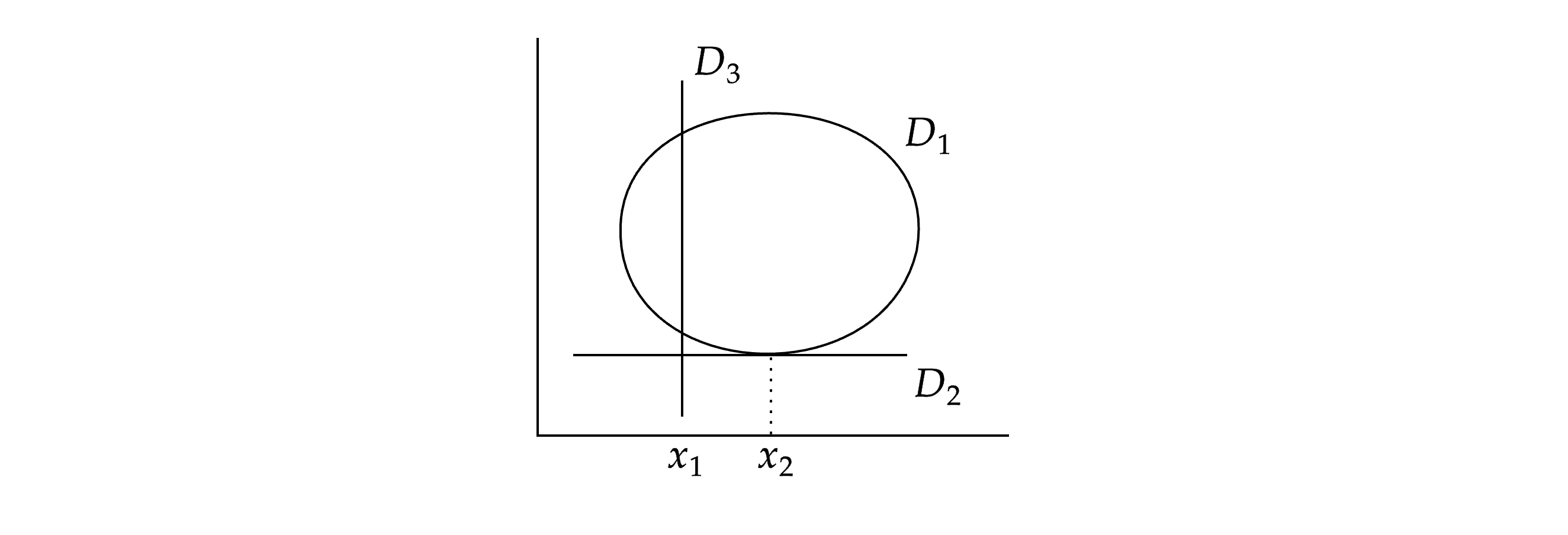}
		\caption{Configuration: (2,2)+(1,0)+(0,1) with One Tacnode}
		\label{figure: (2,2)+(1,0)+(0,1)illustrateH}
	\end{figure}

After suitably shrinking $\calU$, we may assume that for any $S$ in $\calM$ the number of discriminant points of $p\colon S\to \PP^1$ reaches the maximum. This gives rise to an algebraic map $\calM \to \calM_{\mu,H}$. The following result can be regarded as a family version of Theorem \ref{theorem: surface S to DM cohomology}.

\begin{prop}
\label{proposition: Monodromy from surfaces to DM}
If $|A|=\dim^{1,1}_\chi(S)+3=n+3$ and the monodromy of $\LL$ around each point of $A$ is not the identity, then the map $\calM \to \calM_{\mu,H}$ is generically finite. Denote by $k$ the degree. The monodromy group $\Gamma$ arising from surfaces $S$ is a finite index subgroup of $\Gamma_{\mu,H}$ with index dividing $k$.
\end{prop}
	
\begin{proof}
From Theorem \ref{theorem: surface S to DM cohomology}, we have isomorphism $H^2_\chi(S,\QQ[\zeta_d])\cong H^1(\PP^1-A, \LL_\mu)$ of Hodge structures. So the local period map from $\calP\colon \widetilde{\calM} \to \BB^n$ factors through $\calP_{\mu,H}\colon \widetilde{\calM_{\mu,H}}\to \BB^n$. Both maps are local isomorphisms. So $\widetilde{\calM} \to \widetilde{\calM_{\mu,H}}$ and $\calM\to \calM_{\mu,H}$ are local isomorphisms. Then $\calM\to \calM_{\mu,H}$ is generically finite of degree $k=\deg(\pi)$. 

We shrink $\calM$ and $\calM_{\mu,H}$ to make the finite map $\calM\to \calM_{\mu,H}$ a covering map. Then $\pi_1(\calM)\to \pi_1(\calM_{\mu,H})$ is injective with index $k$. The map $\rho$ is equal to the composition of 
\[
\pi_1(\calM)\hookrightarrow \pi_1(\calM_{\mu,H}) \xrightarrow{\rho_{\mu,H}} \PU(1,n).
\]
There is the group $\rho_{\mu,H}^{-1}(\Gamma)$ between $\pi_1(\calM)$ and $\pi_1(\calM_{\mu,H})$, and the index of $\Gamma<\Gamma_{\mu,H}$ is equal to that of $\rho_{\mu,H}^{-1}(\Gamma)<\pi_1(\calM_{\mu,H})$. Therefore, the index of $\Gamma$ in $\Gamma_{\mu,H}$ divides $k$.
\end{proof}

For two such subgroups $\Gamma_1, \Gamma_2$ of $\PU(1,n)$, we simply say that they are commensurable if they are commensurable in $\PU(1,n)$. For a type $T$ in the classification of Proposition \ref{proposition: types}, let $\mu, \nu$ be the two Deligne--Mostow data associated with the two projections $S\to \PP^1$ for a generic $D$ of type $T$ (according to calculations in \S\ref{section: monodromy of fibrations}). Let $H_1, H_2$ be the associated groups of permutations for $\mu, \nu$ respectively. We have the following commensurability result relating $\Gamma_{\mu,H_1}$ and $\Gamma_{\nu, H_2}$.
 	
\begin{thm}
\label{theorem: main}
Two monodromy groups $\Gamma_{\mu, H_2}$ and $\Gamma_{\nu, H_2}$ are commensurable. More precisely, up to conjugation in $\PU(1,n)$, they share a common finite-index subgroup $\Gamma$ from monodromy representations arising from surfaces $S$. Explicit commensurability relations obtained from this are listed in Table \ref{table: relation from geo}. 
\end{thm}
\begin{proof}
The theorem follows by applying Proposition \ref{proposition: Monodromy from surfaces to DM} to two projections of $S$ to $\PP^1$.
\end{proof}
	
\begin{rmk}
    Direct calculations (of discriminant loci $A$, monodromy of $\LL$ from \S\ref{section: monodromy of fibrations} and Hodge numbers of $S$ by Proposition \ref{Proposition: local Torelli}) show that the assumptions, $|A|=\dim H^{1,1}_\chi(S)+3$ and the monodromy of $\LL$ around each point of $A$ not being identity, are satisfied for all cases listed in Proposition \ref{proposition: types}. So we can remove them from Theorems \ref{theorem: surface S to DM cohomology} and \ref{theorem: main}.
\end{rmk}

Next we discuss the generalization of Theorem \ref{theorem: main} when the branching divisor degenerates to non-normal crossing singularities. When two divisors $D_j$ and $D_k$ in Proposition \ref{proposition: types} are tangent at a tacnode, the Hodge structure of $H^2_\chi(S)$ is not necessarily of ball type, since the branching divisor is no longer normal crossing. Two successive blowups give the following proposition.
\begin{prop}[Tacnode singularity]
 Theorem \ref{theorem: main} still holds when $D_j$ and $D_k$ are tangent with weights satisfying ${1\over 2}<{a_j\over d}+{a_k\over d}\leq 1$, and all other singularities of $D$ are normal crossing.
\end{prop}

\begin{proof}
Let $q$ be the tangent intersection point of $D_j$ and $D_k$, then the blowup of $X=\PP^1\times \PP^1$ at $q$ has an exceptional divisor $E$ intersecting with the strict transforms of $D_j$ and $D_k$ at one point $q^\prime$. Then the corresponding blowup of surface $S$ gives a cyclic cover given by normalization of equation
    \begin{equation*}
		z^d=(f_0)^{a_j+a_k}(\widetilde{f}_j)^{a_j} (\widetilde{f}_k)^{a_k}\Pi_{i\neq j,k}(f_i)^{a_i}.
\end{equation*}
where $f_0$ the defining section of $E$, and $(\widetilde{f}_j), (\widetilde{f}_k)$ the defining sections of strict transforms of $D_j$ and $D_k$. A second blowup at $q^\prime$ gives a surface $b\colon \widetilde{X}\to X$ with two exceptional divisors $E_1$ (from strict transform of $E$) and $E_2$. The corresponding blowups of $S$ give a cyclic cover $\widetilde{S}\to \widetilde{X}$ branching along $E_i$ and $\widetilde{D}_i$, the strict transforms of $D_i$. The multiplicity for $E_1$ is $a_j+a_k$, for $E_2$ is $2a_j+2a_k$, for $\widetilde{D}_i$ is $a_i$.

Write $\mu_i={a_i\over d}$. Then we have lemma:
\begin{lem}
\label{lemma: blowup numerical inequality}
The character space $H^{2,0}_\chi (\widetilde{S})$ has Hodge structure of ball-type if ${1\over 2}<\mu_{j}+\mu_{k}.$
Furthermore, the local Torelli (Proposition \ref{Proposition: local Torelli}) holds if	${1\over 2}<\mu_{j}+\mu_{k}\leq 1.$
\end{lem}

The proof of the lemma is straightforward application of Esnault-Viehweg formula and the details are as follows. 

The blowup $b\colon \widetilde{X}\to X$ relates the line bundles by
		\begin{enumerate}
			\item $K_{\widetilde{X}}\cong b^*(K_X)\otimes \calO(E_1+2E_2)$,
			\item $\calO(\widetilde{D}_{i})\cong b^*\calO({D}_{i})\otimes \calO(-E_1-2E_2)$, if $i\in \{j,k\}$,
			\item $\calO(\widetilde{D}_{i})\cong b^*\calO({D}_{i})$, if $i\not\in \{j,k\}$.
		\end{enumerate}
		We apply the Esnault-Viehweg formula in the following three cases.
		\begin{enumerate}
			\item If $2\mu_{j}+2\mu_{k}\not\in \ZZ$, then
			\begin{equation*}
				H^{2,0}_\chi(\widetilde{S})=H^0(\widetilde{X}, K_{\widetilde{X}}\otimes \calO(E_1+E_2+\sum_j \widetilde{D}_j)\otimes b^*(L^{-1})\otimes \calO([\mu_{j}+\mu_{k}]E_1+[2\mu_{j}+2\mu_{k}]E_2))
			\end{equation*}
			The right-hand side is equal to
			\begin{equation*}
				H^0(\widetilde{X}, \calO([\mu_{j}+\mu_{k}]E_1+([2\mu_{j}+2\mu_{k}]-1)E_2)).
			\end{equation*}
			
			\item If $2\mu_{j}+2\mu_{k}\in \ZZ$ and $\mu_{j}+\mu_{k}\not\in \ZZ$, then
			\begin{equation*}
				H^{2,0}_\chi(\widetilde{S})=H^0(\widetilde{X}, K_{\widetilde{X}}\otimes \calO(E_1+\sum_j \widetilde{D}_j)\otimes b^*(L^{-1})\otimes \calO([\mu_{j}+\mu_{k}]E_1+(2\mu_{j}+2\mu_{k})E_2))
			\end{equation*}
			The right hand side is equal to
			\begin{equation*}
				H^0(\widetilde{X}, \calO([\mu_{j}+\mu_{k}]E_1+(2\mu_{j}+2\mu_{k}-2)E_2)).
			\end{equation*}
			
			\item If $\mu_{j}+\mu_{k}\in \ZZ$, then
			\begin{equation*}
				H^{2,0}_\chi(\widetilde{S})=H^0(\widetilde{X}, K_{\widetilde{X}}\otimes \calO(\sum_j \widetilde{D}_j)\otimes b^*(L^{-1})\otimes \calO((\mu_{j}+\mu_{k})E_1+(2\mu_{j}+2\mu_{k})E_2))
			\end{equation*}
			The right-hand side is equal to
			\begin{equation*}
				H^0(\widetilde{X}, \calO((\mu_{j}+\mu_{k}-1)E_1+(2\mu_{j}+2\mu_{k}-2)E_2)).
			\end{equation*}
		\end{enumerate}
	In summary, we have $$H^{2,0}_\chi(\widetilde{S})=H^0(\widetilde{X},\calO(\lceil\mu_{j}+\mu_{k}-1\rceil E_1+\lceil 2\mu_{j}+2\mu_{k}-2\rceil E_2)).$$
		
		So $\dim H^{2,0}_\chi(\widetilde{S})\leq 1$ and the equality is maintained if and only if $2\mu_{j}+2\mu_{k}>1$. The same formula shows that  $\dim H^{2,0}_{\overline{\chi}}(\widetilde{S})=0$. So the Hodge structure on $H^2_\chi(\widetilde{S})$ is still of ball type.
   
   Under the assumption ${1\over 2}<{a_j\over d}+{a_k\over d}\leq 1$, we have $H^{2,0}_\chi (\widetilde{S})\cong H^0(\widetilde{X}, \calO)$,  The local Torelli holds in this case by the same formula as in Proposition \ref{Proposition: local Torelli}. This completes the proof of Lemma \ref{lemma: blowup numerical inequality}.
   
   Then the two fibrations $\widetilde{S}\to \PP^1$ provide the same commensurability result in Theorem \ref{theorem: main}.
\end{proof}

The same argument shows the following generalization.
\begin{prop}
    Theorem \ref{theorem: main} holds when $D_j$ and $D_k$ are tangent, $D_i$ passes through the tangent point with weights inequalities
    $$
    2<{2a_j\over d}+{2a_k\over d}+{a_i\over d}\leq 3\text{ and }1<{a_j\over d}+{a_k\over d}+{a_i\over d}\leq 2
    $$
  and other singularities of $D$ are normal crossing.
\end{prop}

\begin{rmk}[Remark on real weights]
    In \cite{deligne1986monodromy}, Deligne--Mostow also defined monodromy groups $\Gamma_\mu$ for $\mu=(\mu_1,\cdots,\mu_{n+3})$ with real weights $\mu_j\in (0,1)$. The commensurability result in Theorem \ref{theorem: main} still holds for real weights by constructions similar to Deligne--Mostow. Instead of cyclic covers over $\PP^1\times\PP^1$ defined by equation \eqref{equation: define cyclic cover}, we consider rank-one unitary local systems $\LL$ on $U=\PP^1\times \PP^1-\bigcup_j D_j$ with monodromy $\exp(2\pi \sqrt{-1} ({\mu^\prime_j}))$ around each divisor $D_j$, where $\mu^\prime_j$ are real numbers in $(0,1)$ satisfying the same conditions in Proposition \ref{proposition: Hodge numbers for surface S} as ${a_j\over d}$. Then the cohomology group $H^2(U, \LL)$ has a $\CC$-Hodge structure of ball type. The same arguments show that 
 this Hodge structure is the same as those defined in \cite{deligne1986monodromy}. And the monodromy group arising from varying $D_i$ forms a finite index subgroup in $\Gamma_\mu$ and $\Gamma_\nu$.
\end{rmk}

\section{Commensurability Invariant: Conformal Classes}
\label{section: arithmetic lattices}
In this and the next section, we discuss commensurability invariants of Deligne--Mostow monodromy groups $\Gamma_\mu$ by conformal classes of Hermitian forms. 

\subsection{Hermitian Forms and Projective Unitary Groups} 
We first collect some useful facts about Hermitian forms and the corresponding projective unitary groups. The goal is to prove the equivalence between conformal classes of Hermitian forms and isomorphism classes of projective unitary groups.

Let $F$ be a totally real number field and its quadratic extension $K$ be a CM field. The Galois group $\Gal(K/F)$ is generated by $\iota\colon a\mapsto \overline{a}$. A Hermitian form $h$ on a $K$-vector space $V$ is a map
\begin{equation*}
h\colon V\times V\to K 
\end{equation*}
such that $h(x,y)=\overline{h(y,x)}$ and $h$ is linear in the first component. We further assume $h$ is nondegenerate. For each real embedding $\sigma\colon F\to \RR$, denote by $\sgn_\sigma(h)=(p.q)$ the signature of the induced Hermitian form on $V\otimes \CC$ with $p$-positive and $q$-negative index of inertia. For any $K$-basis $v_1,\cdots,v_m$ of $V$, the determinant of Gram-Schmidt matrix $(h(v_i, v_j))$ determines a well-defined element $\det h\in F^\times /N_{K/F}(K^\times)$.

Similarly, if $h$ satisfies $h(x,y)=-\overline{h(y,x)}$, then we call $h$ a skew-Hermitian form. We can interchange skew-Hermitian and Hermitian forms by multiplying $h$ by $x-\overline{x}$ for any $x\in K-F$.

We have the following proposition that describes the isometry classes of $K$-Hermitian spaces. See \cite{landherr1935aquivalenz} or \cite[Page 268, Example 5]{Jacobson1940anoteonhermitianforms}
\begin{prop}[{\cite{landherr1935aquivalenz}}]
\label{prop: isometric classes of hermtian forms}
    Two nondegenerate $K$-Hermitian spaces $(V_1, h_1)$ and $(V_2, h_2)$ of the same dimension are isometric if and only if $sgn_\sigma(h_1)=sgn_\sigma(h_2)$ for all real embeddings $\sigma$ of $F$ and $\det(h_1)=\det(h_2)\in F^\times /N_{K/F}(K^\times)$.
\end{prop}

We say that two nondegenerate $K$-Hermitian spaces $(V_1, h_1)$ and $(V_2, h_2)$ are $F$-conformal if there exists a $K$-vector space isomorphism $f\colon V_1\to V_2$ such that $h_1=\lambda \cdot f^*(h_2)$ for some constant $\lambda\in F^\times$. We have the characterization of conformal classes as follows.

\begin{prop}
\label{prop: conformal classes of hermitian forms}
 Let $(V_1, h_1)$ and $(V_2, h_2)$ be two nondegenerate $K$-Hermitian spaces of dimension $m$. Then $(V_1, h_1)$ and $(V_2, h_2)$ are $F$-conformal if and only if 
 \begin{enumerate}
     \item When $m$ is odd, $\sgn_\sigma(h_1)=\pm \sgn_\sigma(h_2)$ for all real embeddings $\sigma$ of $F$.
     \item When $m$ is even, $\sgn_\sigma(h_1)=\pm \sgn_\sigma(h_2)$ for all real embeddings $\sigma$ of $F$ and $\det h_1=\det h_2\in F^\times/N_{K/F}(K^\times)$.
 \end{enumerate}
 Here $-\sgn_\sigma(h)$ is defined to be $(q, p)$ if $\sgn_\sigma(h)=(p,q)$.
\end{prop}

\begin{proof}
 We first prove the only if part. For any $\lambda\in F^\times$ and $K$-Hermitian form $(V, h)$, we have 
    \[
    \sgn_\sigma  (\lambda h)=\sgn(\sigma(\lambda))\sgn_\sigma(h).
    \]
So $\sgn_\sigma(\lambda h)=\pm \sgn_\sigma(h)$.
For determinants, we have \[
\det(\lambda h)=\lambda^m\det(h). 
\]
Notice that $N_{K/F}(\lambda)=\lambda^2$ for $\lambda\in F$.
When $m=2l+1$ is odd, $$\det(\lambda h)=\lambda N_{K/F}(\lambda^l)\det(h)=\lambda \det (h)\in F^\times/N_{K/F}(K^\times).$$
When $m=2l$ is even, \[\det(\lambda h)=N_{K/F}(\lambda^l)\det(h)=\det (h)\in F^\times/N_{K/F}(K^\times)\] in this case. 

Next we prove the if part. When $m$ is odd and $\sgn_\sigma(h_1)=\pm \sgn_\sigma(h_2)$ for all real embeddings $\sigma$ of $F$, we choose 
$$\lambda={\det (h_1)\over \det (h_2)}\in F^\times.$$
Then $\det (\lambda  h_2)=\det (h_1)\in F^\times/N_{K/F}(K^\times)$. We claim that $\sgn_\sigma(\lambda h_2)=\sgn_\sigma(h_1)$. For an odd-dimensional Hermitian space $(V,h)$, if $\sgn_\sigma(h)\in \{(p, q), (q, p)\}$, then $\sgn(\sigma(\det (h)))=(-1)^q$ or $(-1)^p$ and $(-1)^q\neq (-1)^p$. Hence $\sigma(\det (h))$ determines the choice of $\sgn_\sigma(h)$ in $\{(p, q), (q, p)\}$. Since $\det (\lambda  h_2)=\det (h_1)$ and $\sgn_\sigma(\lambda  h_2)=\pm \sgn_\sigma (h_1)$, we have $\sgn_\sigma(\lambda h_2)=\sgn_\sigma(h_1)$. So by Proposition \ref{prop: isometric classes of hermtian forms}, the Hermitian spaces $(V_1, h_1)$ and $(V_2, \lambda h_2)$ are isometric.

When $m$ is even, assume $\sgn_\sigma(h_1)=\pm \sgn_\sigma(h_2)$ for all real embeddings $\sigma$ of $F$ and $\det h_1=\det h_2\in F^\times/N_{K/F}(K^\times)$. Then $\det (\lambda h_2)=\det (h_2)=\det (h_1)\in F^\times/N_{K/F}(K^\times)$ for any $\lambda\in F^\times$. We choose $\lambda\in F^\times$ such that \[
\sgn(\sigma(\lambda))={\sgn_\sigma(h_1)\over \sgn_\sigma(h_2)}.
\]
This is possible because of the weak approximation theorem. Or denote by $\{\sigma_1, \cdots, \sigma_k\}$ the set of real embeddings of $F$ the embedding map, then the image of the map
\[
 F\to \RR^k, \lambda\mapsto (\sigma_i(\lambda))_{1\leq i\leq k}
\]
is dense in $\RR^k$. We choose $\lambda$ the element approximating the vector in $\RR^k$ with the desired the signature. So we have $\sgn_\sigma(h_1)=\pm \sgn_\sigma(\lambda h_2)$ for all real embeddings $\sigma$ of $F$. Hence $(V_1, h_1)$ is isometric to $(V_2, \lambda h_2)$.
\end{proof}

Denote by $\U(V,h)$ the unitary group consisting of linear transformations of $V$ preserving $h$, and $\PU(V,h)$ the corresponding projective group. Both are $F$-algebraic groups. We have the following exact sequence on the automorphism group of $\PU(V, h)$.

\begin{prop}
\label{prop: auto of PU}
Assume $m\geq 2$. Let $A_{m-1}$ be the Dynkin diagram of type $A$ with $m-1$ nodes. Then we have the following exact sequence of $F$-algebraic groups
\[
1\to \PU(V, h)\to \Aut(\PU(V, h))\to \Aut(A_{m-1})\to 1.
\]
When $m\geq 3$, there is an element of order two in $\Aut(\PU(V, h))(F)$ that maps to the generator of $\Aut(A_{m-1})(F)\cong \ZZ/2\ZZ$. When $m=2$, $\Aut(A_{m-1})(F)$ is trivial. The canonical map
\[ 
H^1(F,\PU(V,h)) \to H^1(F,{\Aut}(\PU(V,h)))
\]
is bijective when $m=2$ and is injective when $m\ge 3$. 
\end{prop}

\begin{proof}
    The proof was given by Mikhail Borovoi in the answer to our question on Mathoverflow \cite{borovoi2022}. Since $\PU(V, h)$ is semisimple and of adjoint type, we have the exact sequence by the isomorphism theorem of reductive groups of adjoint type. 

First we fix a $K$-basis $v_1, v_2,\cdots, v_m$ of $(V, h)$. Then the Gram matrix of $(V, h)$ is $M_h=(h(v_i, v_j))$. Then each element in $\U(V, h)$ is represented by a matrix $A\in \GL(m,K)$ satisfying $A^T M_h \overline{A}=M_h$. The map $A\mapsto (\overline{M_h})^{-1}\overline{A^T}\overline{M_h}$ defines an $F$-algebraic group automorphism of $\U(V, h)$ and descends to an involution of $\PU(V, h)$. When $m\geq 3$, this defines an outer automorphism of $\U(V, h)$ and induces the involution of the Dynkin diagram $A_{m-1}$. So does the automorphism of $\PU(V, h)$. When $m=2$, the automorphism group $\Aut(A_1)$ is trivial. Applying $\Gal(\overline{F}/F)$-action, we obtain a long exact sequence
    \begin{eqnarray*}
1\to &\PU(V, h)(F)\to \Aut(\PU(V, h))(F)\to \Aut(A_{m-1})(F)\to\\ &H^1(F, \PU(V, h))\to H^1(F, \Aut(\PU(V, h)))\to 1
    \end{eqnarray*}  
Since $\Aut(\PU(V, h))(F)\to \Aut(A_{m-1})(F)$ is surjective, the map 
\[
H^1(F, \PU(V, h))\to H^1(F, \Aut(\PU(V, h)))
\]
has trivial kernel. Using twisting, we show that this map is injective.
\end{proof}

When $(V_1, h_1)$ is $F$-conformal to $(V_2, h_2)$, there is an $F$-algebraic group isomorphism $\PU(V_1, h_1)\cong \PU(V_2, h_2)$ from the definition. When $m\geq 3$, we have a stronger form of the converse statement not requiring a priori assumption that $V_i$ are defined on the same CM field.
\begin{prop}
\label{prop: the same CM field}
Let $K_1$ and $K_2$ be two CM fields sharing the same totally real subfield $F$. Assume that $h_i$ is a nondegenerate Hermitian form on finite-dimensional $K_i$-vector space $V_i$, such that $\PU(V_1, h_1)\cong \PU(V_2, h_2)$ as $F$-algebraic groups. We have:
\begin{enumerate}
\item If $m=\dim V_1=\dim V_2\geq 3$, then $K_1=K_2$, and $(V_1, h_1)$ is $F$-conformal to $(V_2, h_2)$. 
\item If $m=\dim V_1=\dim V_2=2$ and $K_1=K_2$, then $(V_1, h_1)$ is $F$-conformal to $(V_2, h_2)$.
\end{enumerate}
\end{prop}

\begin{proof}
  First we prove under the assumption $K_1=K_2=K$. This case was communicated to us by Mikhail Borovoi on Mathoverflow \cite{borovoi2022}. For a Hermitian space, denote by $\GU(V, h)$ the group of conformal isomorphisms of $(V, h)$, or in other words, $K$-linear isomorphisms $f\colon V\to V$ such that $f^*h$ is $F$-conformal to $h$. The $F$-conformal class of Hermitian form $h$ is denoted by $(V, F^\times h)$. Then the twisted $F$-forms of $(V, F^\times h)$ are given by $H^1(F, \GU(V, h))$. On the other hand, the twisted $F$-forms of $\PU(V, h)$ correspond to elements in $H^1(F, \Aut(\PU(V,h)))$. The conjugation of $\GU(V, h)$ on $\U(V, h)$ induces a morphism $\GU(V, h)\to \Aut(\PU(V,h))$. The corresponding map 
  \begin{equation}
  \label{equation: twisted F-forms}
     H^1(F, \GU(V, h))\to H^1(F, \Aut(\PU(V,h))) 
  \end{equation}
represents the induced $F$-forms on $\PU(V, h)$ by twisted $F$-conformal classes of Hermitian forms. So the proposition is equivalent to the injectivity of the above map. The kernel of the homomorphism 
  $\GU(V,h)\to \Aut(\PU(V,h))$ is $\Res_{K/F}\mathbb{G}_m$ and the image is $\Aut^\circ(\PU(V,h))$.
  
The map \eqref{equation: twisted F-forms} factors as   
\[
H^1(F, \GU(V, h))\to H^1(F, \Aut^\circ(\PU(V,h)))\to H^1(F, \Aut(\PU(V,h))).
\]
The first map is injective because it fits into the exact sequence
\begin{equation*}
    H^1(F, \Res_{K/F}\mathbb{G}_m)\to H^1(F, \GU(V, h))\to H^1(F, \Aut^\circ(\PU(V,h))) 
\end{equation*}
and $H^1(F, \Res_{K/F}\mathbb{G}_m)=H^1(K, \mathbb{G}_m)=1$ by Shapiro's lemma and Hilbert's Theorem 90. The second map is injective by Proposition \ref{prop: auto of PU}. 

Next we prove that $K_1=K_2$ when $m\geq 3$. In fact, we have the following description of $K_i$ in terms of Galois cohomology. Let $K_i=F[\sqrt{-d_i}]$. Fixing the $F$-form on $\PU(1,n)$ induced by $(V_1, h_1)$, then the $F$-form induced by $(V_2, h_2)$ corresponds to an element $f$ in $H^1(F, \Aut(\PU(V_1, h_1)))$, which is sent to an element $\widetilde{f}\in H^1(F, \Aut(A_{m-1}))$ by 
\[
\Aut(\PU(V_1, h_1))\to \Aut(A_{m-1}).
\]
Under the isomorphisms $\Aut(A_{m-1})\cong \mu_2$ and 
\[
H^1(F, \Aut(A_{m-1}))\cong F^\times / (F^\times)^2,
\]
the element $\widetilde{f}$ corresponds to ${d_2/ d_1}\in F^\times / (F^\times)^2$.

The detailed calculation is as follows. Fixing orthogonal bases of $(V_i, h_i)$, assume the Gram matrices are diagonal matrices 
\[
M_{h_i}=\diag(\lambda_1^i, \cdots, \lambda_n^i)\in M_m(F). 
\]
Then elements in $\U(V_i, h_i)$ are given by matrices $z_i=x_i+\sqrt{-d_i}y_i$, $x_i, y_i\in M_m(F)$ satisfying equations
\begin{equation*}
	x_i^TM_{h_i}x_i+d_iy_i^TM_{h_i}y_i=M_{h_i},\, x_i^TM_{h_i}y_i=y_i^TM_{h_i}x_i.
\end{equation*}
Let $w=\diag(\sqrt{\lambda_1^1\over \lambda_1^2}, \cdots, \sqrt{\lambda_n^1\over \lambda_n^2})$ and $e=\sqrt{d_1\over d_2}$. Then there is an $\overline{F}$-algebraic group isomorphism $P\colon \U(V_1,h_1)(\overline{F})\to \U(V_2, h_2)(\overline{F})$ given by 
\begin{equation*}
(x_1, y_1)\mapsto w(x_1, \sqrt{d_1\over d_2}y_1)w^{-1}.
\end{equation*}
For any $\tau\in \Gal(\overline{F}/F)$, the $\tau$-twist of $P$ is given by $P^\tau(x_1, y_1)=\tau (P(\tau^{-1}(x_1, y_1)))$. Then we have a cocycle $f\colon \Gal(\overline{F}/F)\to \Aut(\U(V_1, h_1))$ given by $f(\tau)=P^{-1}P^\tau$. Direct calculation shows that 
$$f(\tau)(x_1, y_1)=u(x_1, e^{-1}\tau(e)y_1)u^{-1}$$
where $u=w^{-1}\tau(w)$ and $e^{-1}\tau(e)=\pm 1$. Since the automorphism $(x_1, y_1)\mapsto (x_1, -y_1)$ corresponds to the nontrivial element of $\Aut(A_{m-1})$ for $m\geq 3$ and the conjugation by $u$ is an inner automorphism, the element $\title{f}$ corresponds to ${d_2\over d_1}\in F^\times/(F^\times)^2$. The cocycle is mapped to $H^1(F, \Aut(\PU(V_1, h_1)))$ and corresponds to the same element in $F^\times/(F^\times)^2$.
\end{proof}

\begin{rmk}
The conclusion $K_1=K_2$ in Proposition \ref{prop: the same CM field} only holds when $m\geq 3$. There are examples of $\PU(V_1, h_1)\cong \PU(V_2, h_2)$ with $K_1=\QQ[\sqrt{-3}]$ and $K_2=\QQ[\sqrt{-1}]$, which can be constructed by commensurable triangle groups \cite{takeuchi1977commensurability}. 
\end{rmk}
	
\subsection{Trace Fields}	
We collect some general facts about trace fields. For example, see \cite[\S12]{deligne1986monodromy}. Let $G$ be an adjoint connected semi-simple algebraic group $G$ over field $k$ with zero characteristic. Let $\Gamma\subset G(k)$ be a Zariski dense subgroup. The induced adjoint operation of $G$ on the Lie algebra of $G$ is denoted by $\Ad$. In particular, when $\Gamma$ is an arithmetic subgroup in an adjoint real connected algebraic group $G(\RR)$, then it is Zariski dense.
	
	\begin{defn}
The trace field $F$ of $\Gamma$ is defined to be the field generated by traces of all elements of $\Gamma$ under adjoint operation:
\begin{equation*}
F=\QQ[\Tr \,\Ad \,\Gamma].
\end{equation*}
This is a finite extension field of $\QQ$ when $\Gamma$ is an arithmetic lattice.
	\end{defn}
Regarding commensurability classes, we have
\begin{prop}
	The trace field is a commensurability invariant for Zariski dense subgroups $\Gamma\subset G(k)$.
\end{prop}

The defining field for $G$ can be descent to $F$.
\begin{prop}
	\label{prop: F-structure remain the same}
	Let $F\subset F^\prime\subset k$ be any field extension. The algebraic group $G$ has a unique $F^\prime$-structure such that $\Gamma\subset G(F^\prime)$. This $F^\prime$-structure remains the same when $\Gamma$ changes to a finite-index subgroup. 
\end{prop}

\subsection{Deligne--Mostow Monodromy Groups}
Now we specialize to the monodromy group $\Gamma_\mu\subset \PU(1,n)$ in Deligne--Mostow theory. Note that these $\Gamma_\mu$ are Zariski dense subgroup of $\PU(1,n)$. We quote the calculation of trace fields from Deligne--Mostow \cite[Lemma 12.5]{deligne1986monodromy}. Denote by $d$ the least common denominator of components $\mu_i$ in tuple $\mu$.
\begin{prop}
	When $n\geq 2$, the trace field $F$ of $\Gamma_\mu$ is equal to $\QQ[\zeta_d]\cap \RR$. When $n=1$, the trace field is a subfield of $\QQ[\zeta_d]\cap \RR$.
\end{prop}

Let $F=\QQ[\zeta_d]\cap \RR$. There is another construction of $F$-structure on $\PU(1, n)$ as follows. The vector space $V_\mu=H^0(\PP^1, \LL_\mu)$ is defined over $\QQ[\zeta_d]$, and the Hermitian form $h_\mu$ on $V_\mu$ can also be defined over $\QQ[\zeta_d]$. The group of $\QQ[\zeta_d]$-linear automorphisms of $V_\mu$ preserving $h_\mu$ defines an $F$-algebraic group $\U(V_\mu, h_\mu)$ and its projectivization $\PU(V_\mu, h_\mu)$. This gives an $F$-structure on $\PU(1,n)$. On the other hand, the monodromy group $\Gamma_\mu\subset \PU(V_\mu, h_\mu)(F)$. Thus this induces the same $F$-structure on $\PU(1,n)$. The previous results specialize to a commensurability invariant for Deligne--Mostow monodromy groups.

\begin{prop}
\label{prop: Gamma to F-structures}
Let $n$ be a positive integer. Assume $\Gamma_\mu$ and $\Gamma_\nu$ are commensurable Deligne--Mostow monodromy groups in $\PU(1,n)$ arising from tuples $\mu$ and $\nu$. Then they have the same adjoint trace field $F$ and the $F$-forms induced by $\Gamma_\mu$ and $\Gamma_\nu$ on $\PU(1,n)$ are isomorphic. Let the common denominator of $\mu$ (respectively $\nu$) be $d$ (respectively $d^\prime$).
\begin{enumerate}
    \item When $n\geq 2$, then $\QQ[\zeta_d]=\QQ[\zeta_{d^\prime}]$, and $(V_\mu, h_\mu)$ and $(V_\nu, h_\nu)$ are conformal.
    \item When $n=1$ and $\QQ[\zeta_d]=\QQ[\zeta_{d^\prime}]$, then $(V_\mu, h_\mu)$ and $(V_\nu, h_\nu)$ are conformal.
\end{enumerate}
\end{prop}

\begin{proof}
	Suppose $\Gamma_\mu$ is commensurable to $\Gamma_\nu$. Then there exists $g\in \PU(1,n)$ such that $\Gamma=\Gamma_\mu\cap g^{-1}\Gamma_\nu g$ has finite index in $\Gamma_\mu$ and $g^{-1}\Gamma_\nu g$. By Proposition \ref{prop: F-structure remain the same}, the trace fields of $\Gamma_\mu$ and $\Gamma_\nu$ are equal to the trace field of $\Gamma$, and the $F$-structures induced by $\Gamma_\mu$ and $\Gamma_\nu$ are isomorphic.

 When $n\geq 2$, the adjoint trace field $F$ of $\Gamma_\mu$ is the same as real field of the cyclotomic field for $\mu$. By Proposition $\ref{prop: the same CM field}$, we have $\QQ[\zeta_d]=\QQ[\zeta_{d^\prime}]$, and $(V_\mu, h_\mu)$ and $(V_\nu, h_\nu)$ are conformal. When $n=1$, the adjoint trace field is a subfield of $F^\prime=\QQ[\zeta_d]\cap \RR$, then the same $F$-forms induces the same $F^\prime$-forms, and hence $(V_\mu, h_\mu)$ and $(V_\nu, h_\nu)$ are conformal if they are defined over the same CM field.
 \end{proof}

On the other hand, the converse is true for arithmetic groups. 

\begin{prop}
\label{prop: arithmetic commen by F-forms}
    Let $n$ be a positive integer and $\Gamma_\mu, \Gamma_\nu$ are arithmetic Deligne--Mostow lattices in  $\PU(1,n)$. Assume that they have the same adjoint trace field $F$ and induce the same $F$-form on $\PU(1,n)$, then $\Gamma_\mu$ and $\Gamma_\nu$ are commensurable in $\PU(1,n)$. In terms of conformal classes, we have the following.
    
    Let the common denominator of $\mu$ (respectively $\nu$) be $d$ (respectively $d^\prime$) and assume $\QQ[\zeta_d]=\QQ[\zeta_{d^\prime}]=K$. 
    \begin{enumerate}
        \item When $n\geq 2$ and $(V_\mu, h_\mu)$ is conformal to $(V_\nu, h_\nu)$, then $\Gamma_\mu$ and $\Gamma_\nu$ are commensurable.
        \item When $n=1$, $F=K\cap \RR$, and $(V_\mu, h_\mu)$ is conformal to $(V_\nu, h_\nu)$, then $\Gamma_\mu$ and $\Gamma_\nu$ are commensurable.
    \end{enumerate}
\end{prop}
 \begin{proof}
	Suppose $G_1, G_2$ are the $F$-forms defined by $\Gamma_\mu, \Gamma_\nu$ respectively and $\delta\colon G_1\cong G_2$ an isomorphism as $F$-algebraic groups. Since $\Gamma_\mu$ and $\Gamma_\nu$ are arithmetic subgroups under the $F$-structures, so we have that $\delta(\Gamma_\mu)$ and $\Gamma_\nu$ share a common finite-index subgroup. Denote by $G=\PU(1,n)$ a fixed $\RR$-algebraic group. We have isomorphisms $\delta_i\colon G_i(\RR)\to G$. Now we need to prove $\delta_1(\Gamma_\mu)$ and $\delta_2(\Gamma_\nu)$ share a finite-index subgroup after conjugation in $G$. If $\delta_2\delta_\RR\delta_1^{-1}\colon G\to G$ is an inner automorphism of $G$ induced by $g\in G$, then $g\delta_1(\Gamma_\mu) g^{-1}$ and $\Gamma_\nu$ share a common finite-index subgroup. Assume $\delta_2\delta_\RR\delta_1^{-1}$ is an outer automorphism. The group $\Aut(G)$ fits in the following exact sequence by isomorphism theorem of semisimple algebraic groups.
	\begin{equation*}
		1\to G(\RR)\to \Aut(G(\RR)) \to \Aut(A_n)\to 1,
	\end{equation*}
where $A_n$ represents the Dynkin diagram of type $A$ with $n$ nodes, and $\Aut(A_n)$ is a cyclic group of order two for $n\geq 2$. There is a similar exact sequence for $G_1(F)$. When $n\geq 2$, by Proposition \ref{prop: auto of PU}, there exists an involution $\tau$ in $\Aut(G_1)(F)$ which maps to the generator of $\Aut(A_n)$. And subgroups $\tau(\Gamma_1)$ are both arithmetic subgroup in $G_1(F)$, so $\Gamma_1\cap\tau(\Gamma_1)$ has finite index in $\Gamma_1$ and $\tau(\Gamma_1)$. Then $\delta_2(\delta\tau)_\RR{\delta_1}^{-1}$ is an inner automorphism of $\PU(1,n)$. So the groups $\delta_1(\Gamma_\mu)$ and $\delta_2(\Gamma_\nu)$ are commensurable in $\PU(1,n)$. When $n=1$, $\Aut(G_1)(\RR)$ has no outer automorphisms, so $\delta_1(\Gamma_\mu)$ and $\delta_2(\Gamma_\nu)$ share a finite-index subgroup after conjugation in $G$.
\end{proof}

Combining the previous discussion on $F$-forms on $\PU(1,n)$ and conformal classes of Hermitian forms, we conclude that the conformal class of $(V_\mu, h_\mu)$ is a commensurability invariant of Deligne--Mostow monodromy groups, and a criterion for commensurability of arithmetic ones.

\begin{thm}
\label{theorem: hermitian form}
Take two Deligne--Mostow monodromy groups $\Gamma_\mu$ and $\Gamma_\nu$ in $\PU(1,n)$. Let the common denominator of $\mu$ (respectively $\nu$) be $d$ (respectively $d^\prime$), $K=\QQ[\zeta_d]$, $F=K\cap \RR$, $K'=\QQ[\zeta_{d'}]$ and $F'=K'\cap \RR$. Consider the following statements.
    \begin{enumerate}[(a)]
    \item 
    \label{state: commensurable}
    $\Gamma_\mu$ and $\Gamma_\nu$ are commensurable.
    \item 
    \label{state: the same trace fields} $K=K'$ \item \label{state: the same det}
$K=K'$ and $\det h_\mu=\det h_\nu\in F^\times/N_{K/F} K^\times$. 
        \item 
        \label{state: sign}
$F=F'$ and $\sgn_\sigma(h_\mu)=\pm \sgn_\sigma(h_\nu)$ for all real embeddings $\sigma$ of $F$.
    \end{enumerate}

Then 
\begin{enumerate}[(1)]
\item
\label{theorem: hermitian form n dayu 2}
When $n\geq 2$, the statement \eqref{state: commensurable} implies the statements \eqref{state: the same trace fields} and \eqref{state: sign}. \item When $n\geq 3$ is odd, then the statement \eqref{state: commensurable} implies the statement \eqref{state: the same det} and \eqref{state: sign}. 
\item When $n=1$, then the statements \eqref{state: commensurable} and \eqref{state: the same trace fields} together imply statements \eqref{state: the same det} and \eqref{state: sign}.
    \end{enumerate}
Furthermore, assume that $\Gamma_\mu$ and $\Gamma_\nu$ are both arithmetic. Then 
\begin{enumerate}[(1)]
\setcounter{enumi}{3}
\item \label{theorem: n even b implies a}
When $n$ is even, the statement \eqref{state: the same trace fields} implies the statement \eqref{state: commensurable}. 
\item When $n\geq 3$ is odd, the statement \eqref{state: the same det} implies the statement \eqref{state: commensurable}. 
\item When $n=1$ and assume the adjoint trace fields of $\Gamma_\mu$ (respectively $\Gamma_\nu$) is $F$  (respectively $F'$), then the statement \eqref{state: the same det}  implies the statement \eqref{state: commensurable}.
\end{enumerate}
\end{thm} 

\begin{proof}
First we prove the first three statements. Assume $\Gamma_\mu$ and $\Gamma_\nu$ are commensurable. If $n\geq 2$, then Proposition \ref{prop: Gamma to F-structures} implies that $K=K^\prime$, and $(V_\mu, h_\mu)$ and $(V_\nu, h_\nu)$ are $F$-conformal. If $n=1$ and $K=K^\prime$, then $(V_\mu, h_\mu)$ and $(V_\nu, h_\nu)$ are also 
$F$-conformal by the second part of Proposition \ref{prop: Gamma to F-structures}. Then the first three conclusions follows from the explicit criterion of confomality in Proposition \ref{prop: conformal classes of hermitian forms}.

Next, we prove the last three statements. Assume $\Gamma_\mu$ and $\Gamma_\nu$ are arithmetic subgroups of $\PU(1,n)$ and $K=K^\prime$. Furthermore, we also assume that the adjoint trace fields of $\Gamma_\mu$ and $\Gamma_\nu$ are $F$ when $n=1$. From the arithmeticity criterion of Deligne--Mostow monodromy groups, we have $\sgn_\sigma(h_\mu)$ is $(1,n)$ for the tautological embedding $\sigma=id$. For all the other embeddings $\sigma$, the signature $\sgn_\sigma(h_\mu)$ is definite. The same holds for $\sgn_\sigma(h_\nu)$. So we have $\sgn_\sigma(h_\mu)=\pm \sgn_\sigma(h_\nu)$ for all real embeddings $\sigma$ of $F$. Then each condition in the last three statements implies that $(V_\mu, h_\mu)$ and $(V_\nu, h_\nu)$ are $F$-conformal by Proposition \ref{prop: conformal classes of hermitian forms}. So we have commensurablity between $\Gamma_\mu$ and $\Gamma_\nu$ by Proposition \ref{prop: Gamma to F-structures} in each case. 
\end{proof}

As an application of Theorem \ref{theorem: hermitian form}, we provide an alternative proof for the following result of Kappes--M\"oller \cite{kappes2016lyapunov} and McMullen \cite{McMullen2017Gauss-Bonnet}. 
\begin{cor}
\label{cor: distinguishing nonarithemetic DM by signature}
Let $\mu ={1\over 20}(5,5,5,11,14)$ or ${1\over 20}(6,6,9,9,10)$, $\nu ={1\over 20}(6,6,6,9,13)$. Then $\Gamma_\mu$ and $\Gamma_\nu$ are not commensurable. 
\end{cor}

\begin{proof}
  Under the real embedding $\sigma\colon \QQ[\zeta_{20}]\cap \RR\to \RR$ induced by $\zeta_{20}\mapsto \zeta_{20}^3$, the signature of $(V_\mu, h_\mu)$ is $\sgn_\sigma(h_\mu)=(2,1)$, while $\sgn_\sigma(h_\nu)=(3,0)$. Then Theorem \ref{theorem: hermitian form} \eqref{theorem: hermitian form n dayu 2} implies that $\Gamma_\mu$ and $\Gamma_\nu$ are not commensurable.
\end{proof}
The commensurability invariant used in Corollary \ref{cor: distinguishing nonarithemetic DM by signature} is the signature of Hermitian form. It was already found by Deraux--Parker--Paupert \cite[\S6.2]{deraux2021new} and called signature spectrum.

\section{Explicit Calculation of Hermitian Forms: Degeneration Method}
\label{section: degeneration method}
In this section we develop the method to calculate the determinant of Hermitian forms defined in Deligne--Mostow theory. The main idea is to study how the Hermitian forms change when two weights collide in Deligne--Mostow theory. 

Deligne--Mostow monodromy groups are also related to Burau representations and Gassner representations of braid groups specialized at roots of unity; see \cite{venkataramana2014image, venkataramana2014monodromy}. It is known that the Gassner representation preserves a skew Hermitian form; see \cite[Theorem 3.3]{long1989linear}. The matrix of this Hermitian form under suitable basis was calculated in \cite[\S4.1]{venkataramana2014monodromy}. A geometric way to obtain this skew Hermitian matrix at roots of unity is given in Proposition \ref{prop: intersection form}.

\subsection{Geometric Degeneration}
Now we study how the Hermitian form $V_\mu$ changes when two points $x_k$ and $x_j$ collide. Let $\mu^\prime$ be the weight vector consisting of components of $\mu_k+\mu_j$ and $\mu_i, \, i\neq j,k$. 

\begin{prop}
Assume $\mu_k+\mu_j\not\in \ZZ$ for a pair $k\neq j$. Then $$(V_\mu, h_\mu)=(V_{\mu^\prime}, h_{\mu^\prime})\oplus \langle\gamma_{kj}\rangle,$$
where $\langle\gamma_{kj}\rangle$ is a one-dimensional Hermitian space only depending on parameters $\mu_j$ and $\mu_k$.
\end{prop}

We will give two proofs of the degeneration proposition. One is based on Clemens-Schmid sequence, the other is based on explicit calculation of the Hermitian form. 
\begin{proof}   
The proof follows from an argument similar to \cite[Proposition 4.1]{yu2024calabi} in the case of curves. 
Assume $k=1, j=2$, and $x_3, \cdots, x_{n+3}$ are distinct points on $\CC$ with $|x_i|>1$. Let $\Delta=\{t\in \CC\mid |t|<1\}$ be the unit disc on the $t$-plane. Consider one parameter family of curves $f\colon \mathcal{C}\to \Delta$ formed by normalization of 
$$
y^d=x^{a_1}(x-t)^{a_2}\prod\limits_{i=3}^{n+3}(x-x_i)^{a_i}, \, |t|<1.
$$
The fiber of $f$ over any $t$ is denoted by $C_t$. The morphism $f|_{\widetilde{C}-C_0}\to \Delta-{0}$ is smooth, and the central fiber $C_0$
has at most isolated singularity. Then surface $\mathcal{C}$ is a cyclic cover of $\PP^1\times \Delta$ branched along the normal crossing divisor defined by $x(x-t)\prod\limits_{i=3}^{n+3}(x-x_i)$. So, the total space $\mathcal{C}$ has quotient singularities. The intersection complex on $\mathcal{C}$ is the constant sheaf. So by the argument in \cite[Theorem 5]{kerr2021hodge} (or see \cite[Sequence 0.2]{kerr2021hodge}), we have the following isomorphism
\[
 H^1_{\overline{\chi}}(C_0)\to H^1_{\overline{\chi}}(C_t)^T,
\]
where $T$ is the monodromy operator induced by the action of $\pi_1(\Delta-\{0\})$. On the other hand, since $d\nmid a_1+a_2$, the normalization $\widetilde{C_0}\to C_0$ is an isomorphism away from the ramification points of the cyclic cover $C_0\to \PP^1$, and the induced map $H^1_{\overline{\chi}}(C_0)\to H^1_{\overline{\chi}}(\widetilde{C_0})$ is an isomorphism. So we have $(V_\mu, h_\mu)^T\cong (V_{\mu^\prime}, h_{\mu^\prime})$. On the other hand, we have an identification of $H^1_{\overline{\chi}}(C_t)$ with $H^1(\PP^1, \LL_\mu)$ in Proposition \ref{prop: iso from C to P^1}. By \cite[Proposition 9.2]{deligne1986monodromy}, the action of $T$ on $V_\mu$ is semisimple, and the $e^{2\pi\sqrt{-1}(\mu_1+\mu_2)}$-eigenspace is generated by a cocycle represented by a path connecting $x_1$ and $x_2$. (A similar argument as \cite[Proposition 4.3]{yu2024calabi} implies that $T$ has finite order since the central fiber has ball-type Hodge structure.) The self intersection of this cycle only depends on the local system on its neighborhood. So it only depends on $\mu_1, \mu_2$; see Proposition \ref{prop: intersection form} for the explicit intersection number. 
\end{proof}

\subsection{Explicit Calculation of Hermitian Forms}
We also give a more explicit proof in terms of local systems on punctured projective line.

In \S\ref{subsection: locally finite homology} we define $\gamma_i$ to be a path in $\PP^1$ connecting $x_i$ to $x_{i+1}$. Take a point $B\in \PP^1-A$ and take $n+1$ paths $\beta_i$ connecting $B$ to points on $\gamma_1, \cdots, \gamma_{n+1}$ respectively. Then a nonzero value of $\LL$ at $p$ induces a section $e$ of $\LL$ over $\beta_i, \gamma_i$ simultaneously. The locally finite homology is generated by $\gamma_i \cdot e$ with $1\le i\le n+1$. Then the Hermitian form can be calculated explicitly in Proposition \ref{prop: intersection form}.

Firstly, we need the following geometric description of the inverse map of isomorphism $H_1(\PP^1-A)\to H^1_{lf}(\PP^1-A)$ from \cite[Proposition 2.6.1]{deligne1986monodromy} and its proof,
\begin{prop}
\label{proposition: geometric description of isomorphism of homologies}
The isomorphism $H_1^{lf}(\PP^1-A, \LL)\to H_1(\PP^1-A, \LL)$ can be phrased as follows. For a path $\gamma$ from $x_i$ to $x_j$ with $e\in \Gamma((0,1), \gamma^*\LL)$. Take $\epsilon>0$ small enough. Let $\widetilde{\gamma}$ be the part of $\gamma$ from $\gamma(\epsilon)$ to $\gamma(1-\epsilon)$. Let $\theta_1, \theta_2$ be two circles starting at $\gamma(\epsilon), \gamma(1-\epsilon)$ and going counterclockwise around $x_i, x_j$  once respectively. Let $e_1$ be a section of $\LL$ over $\gamma_1$ such that $e_1(1)=e(\epsilon)$. Let $e_2$ be a section of $\LL$ over $\gamma_2$ such that $e_2(0)=e(1-\epsilon)$. Then the isomorphism $H_1^{lf}(\PP^1-A, \LL)\to H_1(\PP^1-A, \LL)$ replaces $\gamma\cdot e$ by 
\begin{equation*}
{\alpha_i\over\alpha_i-1}\theta_1\cdot e_1+\widetilde{\gamma}\cdot e+{1\over 1-\alpha_j} \theta_2\cdot e_2.
\end{equation*}
See Figure \eqref{figure: gamma}.
\end{prop}
\begin{proof}
We check that ${\alpha_i\over\alpha_i-1}\gamma_1\cdot e_1+\widetilde{\gamma}\cdot e+{1\over 1-\alpha_j} \theta_2\cdot e_2$ is closed. Since the monodromy of $\LL$ aroud $x_i, x_j$ is by multiplication with $\alpha_i, \alpha_j$, we have $\partial(\theta_1\cdot e_1)=\gamma(\epsilon)\cdot e-\alpha_i^{-1} (\gamma(\epsilon)\cdot e)$ and $\partial(\theta_2\cdot e_2)=\alpha_j (\gamma(1-\epsilon)\cdot e)-\gamma(1-\epsilon)\cdot e$.
We have also $\partial(\widetilde{\gamma}\cdot e)=\gamma(1-\epsilon)\cdot e-\gamma(\epsilon)\cdot e$. These imply $\partial({\alpha_i\over\alpha_i-1}\theta_1\cdot e_1+\widetilde{\gamma}\cdot e+{1\over 1-\alpha_j} \theta_2\cdot e_2)=0$.
\end{proof}
\begin{figure}[htp]
\centering
\includegraphics[width=9cm]{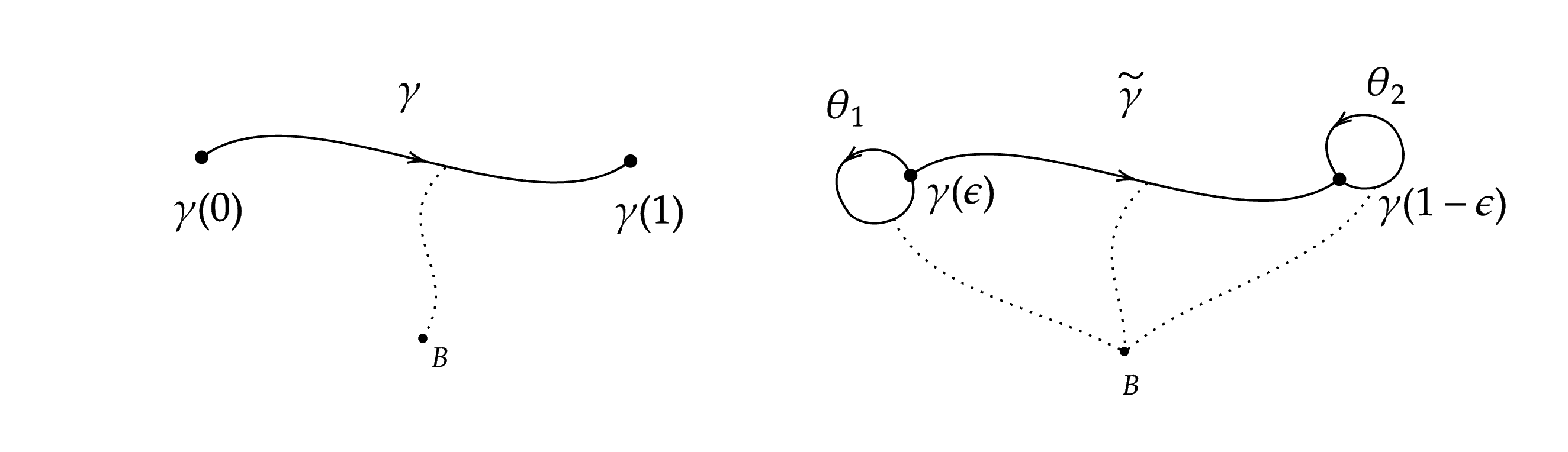}
\caption{$\gamma$ and $\widetilde{\gamma}$}
\label{figure: gamma}
\end{figure}

Let $\omega_1=\gamma_1\cdot e, \cdots, \omega_{n+2}=\gamma_{n+2}\cdot e$. 

\begin{prop}
The only relation among $\omega_1, \cdots, \omega_{n+2}$ is
\begin{equation*}
(1-\alpha_1^{-1})\omega_1+(1-\alpha_1^{-1}\alpha_2^{-1})\omega_2+\cdots+(1-\alpha_1^{-1}\cdots\alpha_{n+2}^{-1})\omega_{n+2}=0
\end{equation*}    
\end{prop}
\begin{proof}
This directly follows from Figure \ref{figure: relation of cycles}.
\end{proof}

\begin{figure}[htp]
		\centering
	\includegraphics[width=9cm]{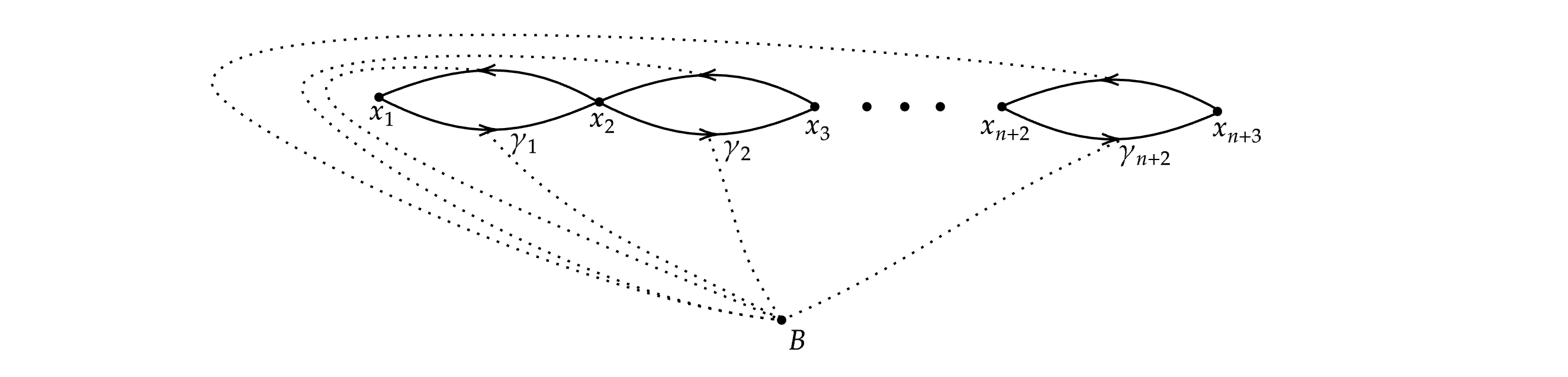}
		\caption{Relation of Cycles}
		\label{figure: relation of cycles}
	\end{figure}

\begin{prop}
\label{prop: intersection form}
The space $H^{lf}_1(\PP^1-A,\LL)$ is generated by $\gamma_i$ ($1\le i\le n+1$) such that 
\begin{enumerate}
\item $(\gamma_i\cdot e, \gamma_i\cdot e)=-1+{1\over 1-\alpha_i}+{1\over 1- \alpha_{i+1}}={1-\alpha_i \alpha_{i+1}\over (1-\alpha_i)(1-\alpha_{i+1})}$ for $1\le i\le n+1$;
\item $(\gamma_i\cdot e, \gamma_{i+1}\cdot e)=-\overline{(\gamma_{i+1}\cdot e, \gamma_i\cdot e)}={1\over \alpha_{i+1}-1}$ for $1\le i\le n+1$;
\item $(\gamma_i\cdot e, \gamma_j\cdot e)=0$ for $|i-j|\ge 2$.
\end{enumerate}
\end{prop}

\begin{proof}
For (1), see Figure \ref{figure: self intersection}, we need to calculate the intersection value at $A_1$ and $A_2$. See Proposition \ref{proposition: geometric description of isomorphism of homologies} for the multiplicities of the two circles. The intersection value at $A_1$ is $(-1)\times {\alpha_i\over \alpha_i-1}$, and the intersection value at $A_2$ is ${1\over 1-\alpha_{i+1}}$. Hence the intersection $(\gamma_i\cdot e, \gamma_i\cdot e)$ is $-1+{1\over 1-\alpha_i}+{1\over 1-\alpha_{i+1}}$.

For (2), see Figure \ref{figure: adjacent intersection}, we need to calculate the intersection value at $A$, which is $(-1)\times {1\over 1-\alpha_{i+1}}={1\over \alpha_{i+1}-1}$. The equality (3) is obvious.
\end{proof}

\begin{figure}[htp]
		\centering
	\includegraphics[width=5cm]{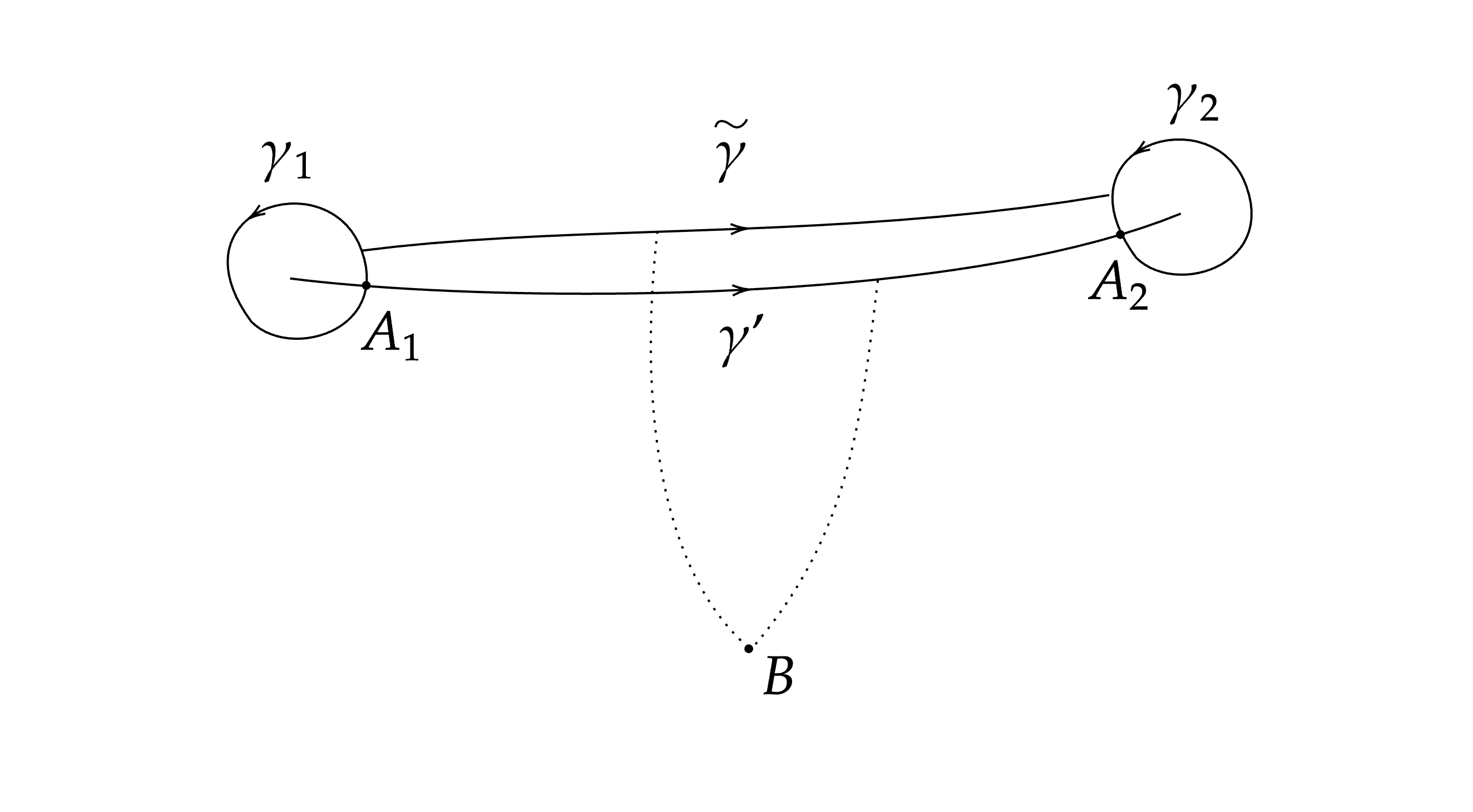}
		\caption{Self Intersection}
		\label{figure: self intersection}
	\end{figure}
 
\begin{figure}[htp]
\centering
\includegraphics[width=4cm]{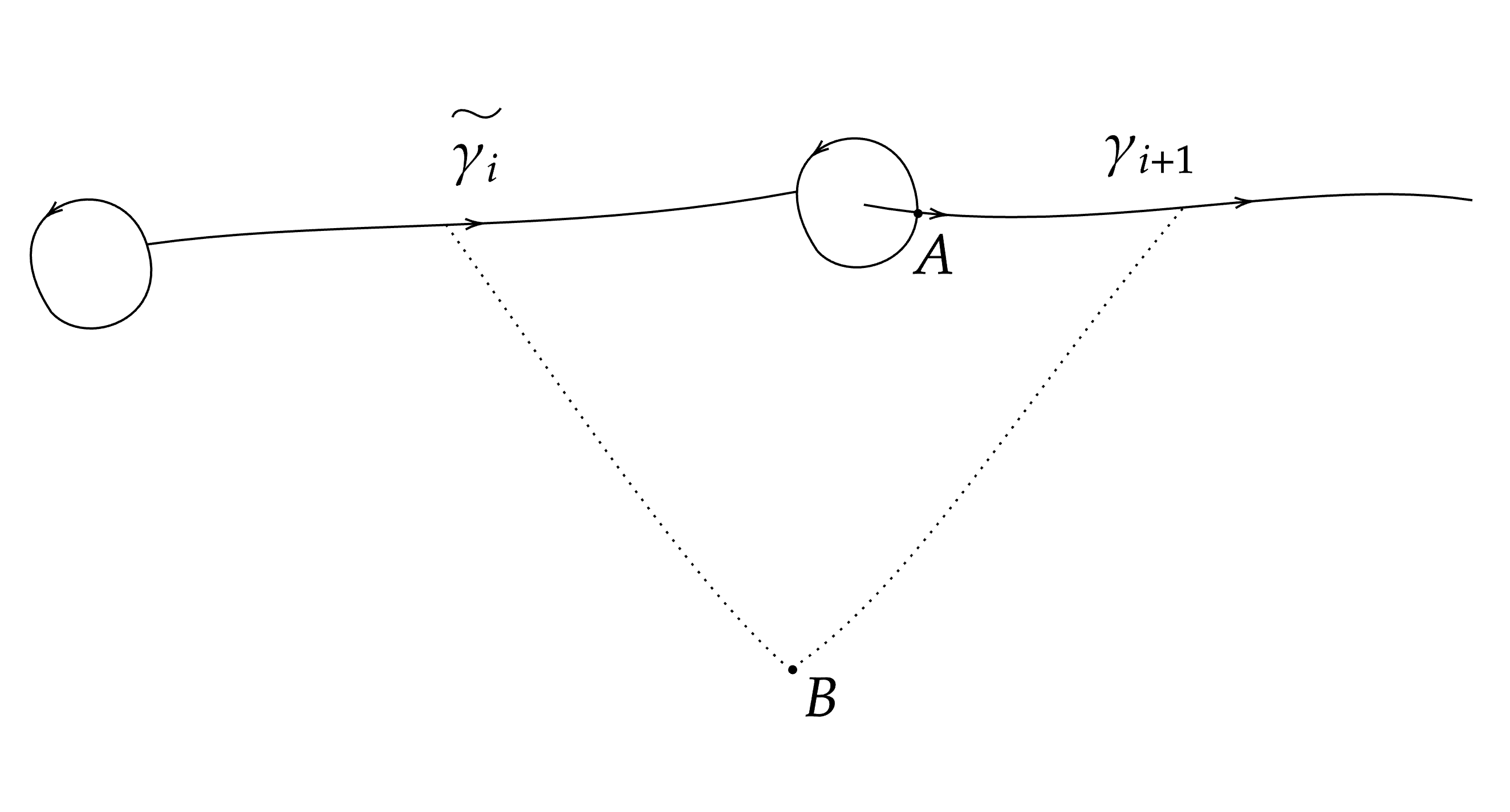}
\caption{Adjacent Intersection}
\label{figure: adjacent intersection}
	\end{figure}

\begin{figure}[htp]
		\centering
	\includegraphics[width=9cm]{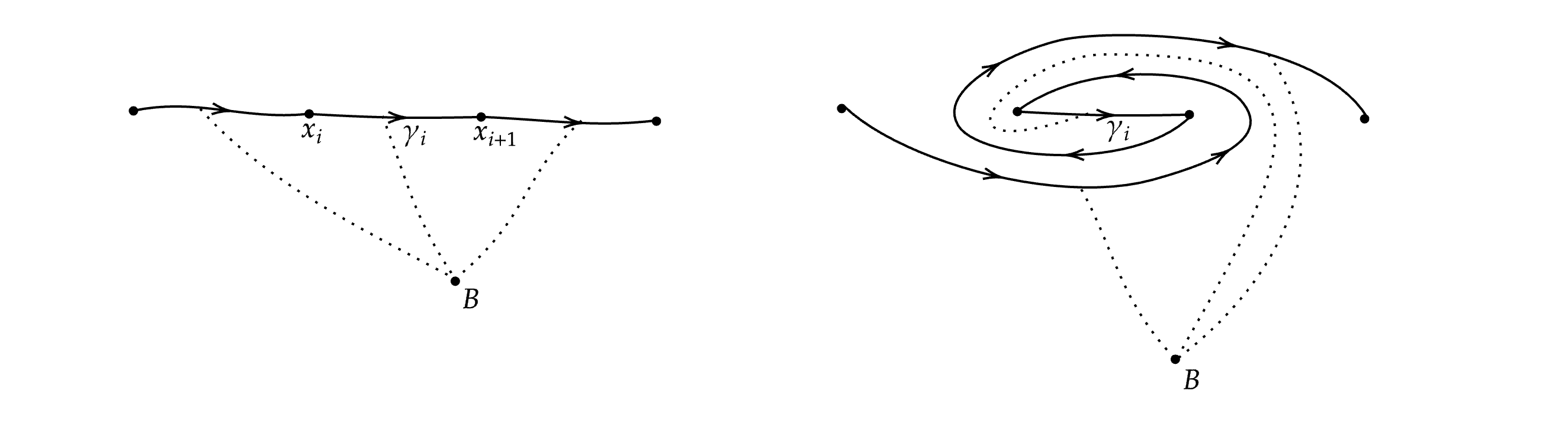}
		\caption{Monodromy Operator}
		\label{figure: monodromy operator}
	\end{figure}

The monodromy operator $T_{i,i+1}$ can be calculated from Figure \ref{figure: monodromy operator}. We have
\begin{align*}
T_{i,i+1}(\omega_{i}) &=\alpha_i\alpha_{i+1}\omega_i \\
T_{i,i+1}(\omega_{i-1}) &= \omega_{i-1}+(1-\alpha_{i+1})\omega_i \\
T_{i,i+1}(\omega_{i+1}) &= \omega_{i+1}+(\alpha_{i+1}-\alpha_i \alpha_{i+1})\omega_i.
\end{align*}

We denote by $H_{\mu}$ the skew-Hermitian lattice $H^{lf}_1(\PP^1-A, \LL)$.
\begin{prop}
\label{proposition: merge two weights}
Suppose $\mu$ degenerates to $\mu'$ with $\mu_1, \mu_2$ merging to $\mu_1+\mu_2$. Suppose moreover $\mu_1+\mu_2\notin \ZZ$. Then there is an orthogonal decomposition 
\begin{equation*}
H_\mu\cong H_{\mu'}\oplus \langle \gamma_1\rangle
\end{equation*}
as skew-Hermitian forms over $\QQ[\zeta_d]$. Here $(\gamma_1, \gamma_1)=-1+{1\over 1-\alpha_1}+{1\over 1-\alpha_2}$.
\end{prop}
\begin{proof}
By winding $x_2$ around $x_1$ counterclockwisely once, we obtain a monodromy operator $T$ on $H_\mu$ given by $T(\gamma_1)=\alpha_1 \alpha_2\gamma_1$ and
\begin{equation*}
T(\gamma_2)=\gamma_2+\alpha_2(1-\alpha_1)\gamma_1.
\end{equation*}
The operator $T$ preserves the skew-Hermitian form on $H_\mu$. If $\alpha_1\alpha_2\ne 1$, then $H_\mu=\mathrm{Ker}(T-id)\oplus\langle \gamma_1\rangle$. 

Let $\widehat{\gamma}_2={\alpha_2-\alpha_1\alpha_2\over 1-\alpha_1\alpha_2}\gamma_1+\gamma_2$, then $\mathrm{Ker}(T-id)=\langle \widehat{\gamma}_2, \gamma_3,\cdots, \gamma_{n+1}\rangle$. By straightforward calculation we have
\begin{equation*}
(\widehat{\gamma}_2, \widehat{\gamma}_2)=-1+{1\over 1-\alpha_1\alpha_2}+{1\over 1-\alpha_3} 
\end{equation*}
and 
\begin{equation*}
\langle \widehat{\gamma}_2, \gamma_3\rangle={1\over \alpha_3-1}.
\end{equation*}
Hence $H_\mu\cong H_{\mu'}\oplus \langle \gamma_1\rangle$.
\end{proof}

\begin{cor}
\label{corollary: degeneration method}
Suppose $\mu,\nu$ are two tuples of the same length, and both $\mu,\nu$ contains factors $a,b$. Replace $a,b$ by $a+b$ we get tuples $\mu',\nu'$ respectively. Then $H_\mu\cong H_\nu$ if and only if $H_{\mu'}\cong H_{\nu'}$.
\end{cor}

\begin{cor}
Suppose $\mu=(a,a,a,a,2-4a)$ and $\nu=(a,a,a,{1\over 2}-a, {3\over 2}-2a)$ for a rational number $a\in(0,{1\over 2})$, then we have conformality between Hermitian spaces:
\[
(V_\mu, h_\mu)\sim (V_\nu, h_\nu).
\]
\end{cor}

\begin{prop}
\label{proposition: determinant}
The determinant of $H_{\mu}$ is equivalent to ${1\over \prod\limits_{i=1}^{n+3}(1-\alpha_i)}$.    
\end{prop}
\begin{proof}
We induct on $n$. If $n=0$, then $\mu=(\mu_1, \mu_2, \mu_3)$ such that none of $\mu_i$ is equal to $1$. Then $\det H_{\mu}$ is ${1-\alpha_1\alpha_2\over (1-\alpha_1)(1-\alpha_2)}={1-\alpha_3^{-1}\over (1-\alpha_1)(1-\alpha_2)}\sim {1 \over (1-\alpha_1)(1-\alpha_2)(1-\alpha_3)}$.

Suppose $n\ge 1$. Assume $\alpha_1\alpha_2\ne 1$. By merging $\mu_1$ and $\mu_2$ to $\mu_1+\mu_2$ we obtain $\mu'$. Then $H_{\mu}=H_{\mu'}\oplus \langle {1-\alpha_1\alpha_2\over (1-\alpha_1)(1-\alpha_2)}\rangle$. By induction assumption, we have 
\begin{equation*}
\det H_{\mu'}={1\over (1-\alpha_1\alpha_2)\prod\limits_{i=3}^{n+3}(1-\alpha_i)}.    
\end{equation*}
The result clearly follows.
\end{proof}

\subsection{Ideal Classes of Determinants}
Next we use Propositsion \ref{proposition: determinant} to calculate the equivalence classes of the determinants for the skew-Hermitian forms $H_\mu$ in $K^\times/\mathrm{N}_{K/F}(K^\times)$. The group of principal fractional ideals of $\calO_K$ is denoted by $\calP_{K}$. We have a natural map 
\[
K^\times/\mathrm{N}_{K/F}(K^\times)\to \calP(K)/\mathrm{N}_{K/F}(\calP(K)).
\]
For Deligne--Mostow data $\mu$, we list the equivalence class of the fractional ideal generated by $\det H_\mu$ in Table \ref{table: main table}. For fixed $n$ and $\QQ[\zeta_d]$, if the ideal class for two data $\mu, \nu$ are different, then according to Theorem \ref{theorem: hermitian form} they are not commensurable. Therefore, to complete the classification shown in Table \ref{table: main table}, it suffices to check that the monodromy groups associated with the data in one block are commensurable to each other. Since the factors of $\det H_\mu$ have the form $1-\zeta_d^a$, the fractional ideals involved are primes ramified over $\QQ$ and the units are cyclotomic units.

\begin{ex}
As an illustration of the calculation, we present the details when $(n,d)=(3,12)$ here. It is known that $1-\zeta_{12}^i$ is a unit for $i=1,2,5,7,10,11$. The ideal $(1-\zeta_{12}^3)= (1-\zeta_{12}^9)$ (denoted by $\fp_2$) is a prime ideal lying over $(2)\subset \ZZ$. The ideal $(1-\zeta_{12}^4)=(1-\zeta_{12}^8)$ (denoted by $\fp_3$) is a prime ideal lying over $(3)\subset \ZZ$. The ideal $(1-\zeta_{12}^6)=(2)=(1+\sqrt{-1})(1-\sqrt{-1})$ is trivial in $\calP(K)/\mathrm{N}_{K/F}(\calP(K))$. Now by Proposition \ref{proposition: determinant} it is straightforward to calculate the ideal classes of $\det H_\mu$ for cases 54--58, as shown in Table \ref{table: main table}. As a result, case 57 is not commensurable to cases 54--56.
\end{ex}

Now we are ready to finish the proof of Theorem \ref{theorem: full classification}.
\begin{proof}[Proof of Theorem \ref{theorem: full classification}]
As mentioned in introduction, the classification of non-arithmetic Deligne--Mostow lattices follows from Theorem \ref{theorem: deligne mostow sauter}, Kappes--M\"oller \cite{kappes2016lyapunov} and McMullen \cite{McMullen2017Gauss-Bonnet}. 

Next we consider arithmetic Deligne--Mostow lattices. By \eqref{theorem: hermitian form n dayu 2} and \eqref{theorem: n even b implies a} of Theorem \ref{theorem: hermitian form}, the classification is completed for cases when $n$ is even. We then only need to deal with the cases $n=5$ or $3$.

In the last column of Table \ref{table: main table}, we list the ideal classes of $\det H_{\mu}$. When $n$ is odd, this class is a commensurability invariant. To finish the proof, we need to verify commensurability relations for Deligne--Mostow data in each block when $n=3,5$.  

For $(n,d)=(5,6)$, we need to check that cases 8--11 are commensurable to each other. This is straightforward. For example, take $\mu=({1\over 6}, {1\over 6},{1\over 6},{1\over 6}, {1\over 6}, {1\over 6}, {1\over 2}, {1\over 2})$ and $\nu=({1\over 6},{1\over 6},{1\over 6},{1\over 6}, {1\over 3}, {1\over 3}, {1\over 3}, {1\over 3})$, by Proposition \ref{proposition: determinant},
\[
{\det H_\mu\over \det H_\nu}={(1-\zeta_6^2)^4 \over 4(1-\zeta_6)^2} = {1\over 4} (1+\zeta_6)^2 (1+\zeta_6^{-1})^2 \in \mathrm{N}_{K/F}(K^\times).
\]
Similar calculation verifies commensurablity relations for cases 19--23, cases 24--27 and cases 39--40.

For $(n,d)=(3,12)$, we need to show cases 54--56 are commensurable. By Theorem \ref{theorem: dimension 3} with $a={1\over 4}$, cases 55 and 56 are commensurable. Next we show cases 54 and 56 are commensurable. Let $\mu=({1\over 12}, {1\over 4}, {5\over 12}, {5\over 12},{5\over 12},{5\over 12})$ and $\nu=({1\over 4}, {1\over 4}, {1\over 4}, {5\over 12}, {5\over 12},{5\over 12})$, then 
\[
{\det H_\mu \over \det H_\nu}
={(1-\sqrt{-1})^2\over (1-\zeta_{12})(1-\zeta_{12}^5)}
=2=(1+\sqrt{-1})(1-\sqrt{-1})\in \mathrm{N}_{K/F}(K^\times).
\]
We complete the proof of Theorem \ref{theorem: full classification}.
\end{proof}

\section{Commensurability Invariant: Boundary Divisors of Toroidal Compactifications}
\label{section: toroidal boundary divisors}
In this section, we describe another commensurability invariant related to boundary divisors in toroidal compactifications in both arithmetic and non-arithmetic ball quotients. In fact, when $n=3$, this is essentially the invariant used by Deraux \cite{deraux2020new} to distinguish the commensurability class of the unique Deligne--Mostow non-arithmetic lattice in $\PU(1,3)$ and a non-arithmetic Couwenberg--Heckman--Looijenga lattice \cite{couwenberg2005geometric}.

We recall the construction of Satake--Baily--Borel and toroidal compactification of ball quotients. In arithmetic quotients of general Hermitian symmetric domains, see Satake \cite{satake1960compactifications} and Baily--Borel \cite{baily1966compactification} for Satake--Baily--Borel compactifications, see \cite{Ash_Mumford_Rapoport_Tai_2010} for toroidal compactifications. Especially, see \cite{looijenga2003compactificationsball} for arithmetic ball quotients. In non-arithmetic cases, see Mok \cite{mok2011projective} for both Satake--Baily--Borel and toroidal compactifications for ball quotients. 

Let $W$ be a complex vector space with Hermitian form $h$ of signature $(1, n)$. The $n$-dimensional complex hyperbolic ball $\BB\subset \PP(W)$ consisting of positive lines in $W$. Let $\Gamma$ be a torsion-free discrete lattice in $\SU(W, h)$. Hence the quotient $X=\Gamma\bs \BB$ is a complex  manifold with K\"ahler metric finite volume. The work of Siu--Yau \cite{siu1982compactification} gives a compactification of $X$ by adding finitely many cusps $b_i$, $1\leq i\leq m$. This is the Satake--Baily--Borel compactification of $X$ when $\Gamma$ is arithmetic. Those cusps correspond to finitely many $\Gamma$-orbits of isotropic vectors $b_i\in W$. Denote by $b=b_i$. Let $v_0\in W-b^\perp$ be any vector, then the space $\PP(W/b)-\PP(b^\perp/b)$ is identified with affine space $v_0+b^\perp/b$ by intersecting lines in $W/b-b^\perp/b$ with $v_0+b^\perp/b$. The stabilizer $\Gamma_b$ of $b$ in $\Gamma$ acts on the $(n-1)$-dimensional complex affine space $\PP(W/b)-\PP(b^\perp/b)$ as affine translations. In fact, those $b\in \partial \BB$ are characterized by the property that $\Gamma_b$ is nontrivial. Then the quotient 
\[
D_b\coloneqq\Gamma_b\bs (\PP(W/b)-\PP(b^\perp/b))
\]
is an abelian variety of dimensional $n-1$. The toroidal compactification of $X$ adds those abelian varieties as boundary divisors to $X$. When $\Gamma$ changes to a finite index subgroup $\Gamma^\prime$, we denote by $X^\prime$ the quotient $\Gamma^\prime\bs \BB$. Each cusp of $X^\prime$ corresponds to a cusp of $X$, and $\Gamma^\prime_b$ is a finite index subgroup of $\Gamma_b$. So, the boundary divisor $D_b^\prime$ is isogeneous to $D_b$. When $\Gamma$ has torsion elements, we pass to a torsion-free subgroup with finite index to obtain the Satake--Baily--Borel and toroidal compactifications, and quotient by $\Gamma/\Gamma^\prime$. When $\Gamma$ is changed to a finite index subgroup, the stabilizer also changes to a finite index subgroup, so we have the following commensurability invariant for discrete lattices in $\PU(1, n)$.

\begin{prop}
\label{prop: toroidal boundary divisors are isogeneous}
    Let $\Gamma_1$ and $\Gamma_2$ be two commensurable discrete lattices in $\PU(1,n)$. Denote by $\calA_i$ the set of isogeny classes of the boundary abelian varieties for the toroidal compactifications of ball quotient $\Gamma_i\bs \BB^n$. Then $\calA_1=\calA_2$.
\end{prop}

Let $\Gamma_\nu$ be a Deligne--Mostow lattice associated with tuple $\mu$. Deligne--Mostow \cite[Section 4]{deligne1986monodromy} describes the compactification of the ball quotient $\Gamma_\mu \bs \BB(V_\mu)$ by adding cusps as follows. Let $\mu=(\mu_1\cdots \mu_{n+3})$. Then each cusp corresponds to a partition of index set $\{1,2,\cdots, n+3\}$ into two parts $S_1\sqcup S_2$, such that the weight of each part is 
$$\sum_{i\in S_1}\mu_i=\sum_{i\in S_2}\mu_i=1.$$
The cusps defined this way are semistable points in Mumford's geometric invariant theory and the corresponding compactification of $\Gamma_\mu\bs \BB(V_\mu)$ is a GIT compactification. In \cite[Corollary 7.3]{deligne1986monodromy}, it is proved that under period map, these semistable points are the cusps in Siu--Yau's metric compactification. 

Denote by $b$ a cusp corresponding to a fixed partition $S_1\sqcup S_2$. Next we describe certain elements in the stabilizer group $\Gamma_b$ of $b$ in $\Gamma_\mu$. First recall the monodromy representation inducing $\Gamma_\mu$. For each pair of indices $1\leq i\neq j\leq n+3$, define the main diagonal in the moduli spaces
\[\Delta_{ij}=\{(x_1, \cdots, x_{n+3})\in(\PP^1)^{n+3}\mid x_i=x_j \}.\]
The configuration space of ordered $n+3$-points on $\PP^1$ is defined by 
$$
Q=(\PP^1)^{n+3}-\bigcup_{1\leq i\neq j\leq n+3}\Delta_{ij}.
$$
The monodromy representation is from 
\begin{equation*}
    \rho_\mu\colon \pi_1(Q)\to \PU(V_\mu, h_\mu).
\end{equation*}
More explicitly a point $A=(x_1, \cdots, x_{n+3})\in Q$ moves around one of the main diagonals $\Delta_{ij}$ in $(\PP^1)^{n+2}$, the loop induces an element $T_{ij}$ in $\Gamma_\mu$ by $\rho_\mu$. The monodromy group $\Gamma_\mu$ is generated by these $T_{ij}$. When $i, j\in S_1$ or $i, j\in S_2$, then $T_{ij} \in \Gamma_b$ under a suitable choice of representative $b\in \PP(V_\mu)$ from the $\Gamma$-orbit of cusps. When $\mu_i+\mu_j<1$, the corresponding element $T_{ij}$ is a complex reflection and not a parabolic element. When $\mu_i+\mu_j=1$, the monodromy $T_{ij}$ is a parabolic element. 

In the following, we look at a special case with $S_1={1,2}$ and $S_2={3, \cdots, n+3}$. Let $\gamma_1$ be a path connecting $x_1$ to $x_2$ in $\PP^1$ and $e$ be a section of local system $\LL$ on $\gamma_1$. Then by Proposition \ref{prop: intersection form}, the element $\gamma_1\cdot e$ is an isotropic vector in $H^1(\PP^1-A, \LL)$. Let $x_1, x_2$ be fixed, and $x_3, \cdots, x_{n+3}$ move in the configuration space and not wind around $x_1$ and $x_2$. Then the corresponding monodromy element $T$ preserves $\gamma_1\cdot e$. Next we calculate explicitly the subgroup of $\Gamma_b$ generated by $T_{ij}, i,j\in S_2$ and the corresponding isogeny of $D_b$. This is essentially the explicit calculation for the Euclidean case of Deligne--Mostow theory; see \cite[Section 13.2]{deligne1986monodromy}.

\begin{ex}
\label{Example: toroidal bdy elliptic curves}
Consider a $5$-tuple $(\mu_1, \cdots, \mu_5)$ with $\mu_1+\mu_2=1$. Let $\omega_i=\gamma_i\cdot e$, $1\le i\le 4$. We have the following relation 
\begin{equation*}
(1-\alpha_1^{-1})\omega_1+(1-\alpha_3^{-1})\omega_3+(1-\alpha_3^{-1}\alpha_4^{-1})\omega_4=0.
\end{equation*}
Thus $\omega_3=-{1-\alpha_1^{-1}\over 1-\alpha_3^{-1}}\omega_1-{1-\alpha_5\over 1-\alpha_3^{-1}}\omega_4$. Now $\omega_1\in V_{\mu}$ is isotropic. The orthogonal complement of $\omega_1$ in $V_{\mu}$ is $\langle \omega_1, \omega_2, \omega_4\rangle$. And the quotient $V_{\mu}\slash \langle \omega_1\rangle=\langle \omega_2, \omega_4\rangle$. We next calculate the action of $T_{45}$ and $T_{34}$ on $\omega_1, \omega_2, \omega_4$.

We have $T_{45}(\omega_1)=\omega_1$, $T_{45}(\omega_2)=\omega_2$ and $T_{45}(\omega_4)=\alpha_4 \alpha_5 \omega_4$. This implies that 
\begin{equation*}
T_{45}(\omega_2+\lambda\omega_4)=\omega_2+\alpha_3^{-1} \lambda \omega_4.
\end{equation*}
Here $\lambda\in\CC$ and  $\omega_2+\lambda\omega_4$ represents an element in $\PP(V_{\mu}\slash\langle\omega_1\rangle)-\PP(\omega_1^{\perp}\slash\langle \omega_1\rangle)$.

On the other hand, we have $T_{34}(\omega_1)=\omega_1$, $T_{34}(\omega_2)=\omega_2+(1-\alpha_4)\omega_3$ and $T_{34}(\omega_4)=\omega_4+(\alpha_4-\alpha_3\alpha_4)\omega_3$. This implies
\begin{equation*}
T_{34}(\omega_2+\lambda\omega_4)\equiv \omega_2+(-{(1-\alpha_4)(1-\alpha_5)\over 1-\alpha_3^{-1}}+\alpha_5^{-1}\lambda)\omega_4 \,(\mathrm{mod} \,\omega_1).
\end{equation*}
Let $f_\mu(\lambda)=-{(1-\alpha_4)(1-\alpha_5)\over 1-\alpha_3^{-1}}+\alpha_5^{-1}\lambda$ and $g_\mu(\lambda)=\alpha_3^{-1}\lambda$.

For $\mu=({5\over 12}, {7\over 12}, {3\over 12}, {3\over 12}, {6\over 12})$, we have $f_\mu(\lambda)=2\sqrt{-1}-\lambda$ and $g_\mu(\lambda)=-\sqrt{-1}\lambda$. Then $(g_\mu\circ g_\mu\circ f_\mu)(\lambda)=\lambda-2\sqrt{-1}$ and $(g_\mu\circ f_\mu\circ g)(\lambda)=\lambda+2$. So the affine translations generated by $f_\mu, g_\mu$ contains the lattice $\ZZ+\ZZ[\sqrt{-1}]$.

For $\nu=({5\over 12}, {7\over 12}, {4\over 12}, {4\over 12}, {4\over 12})$, we have $f_\nu(\lambda)=\zeta_3^2\lambda+\zeta_3-\zeta_3^2$ and $g_\nu(\lambda)=\zeta_3^2\lambda$. Then $(g_\nu\circ g_\nu\circ f_\nu)(\lambda)=\lambda+\zeta_3^2-1$ and $(g_\nu\circ f_\nu\circ f_\nu)(\lambda)=\lambda+\zeta_3^2-\zeta_3$. So the affine translations generated by $f_\nu, g_\nu$ contains the lattice $\ZZ+\ZZ[\zeta_3]$.
\end{ex}

Then we have the following corollary from Proposition \ref{prop: toroidal boundary divisors are isogeneous} distinguishing the only two non-compact Deligne--Mostow non-arithmetic lattices in $\PU(1,2)$.

\begin{cor}
\label{Corollary: nonarithemtic with denominator 12}
The non-arithmetic Deligne--Mostow lattices $\Gamma_\mu$ and $\Gamma_\nu$ associated to the tuples $\mu ={1\over 12}(3,3,5,6,7)$ and $\nu={1\over 12}(4,4,4,5,7)$ are not commensurable.
\end{cor}

\begin{proof}
In each case, the Baily--Borel compactification of the ball quotient has one cusp. The boundary divisor of the corresponding toroidal compactification is an elliptic curve. The elliptic curve $E_\mu$ (respectively $E_\nu$) at the cusp of $\Gamma_\mu\bs \BB^2$ (respectively $\Gamma_\mu\bs \BB^2$) is isogeneous to elliptic curve $\CC/\ZZ[\sqrt{-1}]$ (respectively $\CC/\ZZ[\zeta_3]$). So $E_\mu$ and $E_\nu$ are not isogeneous, and hence $\Gamma_\mu$ and $\Gamma_\nu$ are not commensurable.
\end{proof}

The invariant used in this section has essentially appeared in Deraux's work \cite{deraux2020new} distinguishing the non-arithmetic lattices in $\PU(1,3)$ constructed by Deligne--Mostow and Couwenberg--Heckman--Looijenga \cite{couwenberg2005geometric}. Deraux calculated the stabilizer $\Gamma_b$ in both cases and called them Heisenberg groups at cusps. The two toroidal boundary divisors in those two cases are isogeneous to $(\CC/\ZZ[\sqrt{-1}])^2$ and $(\CC/\ZZ[\zeta_3])^2$, which are products of elliptic curves. Then it is interesting to see that the two elliptic curves also appears as the boundary divisor for the only two non-compact Deligne--Mostow non-arithmetic lattices in $\PU(1,2)$ in Corollary \ref{Corollary: nonarithemtic with denominator 12}. The ball $\Gamma_\mu\bs \BB^2$ with $\mu={1\over 12}(3,3,5,6,7)$ naturally appears as the totally geodesic subball in the three-dimensional non-arithmetic Deligne--Mostow ball quotient. It is natural to ask whether $\Gamma_\nu\bs \BB^2$ with $\nu={1\over 12}(4,4,4,5,7)$ is a subball of Couwenberg--Heckman--Looijenga non-arithmetic ball quotient.

\begin{rmk}
The Baily--Borel compactifications and toroidal compactifications are originally constructed for arithmetic quotients of Hermitian symmetric domains. The Baily--Borel and toroidal compactifications for the non-arithmetic ball quotients in Example \ref{Example: toroidal bdy elliptic curves} are defined by the work of Siu--Yau \cite{siu1982compactification} and Mok \cite{mok2011projective}. Our calculation of Example \ref{Example: toroidal bdy elliptic curves} determines the isogeny classes of boundary divisors in the toroidal compactifications for the two non-arithmetic ball quotients. For both arithmetic and non-arithmetic Deligne--Mostow ball quotients, we include the description of boundary divisors in their toroidal compactifications as follows. Consider the cusp of $\Gamma_\mu\bs \BB(V_\mu)$ represented by $S_1\sqcup S_2$. Then each part $S_i$ determines an abelian variety of dimension $|S_i|-2$ in the Euclidean case of Deligne--Mostow theory as \cite[Corollary 13.2.2]{deligne1986monodromy}. Then the boundary divisor at this cusp is isogeneous to the product of the two abelian varieties.

Gallardo--Kerr--Schaffler \cite{gallardo2021geometric} and Hulek--Maeda \cite{hulek2025universe} studied identifications of Kirwan blowups and toroidal compactifications for arithmetic Deligne--Mostow ball quotients. Such identifications are still unknown for the three non-compact and non-arithmetic Deligne--Mostow ball quotients (58, 73, 74 in Table \ref{table: main table}) up to our knowledge.
\end{rmk}

\section{List of Commensurability Relations and Classification}
\label{section: tables}

\subsection{Commensurable Pairs from Geometry}
We name the degeneration types needed as follows:
	\begin{enumerate}
		\item[2.1.] $(3,2)+(0,1)$ normal crossing;
		\item[2.2.] $(3,2)+(0,1)$, the $(3,2)$ curve has one node;
		\item[2.3.] $(3,2)+(0,1)$ tangent at one tacnode;
		\item[2.4.] $(3,2)+(0,1)$, the $(3,2)$ curve has two nodes;
		\item[2.5.] $(3,2)+(0,1)$ tangent at one tacnode, the $(3,2)$ curve has one node;
		\item[2.6.] $(3,2)+(0,1)$ tangent at one tacnode, the $(3,2)$ curve has two nodes;
		\item[3.1.] $(3,1)+(0,1)+(0,1)$ normal crossing;
		\item[3.2.] $(3,1)+(0,1)+(0,1)$, the first two are tangent at one tacnode;
		\item[5.1.] $(2,2)+(1,0)+(0,1)$, the first two are tangent at one tacnode;
		\item[5.2.] $(2,2)+(1,0)+(0,1)$, the first two are tangent at one tacnode, the $(2,2)$ curve has one node;
        \item[5.3.] $(2,2)+(1,0)+(0,1)$, the first two are tangent at a tacnode, and the $(0,1)$ curve passes through the tacnode;
		\item[7.1.] $(2,1)+(1,1)+(0,1)$ normal crossing;
		\item[7.2.] $(2,1)+(1,1)+(0,1)$ with $(2,1)$ and $(0,1)$ tangent at one tacnode;
        \item[7.3.] $(2,1)+(1,1)+(0,1)$ with $2,1$ and $(1,1)$ tangent at one tacnode;
		\item[8.1.] $(2,1)+(1,0)+(0,1)+(0,1)$ normal crossing.
	\end{enumerate}
Here the indices are compatible with those in Proposition \ref{proposition: types}. 

The calculation of $\deg(\pi)$ in Table \ref{table: relation from geo} is straightforward. We illustrate the method for case 8.1. Consider normal crossing divisors $D=D_1+D_2+D_3+D_4\subset \PP^1_{x_1, x_2}\times \PP^1_{y_1, y_2}$ of type $(2,1)+(1,0)+(0,1)+(0,1)$; see Figure \ref{figure: (2,1)+(1,0)+(0,2)}. 

Suppose that the weights for $D_1, D_2, D_3, D_4$ are $1-a-b, 2a+2b-1, a, b$, respectively. For a generic set $A$ of $5$ points on $\PP^1_{y_1,y_2}$ with weight $(1-a, 1-b, 1-a-b, a+b-{1\over 2}, a+b-{1\over 2})$, while the last two points are unordered, the degree of $\pi_2$ is just the number of $D$ (up to $\SL(2, \CC)\times \SL(2, \CC)$ action) such that the discriminant locus of $p_2$ is $S$. The positions of $D_3, D_4$ are determined by $A$. By direct calculation, the divisors $D_1, D_2$ (up to the action of $\SL(2, \CC)$ on $\PP^1_{x_1, x_2}$) are also uniquely determined by $A$. Thus $\deg(\pi_2)=1$. Similarly, we have $\deg(\pi_1)=1$.
\begin{figure}[htp]
		\centering
\includegraphics[width=10cm]{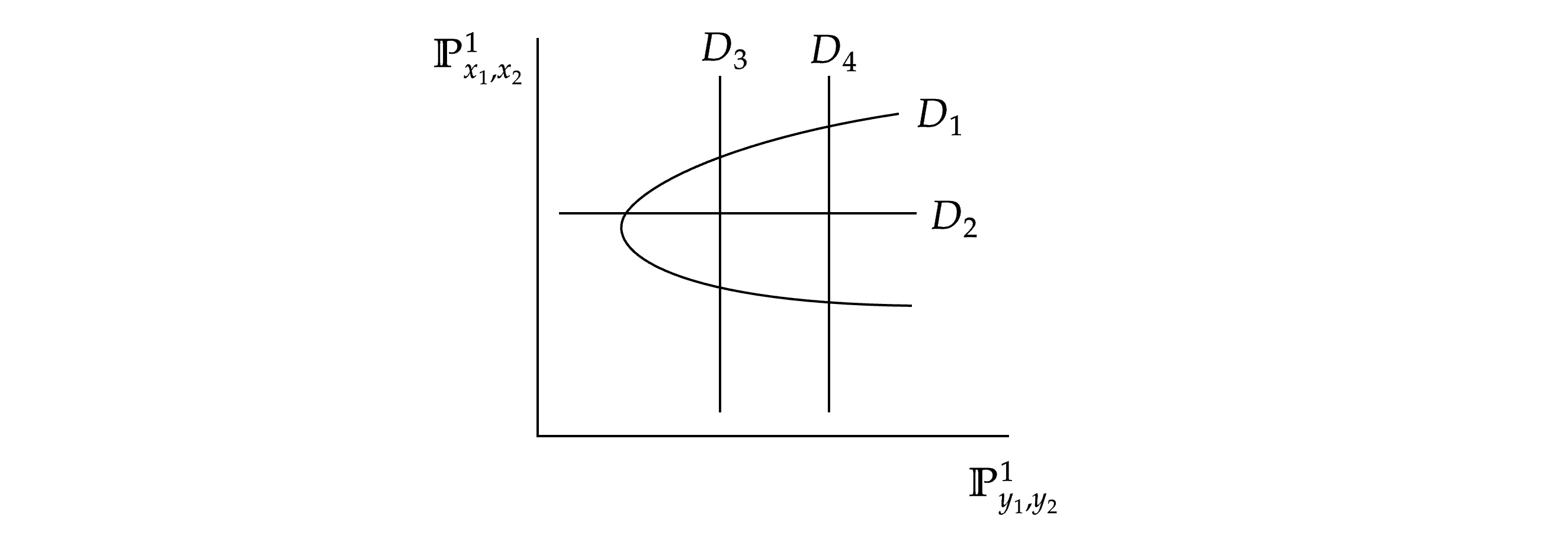}
		\caption{Configuration: (2,1)+(1,0)+(0,1)+(0,1)}
		\label{figure: (2,1)+(1,0)+(0,2)}
\end{figure}

We list the commensurability relations (see Theorem \ref{theorem: deligne mostow sauter} and Theorem \ref{theorem: dimension 3} for the range of $a,b$) obtained from geometry (Theorem \ref{theorem: main}) and the implications for Deligne--Mostow tuples in Table \ref{table: relation from geo}. 
	
	\begin{longtable}
		{|p{.08\textwidth} | p{.04\textwidth} | p{.22\textwidth} |p{.22\textwidth} |  p{.1\textwidth} | p{.1\textwidth} | }

		\toprule
		No. & $d$ & Numerators $d\mu$ & Numerators $d\nu$ & $\deg(\pi_1)$ & $\deg(\pi_2)$  \\
		\midrule
		2.1 & $6$ & $(1,1,1,1,1,1,2,2,2)$ & $(1,1,1,1,1,1,1,1,4)$ &&\\
        \hline
        2.2 & $6$ & $(1,1,1,1,2,2,2,2)$ & $(1,1,1,1,1,1,2,4)$ &&\\
        \hline
        2.3 & $6$ & $(1,1,1,1,1,1,2,4)$ & $(1,1,1,1,1,1,1,5)$ &&\\
        \hline
        2.4 & $6$ & $(1,1,2,2,2,2,2)$ & $(1,1,1,1,2,2,4)$ &&\\
        \hline
        2.5 & $6$ & $(1,1,1,1,2,2,4)$ & $(1,1,1,1,1,2,5)$ &&\\
        \hline
        2.6 & $6$ & $(1,1,2,2,2,4)$ & $(1,1,1,2,2,5)$ &&\\
        \hline
		3.1 & $d$  & $d(a$, $a$, $a$, ${2\over 3}-a$, ${2\over 3}-a$, ${2\over 3}-a)$ & $d({1\over 6}$, ${1\over 6}$, ${1\over 6}$, ${1\over 6}$, $1-a$, ${1\over 3}+a)$ & $1$ & $4$\\
		\hline
		3.1.1 & $6$  & $(2,2,2,2,2,2)$ & $(1,1,1,1,4,4)$ & & \\
		\hline
		3.1.2 & $6$  & $(1,1,1,3,3,3)$ & $(1,1,1,1,3,5)$ & & \\
		\hline
		3.1.3 & $12$  & $(3,3,3,5,5,5)$ & $(2,2,2,2,7,9)$ & &  \\
		\hline
		3.2 & $d$  & $d(a$, $a$, $a$, ${2\over 3}-a$, ${4\over 3}-2a)$ & $d({1\over 6}$, ${1\over 6}$, ${1\over 6}$, ${7\over 6}-a$, ${1\over 3}+a)$ & $1$ & $4$ \\
		\hline
		3.2.1 & $6$  & $(2,2,2,2,4)$ & $(1,1,1,4,5)$ & & \\
		\hline
		3.2.2 & $12$  & $(3,3,3,5,10)$ & $(2,2,2,7,11)$ & & \\
		\hline
		3.2.3 & $12$  & $(3,5,5,5,6)$ & $(2,2,2,9,9)$ & & \\
		\hline
		3.2.4 & $12$  & $(1,2,7,7,7)$ & $(2,2,2,7,11)$ & & \\
		\hline
		3.2.5 & $18$  & $(4,8,8,8,8)$ & $(3,3,3,13,14)$ & & \\
		\hline
		3.2.6 & $18$  & $(5,7,7,7,10)$ & $(3,3,3,13,14)$ & & \\
		\hline
		3.2.7 & $24$  & $(7,9,9,9,14)$ & $(4,4,4,17,19)$ & & \\
		\hline
		3.2.8 & $24$  & $(5,10,11,11,11)$ & $(4,4,4,17,19)$ & &  \\
		\hline
		3.2.9 & $30$  & $(9,9,9,11,22)$ & $(5,5,5,19,26)$ & & \\
		\hline
		3.2.10 & $30$  & $(4,8,16,16,16)$ & $(5,5,5,19,26)$ & & \\
		\hline
		3.2.11 & $30$  & $(7,13,13,13,14)$ & $(5,5,5,22,23)$ & & \\
		\hline
		3.2.12 & $30$  & $(8,12,12,12,16)$ & $(5,5,5,22,23)$ & &  \\
		\hline
		3.2.13 & $42$  & $(13,15,15,15,26)$ & $(7,7,7,29,34)$ & & \\
		\hline
		3.2.14 & $42$  & $(8,16,20,20,20)$ & $(7,7,7,29,34)$ & & \\
		\hline
		5.1 & $d$ & $d(a$, $a$, $a$, ${1\over 2}-a$, ${1\over 2}-a$, $1-a)$ & $d(a$, $a$, $a$, $a$, $1-2a$, $1-2a)$ & $4$ & $4$ \\
		\hline
		5.1.1 & $4$ &  $(1,1,1,1,1,3)$ & $(1,1,1,1,2,2)$ & & \\
		\hline
		5.1.2 & $6$ &  $(1,1,2,2,2,4)$ & $(2,2,2,2,2,2)$ && \\
		\hline
		5.1.3 & $6$  & $(1,1,1,2,2,5)$ & $(1,1,1,1,4,4)$ & & \\
		\hline
		5.2 & $d$ & $d(a$, $2a$, ${1\over 2}-a$, ${1\over 2}-a$, $1-a)$ & $d(a$, $a$, $2a$, $1-2a$, $1-2a)$ & $4$ & $4$ \\
		\hline
		5.2.1 & $4$  & $(1,1,1,2,3)$ & $(1,1,2,2,2)$ & & \\
		\hline
		5.2.2 & $6$  & $(1,1,2,4,4)$ & $(2,2,2,2,4)$ & & \\
		\hline
		5.2.3 & $6$  & $(1,2,2,2,5)$ & $(1,1,2,4,4)$ & & \\
		\hline
        5.3 & $d$ & $d(a$, $a$, $a$, ${1\over 2}-a$, ${3\over 2}-2a)$ & $d(a$, $a$, $a$, $a$, $2-4a)$ & $4$ & $4$ \\
		\hline
        5.3.1 & $6$ & $(1,2,2,2,5)$ &  $(2,2,2,2,4)$& &  \\
		\hline
		5.3.2 & $10$ & $(2,3,3,3,9)$ & $(3,3,3,3,8)$ & &  \\
		\hline
		5.3.3 & $18$ & $(1,8,8,8,11)$ & $(4,8,8,8,8)$ & &  \\
		\hline
		5.3.4 & $18$ & $(2,7,7,7,13)$ & $(7,7,7,7,8)$ & &  \\
		\hline
	
		7.1 & $d$ & $d(a$, $a$, $a$, $a$, $1-2a$, $1-2a)$ & $d(a$, $a$, $a$, ${1\over 2}-a$, ${1\over 2}-a$, $1-a)$ & $4$ & $4$ \\
		\hline
		7.1.1 & $4$ & $(1,1,1,1,2,2)$ & $(1,1,1,1,1,3)$ & &  \\
		\hline
		7.1.2 & $6$ & $(2,2,2,2,2,2)$ & $(1,1,2,2,2,4)$ & &  \\
		\hline
		7.2 & $d$ & $d(a$, $a$, $a$, $a$, $2-4a)$ & $d(a$, $a$, $a$, ${1\over 2}-a$, ${3\over 2}-2a)$ & $4$ & $4$\\
		\hline
        7.3 & $d$ & $d(a$, $a$, $2a$, $1-2a$, $1-2a)$ & $d(a$, $2a$, ${1\over 2}-a$, ${1\over 2}-a$, $1-a)$ & $4$ & $4$ \\
        \hline
		8.1 & $d$ & $d(a$, $a$, $b$, $b$, $2-2a-2b)$ & $d(1-a$, $1-b$, $1-a-b$, $a+b-{1\over 2}$, $a+b-{1\over 2})$ & $1$ & $1$ \\
		\hline
		8.1.1 & $6$ & $(2,2,2,2,4)$ & $(1,1,2,4,4)$ & &  \\
		\hline
		8.1.2 & $4$ & $(1,1,2,2,2)$ & $(1,1,1,2,3)$ & &  \\
		\hline
		8.1.3 & $10$ & $(4,4,4,4,4)$ & $(2,3,3,6,6)$ & &  \\
		\hline
		8.1.4 & $6$ & $(1,1,2,4,4)$ & $(1,2,2,2,5)$ & &  \\
		\hline
		8.1.5 & $6$ & $(1,1,3,3,4)$ & $(1,1,2,3,5)$ & &  \\
		\hline
		8.1.6 & $6$ & $(2,2,2,3,3)$ & $(1,2,2,3,4)$ & &  \\
		\hline
		8.1.7 & $8$ & $(2,2,2,5,5)$ & $(1,3,3,3,6)$ & &  \\
		\hline
		8.1.8 & $8$ & $(3,3,3,3,4)$ & $(2,2,2,5,5)$ & &  \\
		\hline
		8.1.9 & $18$ & $(4,8,8,8,8)$ & $(2,7,7,10,10)$ & &  \\
		\hline
		8.1.10 & $10$ & $(2,3,3,6,6)$ & $(1,4,4,4,7)$ & &  \\
		\hline
		8.1.11 & $10$ & $(3,3,3,3,8)$ & $(1,1,4,7,7)$ & &  \\
		\hline
		8.1.12 & $10$ & $(1,1,4,7,7)$ & $(2,3,3,3,9)$ & &  \\
		\hline
		8.1.13 & $12$ & $(2,2,6,7,7)$ & $(3,3,3,5,10)$ & &  \\
		\hline
		8.1.14 & $12$ & $(2,4,4,7,7)$ & $(1,5,5,5,8)$ & &  \\
		\hline
		8.1.15 & $12$ & $(3,3,5,5,8)$ & $(2,2,4,7,9)$ & &  \\
		\hline
		8.1.16 & $12$ & $(4,4,5,5,6)$ & $(3,3,3,7,8)$ & &  \\
		\hline
		8.1.17 & $12$ & $(4,5,5,5,5)$ & $(2,4,4,7,7)$ & &  \\
		\hline
		8.1.18 & $12$ & $(2,2,2,9,9)$ & $(1,3,5,5,10)$ & &  \\
		\hline
		8.1.19 & $14$ & $(5,5,5,5,8)$ & $(3,3,4,9,9)$ & &  \\
		\hline
		8.1.20 & $14$ & $(3,3,4,9,9)$ & $(2,5,5,5,11)$ & &  \\
		\hline
		8.1.21 & $18$ & $(2,7,7,10,10)$ & $(1,8,8,8,11)$ & &  \\
		\hline
		8.1.22 & $18$ & $(7,7,7,7,8)$ & $(4,5,5,11,11)$ & &  \\
		\hline
		8.1.23 & $18$ & $(4,5,5,11,11)$ & $(2,7,7,7,13)$ & &  \\
		\hline
		8.1.24 & $20$ & $(6,6,9,9,10)$ & $(5,5,5,11,14)$ & &  \\
\hline

\caption{commensurability relations from Geometry}
\label{table: relation from geo}
	\end{longtable}

\subsection{Table of Commensurability Classes for Deligne--Mostow Lattices}

\begin{longtable}
{|p{.04\textwidth} |p{.07\textwidth}  | p{.10\textwidth}  |p{.27\textwidth} | p{.07\textwidth} | p{.07\textwidth} | p{.07\textwidth} |}
\hline
No. & $(n,d)$ & Trace Field  & Numerators & C/NC & A/NA & Ideal Class \\\hline
1 & $(9,6)$ & $\QQ$  & (1,1,1,1,1,1,1,1,1,1,1,1) & NC & A & $1$ \\
\hline
2 & $(8,6)$ & $\QQ$ & (1,1,1,1,1,1,1,1,1,1,2) & NC & A & $\fp_3$  \\
\hline
3 & $(7,6)$ & $\QQ$ & (1,1,1,1,1,1,1,1,1,3) & NC & A & $2$ \\
\hline
4 & $(7,6)$ & $\QQ$ & (1,1,1,1,1,1,1,1,2,2) & NC & A & $1$ \\
\hline
5 & $(6,6)$ & $\QQ$ & (1,1,1,1,1,1,1,1,4) & NC & A & $\fp_3$ \\
6 & $(6,6)$ & $\QQ$ & (1,1,1,1,1,1,1,2,3) & NC & A & $2\fp_3$ \\
7 & $(6,6)$ & $\QQ$ & (1,1,1,1,1,1,2,2,2) & NC & A & $\fp_3$ \\
\hline
8 & $(5,6)$ & $\QQ$ & (1,1,1,1,1,1,1,5) & NC & A & $1$ \\
9 & $(5,6)$ & $\QQ$ & (1,1,1,1,1,1,2,4) & NC & A & $1$ \\
10 & $(5,6)$ & $\QQ$ & (1,1,1,1,1,1,3,3) & NC & A & $1$ \\
11 & $(5,6)$ & $\QQ$ & (1,1,1,1,2,2,2,2) & NC & A & $1$ \\
\hline
12 & $(5,6)$ & $\QQ$ & (1,1,1,1,1,2,2,3) & NC & A & $2$ \\
\hline
13 & $(4,6)$ & $\QQ$ & (1,1,1,1,1,2,5) & NC & A & $\fp_3$ \\
14 & $(4,6)$ & $\QQ$ & (1,1,1,1,1,3,4) & NC & A & $2\fp_3$ \\
15 & $(4,6)$ & $\QQ$ & (1,1,1,1,2,2,4) & NC & A & $\fp_3$ \\
16 & $(4,6)$ & $\QQ$ & (1,1,1,1,2,3,3) & NC & A & $\fp_3$ \\
17 & $(4,6)$ & $\QQ$ & (1,1,1,2,2,2,3) & NC & A & $2\fp_3$ \\
18 & $(4,6)$ & $\QQ$ & (1,1,2,2,2,2,2) & NC & A & $\fp_3$ \\
\hline
19 & $(3,3)$ & $\QQ$ & (1,1,1,1,1,1) & NC & A & $1$ \\
20 & $(3,6)$ & $\QQ$ & (1,1,1,1,4,4) & NC & A & $1$ \\
21 & $(3,6)$ & $\QQ$ & (1,1,1,2,2,5) & NC & A & $1$ \\
22 & $(3,6)$ & $\QQ$ & (1,1,2,2,2,4) & NC & A & $1$ \\
23 & $(3,6)$ & $\QQ$ & (1,1,2,2,3,3) & NC & A & $1$ \\
\hline
24 & $(3,6)$ & $\QQ$ & (1,1,1,1,3,5) & NC & A & $2$ \\
25 & $(3,6)$ & $\QQ$ & (1,1,1,2,3,4) & NC & A & $2$ \\
26 & $(3,6)$ & $\QQ$ & (1,1,1,3,3,3) & NC & A & $2$ \\
27 & $(3,6)$ & $\QQ$ & (1,2,2,2,2,3) & NC & A & $2$ \\
\hline
28 & $(2,3)$ & $\QQ$ & (1,1,1,1,2) & NC & A & $\fp_3$ \\
29 & $(2,6)$ & $\QQ$ & (1,1,1,4,5) & NC & A & $\fp_3$ \\
30 & $(2,6)$ & $\QQ$ & (1,1,2,3,5) & NC & A & $2\fp_3$ \\
31 & $(2,6)$ & $\QQ$ & (1,1,2,4,4) & NC & A & $\fp_3$ \\
32 & $(2,6)$ & $\QQ$ & (1,1,3,3,4) & NC & A & $\fp_3$ \\
33 & $(2,6)$ & $\QQ$ & (1,2,2,2,5) & NC & A & $\fp_3$ \\
34 & $(2,6)$ & $\QQ$ & (1,2,2,3,4) & NC & A & $2\fp_3$ \\
35 & $(2,6)$ & $\QQ$ & (1,2,3,3,3) & NC & A & $2\fp_3$ \\
36 & $(2,6)$ & $\QQ$ & (2,2,2,3,3) & NC & A & $\fp_3$ \\
\hline
37 & $(5,4)$ & $\QQ$ & (1,1,1,1,1,1,1,1) & NC & A & $1$ \\
\hline
38 & $(4,4)$ & $\QQ$ & (1,1,1,1,1,1,2) & NC & A & $1$ \\
\hline
39 & $(3,4)$ & $\QQ$ & (1,1,1,1,1,3) & NC & A & $1$ \\
40 & $(3,4)$ & $\QQ$ & (1,1,1,1,2,2) & NC & A & $1$ \\
\hline
41 & $(2,4)$ & $\QQ$ & (1,1,1,2,3) & NC & A & $1$ \\
42 & $(2,4)$ & $\QQ$ & (1,1,2,2,2) & NC & A & $1$ \\
\hline
43 & $(4,10)$ & $\QQ[\sqrt{5}]$ & (2,3,3,3,3,3,3) & C & A & $\fp_5$ \\
\hline
44 & $(3,10)$ & $\QQ[\sqrt{5}]$ & (2,3,3,3,3,6) & C & A & $1$  \\
\hline
45 & $(3,10)$ & $\QQ[\sqrt{5}]$ & (3,3,3,3,3,5) & C & A & $2$ \\
\hline
46 & $(2,5)$ & $\QQ[\sqrt{5}]$ & (2,2,2,2,2) & C & A & $\fp_5$ \\
47 & $(2,10)$ & $\QQ[\sqrt{5}]$ & (1,1,4,7,7) & C & A & $\fp_5$ \\
48 & $(2,10)$ & $\QQ[\sqrt{5}]$ & (1,4,4,4,7) & C & A & $\fp_5$\\
49 & $(2,10)$ & $\QQ[\sqrt{5}]$ & (2,3,3,3,9) & C & A & $\fp_5$ \\
50 & $(2,10)$ & $\QQ[\sqrt{5}]$ & (2,3,3,6,6) & C & A & $\fp_5$ \\
51 & $(2,10)$ & $\QQ[\sqrt{5}]$ & (3,3,3,3,8) & C & A & $\fp_5$ \\
52 & $(2,10)$ & $\QQ[\sqrt{5}]$ & (3,3,3,5,6) & C & A & $2\fp_5$ \\
\hline
53 & $(4,12)$ & $\QQ[\sqrt{3}]$ & (2,2,2,2,2,7,7) & C & A & $1$ \\
\hline
54 & $(3,12)$ & $\QQ[\sqrt{3}]$ & (1,3,5,5,5,5) & C & A & $\fp_2$ \\
55 & $(3,12)$ & $\QQ[\sqrt{3}]$ & (2,2,2,2,7,9) & C & A & $\fp_2$ \\
56 & $(3,12)$ & $\QQ[\sqrt{3}]$ & (3,3,3,5,5,5) & C & A & $\fp_2$ \\
\hline
57 & $(3,12)$ & $\QQ[\sqrt{3}]$ & (2,2,2,4,7,7) & C & A & $\fp_3$ \\
\hline
58 & $(3,12)$ & $\QQ[\sqrt{3}]$  & (3,3,3,3,5,7) & NC & NA & $1$ \\
\hline
59 & $(2,12)$ & $\QQ[\sqrt{3}]$ & (1,2,7,7,7) & C & A & $1$ \\
60 & $(2,12)$ & $\QQ[\sqrt{3}]$ & (1,3,5,5,10) & C & A & $\fp_2$ \\
61 & $(2,12)$ & $\QQ[\sqrt{3}]$ & (1,5,5,5,8) & C & A & $\fp_3$ \\
62 & $(2,12)$ & $\QQ[\sqrt{3}]$ & (2,2,2,7,11) & C & A & $1$ \\
63 & $(2,12)$ & $\QQ[\sqrt{3}]$ & (2,2,2,9,9) & C & A & $1$ \\
64 & $(2,12)$ & $\QQ[\sqrt{3}]$ & (2,2,4,7,9) & C & A & $\fp_2\fp_3$ \\
65 & $(2,12)$ & $\QQ[\sqrt{3}]$ & (2,2,6,7,7) & C & A & $1$ \\
66 & $(2,12)$ & $\QQ[\sqrt{3}]$ & (2,4,4,7,7) & C & A & $1$ \\
67 & $(2,12)$ & $\QQ[\sqrt{3}]$ & (3,3,3,5,10) & C & A & $\fp_2$ \\
68 & $(2,12)$ & $\QQ[\sqrt{3}]$ & (3,3,5,5,8) & C & A & $\fp_3$ \\
69 & $(2,12)$ & $\QQ[\sqrt{3}]$ & (3,5,5,5,6) & C & A & $\fp_2$ \\
70 & $(2,12)$ & $\QQ[\sqrt{3}]$ & (4,5,5,5,5) & C & A & $\fp_3$ \\
\hline
71 & $(2,12)$ & $\QQ[\sqrt{3}]$ & (3,3,3,7,8) & C & NA & $\fp_2\fp_3$ \\
72 & $(2,12)$ & $\QQ[\sqrt{3}]$ & (4,4,5,5,6) & C & NA & $1$ \\
\hline
73 & $(2,12)$ & $\QQ[\sqrt{3}]$ & (3,3,5,6,7) & NC & NA & $1$ \\
\hline
74 & $(2,12)$ & $\QQ[\sqrt{3}]$ & (4,4,4,5,7) & NC & NA & $\fp_3$ \\
\hline
75 & $(3,8)$ & $\QQ[\sqrt{2}]$ & (1,3,3,3,3,3) & C & A & $\fp_2$ \\
\hline
76 & $(2,8)$ & $\QQ[\sqrt{2}]$ & (1,3,3,3,6) & C & A & $1$ \\
77 & $(2,8)$ & $\QQ[\sqrt{2}]$ & (2,2,2,5,5) & C & A & $1$ \\
78 & $(2,8)$ & $\QQ[\sqrt{2}]$ & (3,3,3,3,4) & C & A & $1$ \\
\hline
79 & $(2,14)$ & $\QQ[\cos {\pi\over 7}]$ & (2,5,5,5,11) & C & A & $\fp_7$ \\
80 & $(2,14)$ & $\QQ[\cos {\pi\over 7}]$ & (3,3,4,9,9) & C & A & $\fp_7$ \\
81 & $(2,14)$ & $\QQ[\cos {\pi\over 7}]$ & (5,5,5,5,8) & C & A & $\fp_7$ \\
\hline
82 & $(2,9)$ & $\QQ[\cos {\pi\over 9}]$ & (2,4,4,4,4) & C & A & $\fp_3$ \\
83 & $(2,18)$ & $\QQ[\cos {\pi\over 9}]$ & (1,8,8,8,11) & C & A & $\fp_3$ \\
84 & $(2,18)$ & $\QQ[\cos {\pi\over 9}]$ & (2,7,7,10,10) & C & A & $\fp_3$ \\
85 & $(2,18)$ & $\QQ[\cos {\pi\over 9}]$ & (3,3,3,13,14) & C & A & $\fp_3$ \\
86 & $(2,18)$ & $\QQ[\cos {\pi\over 9}]$ & (5,7,7,7,10) & C & A & $\fp_3$ \\
\hline
87 & $(2,18)$ & $\QQ[\cos {\pi\over 9}]$ & (2,7,7,7,13) & C & NA & $\fp_3$ \\
88 & $(2,18)$ & $\QQ[\cos {\pi\over 9}]$ & (4,5,5,11,11) & C & NA & $\fp_3$\\
89 & $(2,18)$ & $\QQ[\cos {\pi\over 9}]$ & (7,7,7,7,8) & C & NA & $\fp_3$\\
\hline
90 & $(2,20)$ & $\QQ[\cos {\pi\over 10}]$ & (5,5,5,11,14) & C & NA & $\fp_2$\\
91 & $(2,20)$ & $\QQ[\cos {\pi\over 10}]$ & (6,6,9,9,10) & C & NA & $1$ \\
\hline
92 & $(2,20)$ & $\QQ[\cos {\pi\over 10}]$ & (6,6,6,9,13) & C & NA & $1$ \\
\hline
93 & $(2,24)$ & $\QQ[\cos {\pi\over 12}]$ & (4,4,4,17,19) & C & NA & $1$\\
94 & $(2,24)$ & $\QQ[\cos {\pi\over 12}]$ & (5,10,11,11,11) & C & NA & $1$\\
95 & $(2,24)$ & $\QQ[\cos {\pi\over 12}]$ & (7,9,9,9,14) & C & NA & $\fp_2$\\
\hline
96 & $(2,15)$ & $\QQ[\cos {\pi\over 15}]$ & (2,4,8,8,8) & C & A & $1$\\
97 & $(2,30)$ & $\QQ[\cos {\pi\over 15}]$ & (5,5,5,19,26) & C & A & $1$\\
98 & $(2,30)$ & $\QQ[\cos {\pi\over 15}]$ & (9,9,9,11,22) & C & A & $1$\\
\hline
99 & $(2,15)$ & $\QQ[\cos {\pi\over 15}]$ & (4,6,6,6,8) & C & NA & $\fp_5$ \\
100 & $(2,30)$ & $\QQ[\cos {\pi\over 15}]$ & (5,5,5,22,23) & C & NA & $1$\\
101 & $(2,30)$ & $\QQ[\cos {\pi\over 15}]$ & (7,13,13,13,14) & C & NA & $1$\\
\hline
102 & $(2,21)$ & $\QQ[\cos {\pi\over 21}]$ & (4,8,10,10,10) & C & NA & $1$\\
103 & $(2,42)$ & $\QQ[\cos {\pi\over 21}]$ & (7,7,7,29,34) & C & NA & $1$\\
104 & $(2,42)$ & $\QQ[\cos {\pi\over 21}]$ & (13,15,15,15,26) & C & NA & $1$\\
\hline
\caption{Commensurability classes of Deligne--Mostow Lattices}
\label{table: main table}
\end{longtable}

\begin{rmk}
The $10$ cases which are discrete but not satisfying the half-integer condition are (see \cite[\S 5.1]{mostow1988discontinuous}): 
\begin{center}
$({1\over 12},{3\over 12}, {5\over 12},{5\over 12},{5\over 12},{5\over 12})$, $({1\over 12},{3\over 12},{5\over 12},{5\over 12},{10\over 12})$,
 $({1\over 10},{1\over 10},{4\over 10},{7\over 10},{7\over 10})$,
 $({1\over 12},{2\over 12},{7\over 12},{7\over 12},{7\over 12})$, \\
 \medskip
 $({3\over 14},{3\over 14},{4\over 14},{9\over 14},{9\over 14})$,
 $({2\over 15}, {4\over 15}, {8\over 15}, {8\over 15}, {8\over 15})$,
 $({4\over 18},{5\over 18},{5\over 18},{11\over 18},{11\over 18})$,
 $({4\over 21}, {8\over 21}, {10\over 21}, {10\over 21}, {10\over 21})$,\\
 \medskip
 $({5\over 24}, {10\over 24}, {11\over 24}, {11\over 24}, {11\over 24})$,
 $({7\over 30}, {13\over 30}, {13\over 30}, {13\over 30}, {14\over 30})$.
\end{center}
\end{rmk}

\bibliography{reference}
\bibstyle{alpha}
	
\Addresses
\end{document}